\documentclass[11pt]{article}
\usepackage[utf8]{inputenc}
\usepackage[english]{babel}
\usepackage{graphicx}
\usepackage{amsmath}
\usepackage{amssymb}
\usepackage{amsfonts}
\usepackage{amsthm}
\usepackage[left=2.1cm,top=1.5cm,right=2.1cm, bottom=2.2cm,letterpaper]{geometry}
\usepackage{mathrsfs}
\usepackage{caption}
\usepackage{subcaption}
\usepackage{latexsym}
\usepackage{xfrac}
\usepackage{tcolorbox}
\usepackage{enumitem}
\usepackage{float}
\usepackage{hyperref}
\usepackage[all]{hypcap}
\usepackage{autonum}
 \usepackage{url}
\usepackage[toc]{appendix}
\newtheorem{theorem}{Theorem}[section]

\newtheorem{lemma}{Lemma}[section]
\newtheorem{definition}{Definition}[section]
\newtheorem{proposition}{Proposition}[section]

\newtheorem{remark}{Remark}[section]
\numberwithin{equation}{section}
\numberwithin{figure}{section}

\newcommand{\R}{\mathbb{R}}
\newcommand{\eps}{\varepsilon}

\def\LL{\mathcal{L}}
\def\A{\mathcal{A}}

\def\HH{\mathcal{H}}
\def\VV{\mathcal{V}}
\newcommand{\bs}{\boldsymbol}

\makeatletter
\renewcommand*\env@matrix[1][*\c@MaxMatrixCols c]{%
	\hskip -\arraycolsep
	\let\@ifnextchar\new@ifnextchar
	\array{#1}}
\makeatother

\hypersetup{
	colorlinks=true,
	linkcolor=red,
	filecolor=magenta,
	urlcolor=cyan,
}

\def\neweq#1{\begin{equation}\label{#1}}
\def\endeq{\end{equation}}

\begin{document}
\title{\vspace{-0.7cm}{On the evolutionary Navier-Stokes equations in distorted pipes under dynamic and energy-stable outflow boundary conditions}}

\author{Alessio Falocchi -- Ana Leonor Silvestre -- Gianmarco Sperone}
\date{}
\maketitle
\vspace*{-6mm}
\begin{abstract}
	\noindent
	We consider the evolution of a viscous incompressible fluid in three-dimensional distorted pipes, of finite length, modeled through the Navier-Stokes equations with mixed boundary conditions. Specifically, the inflow is given by an arbitrary datum, the outflow is subject to a dynamic condition including a directional do-nothing boundary condition; standard no-slip assumptions are imposed on the remaining walls of the domain. We introduce a new functional framework that accommodates the dynamic boundary condition and the incompressibility constraint. Within this framework, we establish the existence of weak solutions. Moreover, under a smallness assumption on the data, we prove the existence and uniqueness of global strong solutions. 
	The newly introduced functional setting enables us to deal, at first, with the associated Stokes problem, and then with the full Navier-Stokes system, via the Galerkin method implemented  with a basis of eigenfunctions of a suitably modified Stokes operator.
\end{abstract}

	
\section{Introduction and presentation of the problem}
The analysis of the laminar flow of a viscous incompressible fluid through pipe junctions $\Omega \subset \mathbb{R}^{3}$, which are bounded by rigid and impermeable walls, constitutes a fundamental and widespread topic in Fluid Mechanics \cite{landau}. This area of study is crucial for numerous practical applications, including aerodynamics \cite{von2004aerodynamics}, haemodynamics \cite{galdi2008hemodynamical}, and petroleum engineering \cite{bradley1987petroleum}. From a mathematical perspective \cite{ladyzhenskaya1969mathematical}, such motion is investigated by means of the Navier-Stokes equations in $\Omega$.
Whether $\Omega$ is assumed to be bounded or not depends entirely on the specific physical configuration to be modeled. Real-world engineering applications rarely permit a precise characterization of fluid behaviour at large distances, meaning that their numerical implementation inherently requires a bounded computational domain. Truncating the fluid domain introduces \textit{artificial boundaries} into the pipe system, where appropriate boundary conditions must be prescribed, see \cite{blazy2007artificial, nazarov2008artificial} and the references therein. Selecting the conditions to impose on these artificial outlets is a delicate question from both a mathematical and physical standpoint \cite{heywood1996artificial}. While the inflow is typically prescribed and viscous effects dictate a zero-velocity (no-slip) condition on the solid walls, the fluid behaviour on the outlet remains open. These artificial boundary conditions must guarantee both the well-posedness of the mathematical model and the numerical stability of the simulations; we refer the reader to \cite{braack2014directional,bruneau2000boundary,bruneau1996new,fursikov2009optimal,kravcmar2018modeling,neustupa2022maximum,neustupa2023existence,nogueira2025regularized,nogueira2025steady,sperone2021steady} for further discussions.

In order to introduce such conditions we define the domains we are interested in. 
\begin{definition} \label{addomain3}
	An open bounded set $\Omega \subset \mathbb{R}^{3}$ will be called \textbf{admissible} if $\partial \Omega$ is piecewise of class $\mathcal{C}^{2}$, $\Omega$ is simply connected and $\Omega$ is the union of three disjoint subsets as follows:
	\begin{itemize}
		\item [(1)] In some coordinate system, $\Omega_{1} := \Theta_{1} \times (0,\ell_{1})$, for some $\ell_{1}>0$ and some open, bounded and planar domain $\Theta_{1} \subset \mathbb{R}^{2}$ having boundary of class $\mathcal{C}^{2}$. Therefore, $\Omega_{1}$ is a cylinder of length $\ell_{1}$ and fixed cross-section $\Theta_{1}$;
		\item [(2)] In some coordinate system, $\Omega_{2} := \Theta_{2} \times (0,\ell_{2})$, for some $\ell_{2}>0$ and some open, bounded and planar domain $\Theta_{2} \subset \mathbb{R}^{2}$ having boundary of class $\mathcal{C}^{2}$. Therefore, $\Omega_{2}$ is a cylinder of length $\ell_{2}$ and fixed cross-section $\Theta_{2}$;
		\item [(3)] $\Omega_{0} := \Omega \setminus (\Omega_{1} \cup \Omega_{2})$  (note that $\Omega_{0}$ is not open).
	\end{itemize}
	The boundary of $\Omega$ is decomposed as $ \partial \Omega = \Gamma_{I} \cup \Gamma_{W} \cup \Gamma_{O}$, where
	\begin{equation}\label{boundaryomega3d}
		\begin{aligned}
			& \Gamma_{I} := \Theta_{1} \times \{0\} \quad \text{(in the coordinate system defining $\Omega_{1}$)}  \,, \\[5pt]
			& \Gamma_{O} := \Theta_{2} \times \{\ell_{2}\} \quad \text{(in the coordinate system defining $\Omega_{2}$)} \, ,
		\end{aligned}
	\end{equation}	
	and $\Gamma_{W} \subset \mathbb{R}^{3}$ represents the $\mathcal{C}^{2}$-surface connecting $\Gamma_{I}$ with $\Gamma_{O}$.
\end{definition}
An admissible domain $\Omega \subset \mathbb{R}^{3}$ is the truncation (orthogonal to the symmetry axis of the inlet and outlet) of the domain considered in the celebrated \textbf{Leray problem}, see \cite[Definition 1.1]{amick1977steady}, \cite[Chapter III]{galdi2008hemodynamical} and also \cite[Chapter 3]{pileckas2007navier}. Figure \ref{dom33d} depicts an admissible three-dimensional pipe. Notice that, in Definition \ref{addomain3}, the cylinders $\Omega_{1}$ and $\Omega_{2}$ may possibly have different cross-sections and inclinations. 
\begin{figure}[H]
	\begin{center}
		\includegraphics[scale=0.6]{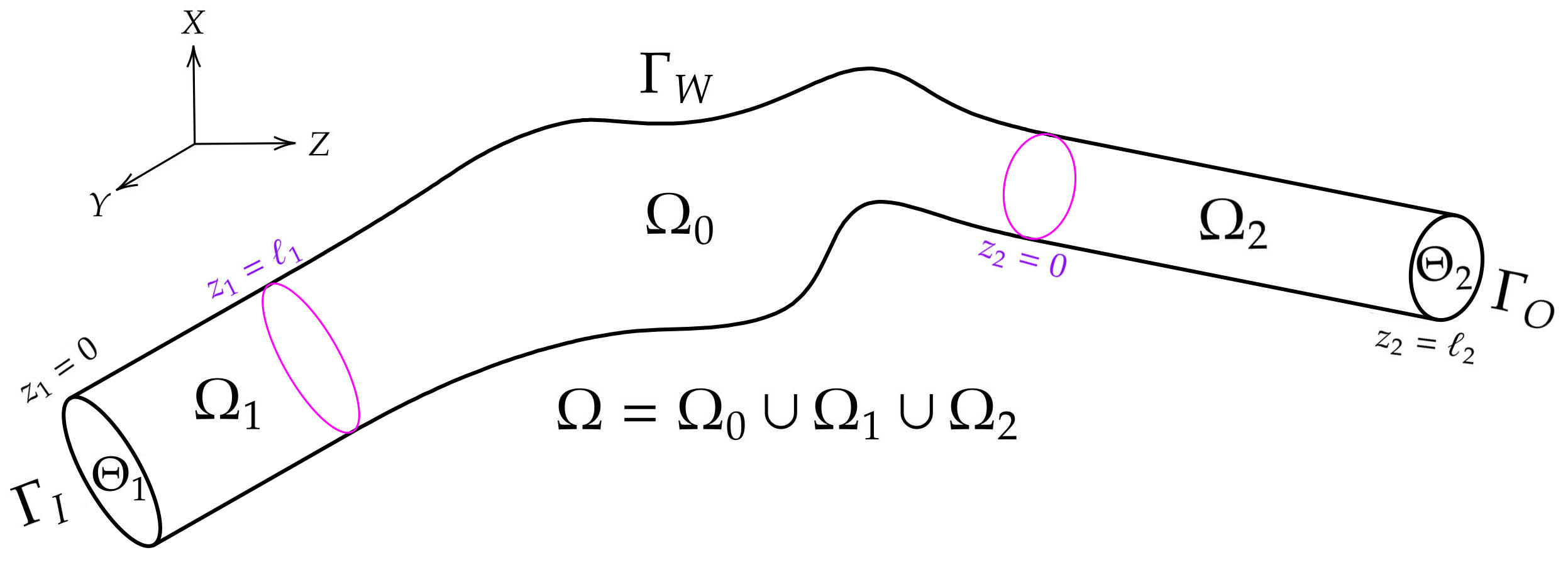}
	\end{center}
	\vspace*{-6mm}
	\caption{An admissible domain $\Omega \subset \mathbb{R}^{3}$.}\label{dom33d}
\end{figure}
\noindent
We denote the outward unit normal to $\partial \Omega$ by $\bs{n} \in \mathbb{R}^{3}$. We will refer to $\Gamma_{I}$ and $\Gamma_{O}$ in \eqref{boundaryomega3d} as the \textit{inlet} and \textit{outlet} of $\Omega$, respectively, while $\Gamma_{W}$ is composed by the \textit{physical walls} of $\Omega$.

Let $T>0$ and $Q_T:=\Omega\times (0,T)$ where $\Omega \subset \mathbb{R}^{3}$ is an admissible domain in the sense of Definition \ref{addomain3}. Given a (constant) kinematic viscosity coefficient $\nu > 0$, an inlet velocity field $\bs{g}_{*} : \Gamma_{I}\times [0,T) \longrightarrow \mathbb{R}^{3}$ and a function $\sigma_{*} : \Gamma_{O}\times (0,T) \longrightarrow \mathbb{R}$, by \textit{reference flow} we mean a sufficiently smooth pair $(\bs{W}_{\! \! *}, \Pi_{*})$ (whose existence will be proved in Lemma \ref{lemma0}) satisfying the following system in $Q_T$:
\begin{equation}\label{stokesintro}
	\left\{
	\begin{aligned}
		& \nabla\cdot \bs{W}_{\! \! *} = 0 \ \ \hspace{16mm}\mbox{ in } \ \ Q_T \, , \\[5pt]
		& \bs{W}_{\! \! *}=\bs{g}_{*}\ \ \hspace{20mm} \mbox{ on } \ \ \Gamma_{I}\times (0,T)  \, , \\[5pt] &\bs{W}_{\! \! *}=\bs{0} \ \ \hspace{22mm}\mbox{ on } \ \ \Gamma_{W}\times (0,T) \, , \\[5pt]
		& \nu \dfrac{\partial \bs{W}_{\! \! *}}{\partial \bs{n}} - \Pi_{*} \, \bs{n} = \sigma_{*} \, \bs{n} \ \ \mbox{ on } \ \ \Gamma_{O}\times (0,T)  \, .
	\end{aligned}
	\right.
\end{equation}

In this paper we deal with the time-dependent Navier-Stokes equations with mixed boundary conditions on the different parts of $\partial \Omega$, that is, the following system of partial differential equations:
\begin{equation}\label{ns}
	\left\{
	\begin{aligned}
		&\dfrac{\partial\bs{v}}{\partial t}-\nu\Delta \bs{v}+(\bs{v}\cdot\nabla)\bs{v}+\nabla p=\bs{f} \, , \quad  \nabla\cdot \bs{v}=0 \ \ \mbox{ in }  \ Q_T \, , \\[5pt]
		& \bs{v}=\bs{g}_{*} \ \ \mbox{ on } \ \ \Gamma_{I}\times (0,T) \, , \\[5pt]
		& \bs{v}=\bs{0} \ \ \hspace{2mm}\mbox{ on } \ \ \Gamma_{W}\times (0,T) \, , \\[5pt]	
		&\dfrac{\partial\bs{v}}{\partial t}+ \nu \dfrac{\partial \bs{v}}{\partial \bs{n}} - p \, \bs{n} + \dfrac{1}{2} [\bs{v} \cdot \bs{n}]^{-}(\bs{v} - \bs{W}_{\! \! *}) = \sigma_{*} \, \bs{n}\ \ \mbox{ on } \ \ \Gamma_{O}\times (0,T) \, , \\[5pt]
		&\bs{v}(\cdot ,0)=\bs{v}_{0,\Omega} \ \ \mbox{ in } \  \Omega \,, \\[5pt]
		&\bs{v}(\cdot,0)=\bs{v}_{0,\Gamma_O} \ \ \mbox{ on } \ \Gamma_O \, ,
	\end{aligned}
	\right.
\end{equation}
where $\bs{v} : \Omega \longrightarrow \mathbb{R}^3$ is the velocity vector field, $p : \Omega \longrightarrow \mathbb{R}$ is the scalar pressure and $\bs{f} : \Omega \longrightarrow \mathbb{R}^3$ represents a given external force acting on the fluid. 
Condition \eqref{ns}$_{2}$ prescribes the inlet velocity on $\Gamma_{I}$, condition \eqref{ns}$_{3}$ describes the no-slip boundary condition on $\Gamma_{W}$, while identity \eqref{ns}$_{4}$ dictates that the fluid flow is subject to a \textit{directional do-nothing-type} (DDN) boundary condition on $\Gamma_{O}$, where 
$$
[z]^{-} := \dfrac{|z | - z}{2} \qquad \forall z \in \mathbb{R} \, ,
$$
denotes the \textit{negative part} of any real number. The initial datum is $\bs{v}_0 := (\bs{v}_{0,\Omega},\bs{v}_{0,\Gamma_O})$, with given functions $\bs{v}_{0,\Omega} : \Omega \longrightarrow \mathbb{R}^3$ and $\bs{v}_{0,\Gamma_O} : \Gamma_O \longrightarrow \mathbb{R}^3$. In order to get compatibility between $\bs{v}_0$ and $\bs{g}_{*}$ at $t=0$, we impose that $\gamma_{\Gamma_I}(\bs{v}_{0,\Omega}) = \bs{g}_{*}(\cdot,0)$, where $\gamma_{\Gamma_I}$ is the restriction to $\Gamma_I$ of the standard trace operator, see also Section \ref{functional}. 

The presence of the time-derivative on the boundary condition \eqref{ns}$_{4}$ is motivated by the classical works of Ka\v{c}ur \cite{kacur} and Solonnikov et al. \cite{bizhanova1994solvability,bizhanova1994some,solonnikov}, among others; in these papers, the Authors
show the existence of classical solutions to several differential systems, focusing on the general theory of parabolic problems
with time-derivatives on the boundary condition. Therefore, we also include $\partial_{t} \bs{v}$ on the outlet $\Gamma_{O}$, highlighting its necessity in order to establish the existence of strong solutions to the evolutionary Navier-Stokes equations under Neumann-type boundary conditions.

Problem \eqref{ns} represents the unsteady version of the stationary system treated in \cite{fss}; for a deeper insight on this problem in dimensionless form, as well as the physical meaning of the DDN boundary conditions, we refer to \cite{fss}. Here we recall the main motivations leading us to consider this Neumann-type condition on the outlet. 
First of all, it originates as a modification of the standard \textit{traction} boundary condition, which reads
\begin{equation}\label{ctbc}
	\nu \dfrac{\partial \bs{u}}{\partial \bs{n}} - p \, \bs{n} = \sigma_* \bs{n} \ \ \mbox{ on } \ \ \Gamma_{O} \, 
\end{equation}
and reduces to the \textit{do-nothing} (homogeneous) boundary condition when $\sigma_* \equiv 0$.
Mathematically, condition \eqref{ctbc} arises naturally from the weak formulation of the Navier-Stokes equations, where an integration by parts in $\Omega$, combined with \eqref{ctbc}, allows to obtain a precise information from the boundary integral on $\Gamma_O$. Since the formulation by Gresho \cite{gresho1991some} in 1991, the boundary condition \eqref{ctbc} has been widely employed in Computational Fluid Dynamics \cite{braack2014directional, heywood1996artificial, john2002higher, lanzendorfer2020multiple, rannacher2012short}. Nevertheless, it is well-known that \eqref{ctbc} is not an energy-stable condition due to the possibility of a backwards flow coming into $\Omega$ from $\Gamma_{O}$, see the works by Kra\v{c}mar \& Neustupa \cite{kravcmar2001weak, kravcmar2018modeling}. This transfer of kinetic energy is strictly related to the nonlinear term $(\bs{u}\cdot\nabla)\bs{u}$, as the following integration by parts shows: 
\begin{equation} \label{introeq1}
	\int_\Omega (\bs{u}\cdot\nabla)\bs{u}\cdot \bs{u}=\dfrac{1}{2}\int_{\Gamma_I}|\bs{g}_*|^2(\bs{g}_*\cdot \bs{n})+\dfrac{1}{2}\int_{\Gamma_O}|\bs{u}|^2(\bs{u}\cdot \bs{n}) \, .
\end{equation}
Here the integral over $\Gamma_{I}$ is controlled by the datum $\bs{g}_{*}$, while the second term can be rewritten as
$$
\dfrac{1}{2}\int_{\Gamma_O}|\bs{u}|^2(\bs{u}\cdot \bs{n})= \dfrac{1}{2}\int_{\Gamma_O}|\bs{u}|^2\left([\bs{u}\cdot \bs{n}]^+-[\bs{u}\cdot\bs{n}]^-\right) \, ,
$$ 
revealing the possibility of backflow on $\Gamma_O$ as soon as $[\bs{u}\cdot\bs{n}]^-\neq 0$. This is the reason why existence results on Navier-Stokes equations under classical do-nothing boundary conditions are known under a smallness assumption on the data (inlet velocity and external force), see \cite[Chapter III]{galdi2008hemodynamical}.
On the other hand, the DDN-type boundary condition has the advantage of being energy-stable, ensuring the uniqueness of the rest state \cite{braack2014directional} and preventing backflow at outflow boundary portions \cite{BCBBG2018,DONG2015300}.
For all these reasons we include it in \eqref{ns}$_4$.

The method we employ to establish the existence of weak and strong solutions to \eqref{ns}, respectively in Theorems \ref{existence}-\ref{strong} (the main results of this paper), is based on the Galerkin approximation, see e.g. \cite{Gal2000a,RRS2016,Tem1977b}. The existence of weak solutions under minimal assumptions on the data follows by standard estimates for the three-dimensional Navier-Stokes equations; in particular, we exploit the DDN boundary condition while handling the nonlinear term. The proof concerning the existence of a (unique) strong solution, under a smallness assumption on the data, in Theorem \ref{strong} is more delicate, due to the presence of a nonhomogeneous datum on the outlet. For this aim, we introduce a specific functional setting, see Subsection \ref{functional}, and we prove some preliminary results on these new functional spaces.
Since the associated Stokes problem has a nonhomogenous datum on $\Gamma_O$ belonging to $L^2(\Gamma_O)$, see \eqref{stokesf}$_3$ in Subsection \ref{substokes}, we expect solutions having $H^{3/2}(\Omega)$-regularity \cite{brown2010mixed}, instead of the usual $H^2(\Omega)$-regularity.
This introduces several non-trivial technical difficulties; in fact, a corresponding 2D regularity result is, to the best of our knowledge, unavailable, preventing us from establishing an analogue of our main findings in the two-dimensional case.
Preparatory results concerning the three-dimensional Stokes operator, which are essential for the Galerkin method in the proof of Theorem \ref{strong}, are presented in Subsection \ref{substokes}. We complete the preliminary part by constructing the time-dependent reference flow $(\bs{W}_{\! \! *}, \Pi_{*})$ in Subsection \ref{reference}.
The main results are stated in Section \ref{main}, while all the proofs are given in Section \ref{proofs}.

We conclude this Introduction by pointing out that the results presented in this manuscript can be easily adapted and extended to the case when the admissible domain $\Omega \subset \mathbb{R}^3$ has several cylindrical inlets and outlets. An example of such junction of pipes is depicted in Figure \ref{polpo2d}.

\begin{figure}[H]
	\begin{center}
		\includegraphics[scale=0.43]{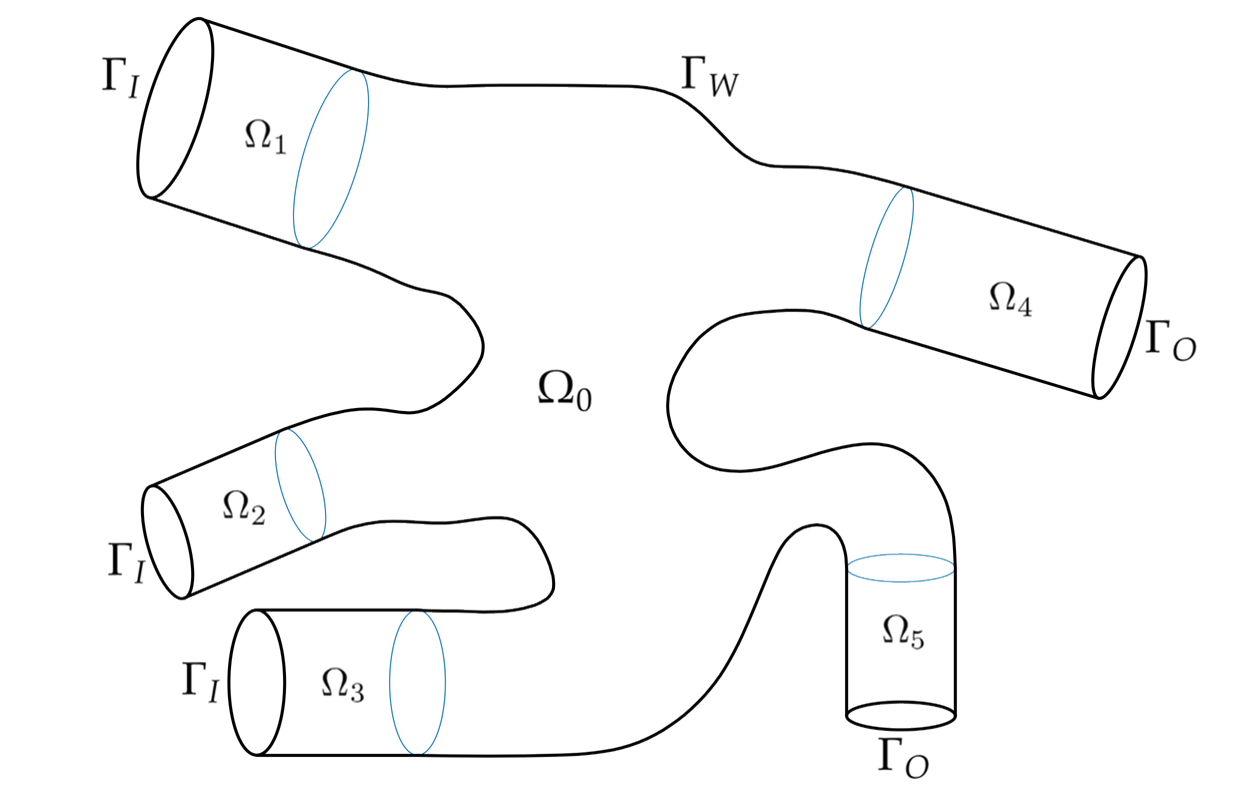}
	\end{center}
	\vspace*{-6mm}
	\caption{Representation of an admissible domain $\Omega \subset \mathbb{R}^{3}$ with several inlets and outlets.}\label{polpo2d}
\end{figure}

\section{Preliminary results}
\subsection{Functional setting}\label{functional}
Given a bounded Lipschitz domain $D \subset \mathbb{R}^n$, we denote by: 
\begin{itemize}[leftmargin=5mm]
	\item[$\bullet$] $\mathcal{C}(D)$ the space of continuous scalar functions, $\mathcal{C}^\infty(D)$ the space of infinitely differentiable scalar functions and  $\mathcal{C}^\infty_0(D)$ the space of	$\mathcal{C}^\infty(D)$  functions having compact support in $D$. We shall also use, for vector fields, the symbol $\mathcal{C}^\infty_{0,\sigma}(D)$, meaning that each component belongs to   $\mathcal{C}_0^\infty(D)$ and the vector field is divergence free in $D$.
	\item[$\bullet$] $L^p(D)$, for $1\leq p \leq\infty$, the usual Lebesgue spaces for scalar functions, equipped
	with norm $\|\cdot\|_{L^p(D)}$. Additionally, we introduce the space of scalar functions with zero mean value in $D$, defined as
	\begin{equation} \label{l02}
		L^2_0(D) := \left\{ h\in L^2(D) \ \Big | \ \int_{D} h = 0 \right\} \, .
	\end{equation}
	
	\item[$\bullet$] $W^{s,p}(D)$ the Sobolev spaces, equipped
	with norm $\|\cdot\|_{W^{s,p}(D)}$ for $s\geq 0$ and $p \in [1,\infty]$; when $p=2$ we shall simply write $H^s(D)$ with $L^2(D)=H^0(D)$.
	
	\item[$\bullet$] $H^{-s}(D)$ the dual space of $H^{s}_{0}(D)$, for $s \geq 1$, with the corresponding duality pair $\langle\cdot,\cdot\rangle_{H^{-s}(D),H^s_0(D)}$  and dual norm $\|\cdot\|_{H^{-s}(D)}$.

	\item[$\bullet$] $W^{k - \frac{1}{p},p}(\partial D)$ the usual trace spaces, with any integer $k\geq 1$ and $p \in [1,\infty)$; when $p=2$ we shall simply write $H^{k- \frac{1}{2}}(\partial D)$. We also consider $H^{-1/2}(\partial D)$ as the dual space of $H^{1/2}(\partial D)$ with the corresponding duality pair $\langle\cdot,\cdot\rangle_{H^{-1/2}(\partial D),H^{1/2}(\partial D)}$  and dual norm $\|\cdot\|_{H^{-1/2}(\partial D)}$.
\end{itemize}

Given an integer $k \geq 1$ and a relatively open portion $\Gamma \subset \partial D$ with $|\Gamma| >0$, the \textit{Lions-Magenes spaces} \cite[Chapter VII]{dautray1999mathematical} are defined as
$$
H^{k- \frac{1}{2}}_{00}(\Gamma) := \left\{ u \in L^2(\Gamma) \  \bigg| \  \widetilde{u}  \in  H^{k- \frac{1}{2}}(\partial D) \right\},
$$
where $\widetilde{u}$ denotes the zero extension of $u \in L^2(\Gamma)$ to the whole boundary $\partial D$. The dual space of $H^{\frac{1}{2}}_{00}(\Gamma)$ is denoted by $H^{-\frac{1}{2}}(\Gamma) $ and endowed with the norm induced from $H^{- \frac{1}{2}}(\partial D)$.

When dealing with vector functions of $n$-components, we use the space notation $V^n$, meaning that each component of the vector is in the space $V$, e.g. $\bs{v}\in H^k(D)^n$ with $\|\bs{v}\|_{H^k(D)}$ for any $k \in \mathbb{N}$; observe that in the norm notation we do not specify the dimension $n$. The usual  inner product of $L^2(D)^3$ is $$
(\bs{v}, \bs{w})_{D}:= \int_{D} \bs{v} \cdot \bs{w}\qquad \forall \bs{v},\bs{w}\in L^2(D)^3.$$

Let  $\Omega \subset \mathbb{R}^{3}$ be an admissible domain in the sense of Definition \ref{addomain3}. In the context of \eqref{ns}, we need a functional setting appropriate to handle the dynamic boundary conditions on $\Gamma_O$ and the divergence-free condition for the velocity in $\Omega$. We begin by introducing the Hilbert space
$$
{\mathcal L}^2 := L^2(\Omega)^3 \times L^2(\Gamma_O)^3 \, ,
$$
consisting of all pairs $\bs{v}:=(\bs{v}_\Omega,\bs{v}_{\Gamma_O}) \in L^2(\Omega)^3 \times L^2(\Gamma_O)^3$, endowed with the scalar product and norm
\begin{equation} \label{dirscal0}
	(\bs{v}, \bs{w})_{\mathcal L^2} := \int_{\Omega} \bs{v}_\Omega \cdot \bs{w}_\Omega + 
	\int_{\Gamma_O} \bs{v}_{\Gamma_O} \cdot \bs{w}_{\Gamma_O},\quad \|\bs{v}\|_{{\mathcal L}^2}:=\left(\|\bs{v}_\Omega\|_{L^2(\Omega)}^2+\|\bs{v}_{\Gamma_O}\|_{L^2(\Gamma_O)}^2\right)^{\frac{1}{2}}\quad \forall \bs{v},\bs{w}\in\mathcal{L}^2.
\end{equation}
We note that if $\bs{v}\in H^1(\Omega)^3$ then $\bs{v}\in \mathcal{L}^2$ because $\bs{v}_{\Gamma_O}$ coincides with the trace of $\bs{v}_\Omega$ on $\Gamma_O$.
We introduce the linear mapping $\mathcal{J} : H^{1}(\Omega)^3 \to {\mathcal L}^2$ through the expression
\[
\mathcal{J}(\bs{v}) = (\bs{v}, \gamma_{\Gamma_O}(\bs{v})) \qquad \forall\, \bs{v} \in H^{1}(\Omega)^3,
\]
where $\gamma_{\Gamma_O} : H^{1}(\Omega)^3 \to L^2(\Gamma_O)^3$ represents the restriction to $\Gamma_O$ of the standard trace operator defined on $H^{1}(\Omega)^3$. Since the mapping $\mathcal{J} : H^{1}(\Omega)^3 \to \mathcal{J}(H^{1}(\Omega)^3)$ is certainly injective, from now on we shall identify $H^{1}(\Omega)^3$ (or any subspace of it) with $\mathcal{J}(H^{1}(\Omega)^3)$.

Let
$$
\begin{aligned}
{\mathcal D}:=&\left\{\bs{v}\in \mathcal{C}^\infty(\overline\Omega)^3 \mid  \nabla\cdot \bs{v}=0,\ \ \text{supp}(\bs{v})\cap\overline{\Gamma_W}=\emptyset \right\} \, ,\\[5pt]
{\mathcal D}_*:=& \left\{\bs{v}\in \mathcal{C}^\infty(\overline\Omega)^3 \mid \nabla\cdot \bs{v}=0,\ \ \text{supp} (\bs{v})\cap(\overline{\Gamma_I\cup\Gamma_W})=\emptyset \right\}.
\end{aligned}
$$ 
By taking the closure of the sets ${\mathcal D}$ and ${\mathcal D}_*$ in the topology of $H^1(\Omega)^3$, we obtain the following spaces:
\begin{equation}\label{funcspacesH2}
	\begin{aligned}
		\mathcal{V} :=& \left\lbrace \bs{v} \in H^1(\Omega)^3 \mid  \nabla \cdot \bs{v}=0 \ \ \mbox{in} \ \ \Omega \, , \quad \bs{v}= \bs{0} \ \ \mbox{on} \ \ \Gamma_{W} \right\rbrace \,,  \\[5pt]
		\mathcal{V}_{*} :=& \left\lbrace \bs{v} \in H^1(\Omega)^3 \mid  \nabla \cdot \bs{v}=0 \ \ \mbox{in} \ \ \Omega \, , \quad  \bs{v} = \bs{0} \ \ \mbox{on} \ \ \Gamma_{I} \cup \Gamma_{W} \right\rbrace \, .
	\end{aligned}
\end{equation}
Endowing them with the scalar product
\begin{equation} \label{dirscal}
	(\nabla \bs{v}, \nabla \bs{w})_{\Omega} := \int_{\Omega} \nabla \bs{v} : \nabla \bs{w}\, \qquad \forall \bs{v},\bs{w}\in \VV,
\end{equation}
they become Hilbert spaces.  In particular, for elements of $\mathcal{V}$ and $\mathcal{V}_{*} $, we shall often write $\bs{v}$ instead of $(\bs{v}_\Omega,\bs{v}_{\Gamma_O}) = \big(\bs{v}_\Omega, \gamma_{\Gamma_O}({\bs{v}_\Omega})\big)$.

We introduce the sets $\mathcal{H}$ and $\mathcal{H}_*$ by taking the closure, in the ${\LL^2}$-norm, of ${\mathcal D}$ and ${\mathcal D}_*$, respectively. 
The following characterization of $\mathcal{H}$ and $\mathcal{H}_*$ will be proved in Section \ref{prooflemmaH}.
\begin{lemma}
	\label{lemaH}
	Let  $\Omega \subset \mathbb{R}^{3}$ be an admissible domain. The following characterization holds:
	\begin{equation}\label{funcspacesH}
		\begin{split}
			\mathcal{H} =& \left\lbrace \bs{v} \in {\LL^2}  \mid  \nabla \cdot \bs{v}_\Omega=0 \ \ \mbox{in} \ \ \Omega \, , \ \ \ \bs{v}_\Omega\cdot \bs{n} = 0 \ \ \mbox{on} \ \ \Gamma_{W} \, \ \ \ \bs{v}_\Omega\cdot \bs{n}=\bs{v}_{\Gamma_O}\cdot \bs{n}\ \ \mbox{on} \ \ \Gamma_{O} \right\rbrace \,,  \\[6pt]
			\mathcal{H}_{*} =& \left\lbrace \bs{v} \in {\LL^2} \mid  \nabla \cdot \bs{v}_\Omega=0 \ \  \mbox{in} \ \  \Omega \, , \ \ \  \bs{v}_\Omega\cdot\bs{n} = 0 \ \ \mbox{on} \ \ \Gamma_{I} \cup \Gamma_{W}\, , \ \ \ \bs{v}_\Omega\cdot \bs{n}=\bs{v}_{\Gamma_O}\cdot \bs{n} \ \ \mbox{on} \ \  \Gamma_{O} \right\rbrace \, ,
		\end{split}
	\end{equation}
	where  $\bs{v}_\Omega\cdot\bs{n}$ denotes the normal trace of $\bs{v}_\Omega$ on $\partial \Omega$ which, as the divergence operator, must be understood in distributional sense.
\end{lemma}

\begin{remark} \label{compactemb}
	Recalling that the embeddings $H^{1}(\Omega)^{3} \subset L^{2}(\Omega)^{3}$ and $H^{1}(\Omega)^{3} \subset L^{2}(\Gamma_{O})^{3}$, via the trace operator $\gamma_{\Gamma_O}$, are compact (see, for example, \cite[Theorem 6.2]{necas2011direct}), it follows that the inclusions $\mathcal{V} \subset \mathcal{H}$ and $\mathcal{V}_{*} \subset \mathcal{H}_{*}$ are also compact. 
\end{remark}

\begin{remark}
	Since $\bs{v}_{\Gamma_O}\cdot \bs{n}  \in L^2(\Gamma_O)$, we also have $\bs{v}_{\Omega}\cdot \bs{n}  \in L^2(\Gamma_O)$, for every $\bs{v} \in \mathcal{H} $.
\end{remark}

In view of Lemma \ref{lemaH}, an application of the generalized Stokes formula \cite[Theorem 1.2]{Tem1977b} furnishes
\begin{equation} \label{weakdiv}
\begin{aligned}
0 = \int_{\Omega} \psi (\nabla \cdot \bs{v}_{\Omega} ) & = \langle \bs{v}_{\Omega} \cdot \bs{n}, \psi \rangle_{H^{-1/2}(\partial \Omega), \, H^{1/2}(\partial \Omega)} - \int_{\Omega} \bs{v}_{\Omega} \cdot \nabla \psi \\[6pt]
& = \int_{\Gamma_O} \psi (\bs{v}_{\Gamma_{O}} \cdot \bs{n}) - \int_{\Omega} \bs{v}_{\Omega} \cdot \nabla \psi \qquad \forall \bs{v} \in \mathcal{H}_{*} \, , \ \ \forall \psi \in H^{1}(\Omega) \, ,
\end{aligned}
\end{equation}
yielding, in particular,
$$
\int_{\Gamma_O} \bs{v}_{\Gamma_{O}} \cdot \bs{n} = 0 \qquad \forall \bs{v} \in \mathcal{H}_{*} \, .
$$
This motivates the introduction of the set
\begin{equation} \label{l2zero}
	L^2_{\bs{n},0}(\Gamma_O)^3 := \left\{ \bs{v} \in L^2(\Gamma_O)^3  \ \Big | \ \int_{\Gamma_O} \bs{v} \cdot \bs{n} = 0 \right\} \, ,
\end{equation}
representing the closed subspace of $L^2(\Gamma_O)^{3}$ comprising vector fields which are orthogonal to $\bs{n}$ on $\Gamma_O$. Therefore, we have the following orthogonal decomposition:
\begin{equation} 
	\label{decL1O}
	L^2(\Gamma_O)^{3} = L^{2}_{0, \bs{n}}(\Gamma_{O})^3 \oplus \mathbb{R} \bs{n} \, .
\end{equation}
In fact, given any vector field $\bs{\varphi} \in L^2(\Gamma_O)^{3}$, define
$$
C_{\bs{\varphi}} := \dfrac{1}{| \Gamma_{O} |} \int_{\Gamma_O} \bs{\varphi} \cdot \bs{n} \, ,
$$
and observe that $\bs{\varphi} - C_{\bs{\varphi}} \, \bs{n} \in L^{2}_{0, \bs{n}}(\Gamma_{O})^3$.

Now,  consider the set
$$
\mathcal{G}_{*} := \left\lbrace \bs{v} \in {\mathcal L}^2 \mid \exists p\in H^1(\Omega) \ :\ \bs{v}_\Omega=\nabla p \ \ \mbox{in} \ \ \Omega \, , \quad  \bs{v}_{\Gamma_O} = -p\bs{n} \ \ \mbox{on} \ \ \Gamma_{O} \right\rbrace . 
$$
In Section \ref{prooflemma00} we prove the following result.
\begin{lemma}\label{leraydecomp}
	Let  $\Omega \subset \mathbb{R}^{3}$ be an admissible domain. The following characterization holds: 
	$$
	{\HH_*}^\perp = \mathcal{G}_{*}.
	$$
	This means that the space ${\mathcal L}^2$ admits the decomposition (with respect to the inner product \eqref{dirscal0})
	$$
	{\mathcal L}^2  = \HH_*\oplus\mathcal{G}_{*}.
	$$
\end{lemma}
We conclude this section by introducing the function space 
\begin{equation}\label{U}
	\mathcal {U}_*:=\left\{ \bs{v} \in H^1(\Omega)^3 \mid  \bs{v} = \bs{0} \ \ \mbox{on} \ \ \Gamma_{I} \cup \Gamma_{W} \right\}, 
\end{equation}
its dual $\mathcal{U}'_*$ and $\langle\cdot,\cdot\rangle_{*}:=\langle\cdot,\cdot\rangle_{\mathcal{U}'_*,\,\mathcal{U}_{*}}$ as the corresponding duality product.
This space will be useful throughout the paper. This is because, when we consider an external force $\bs{f} \in L^2(0,T;H^{-1}(\Omega)^3)$,  it should be supported in $\Omega$; otherwise, the boundary condition on $\Gamma_{O}$ would be affected by $\bs{f}$. Therefore, we define its action on test functions $\bs{\varphi} \in {\mathcal V}_*$ through
\begin{equation}\label{fomega}
	\langle\boldsymbol{f}(t),\boldsymbol{\varphi}\rangle_{*}  = \langle \boldsymbol{f}(t) , \boldsymbol{\varphi} \rangle : = \langle \boldsymbol{f}(t)   ,  P_{H^{1}_{0}(\Omega)} (\bs{\varphi}) \rangle_{H^{-1}(\Omega),H^{1}_{0}(\Omega)},
\end{equation}
where $P_{H^{1}_{0}(\Omega)}$ is the projection of $\mathcal{U}_*$ onto $H^{1}_{0}(\Omega)^3$. 
\subsection{The associated Stokes problem under mixed boundary conditions}\label{substokes}
The associated Stokes problem with mixed boundary conditions in $\Omega$ is: 
\begin{equation}\label{stokesf}
	\left\{
	\begin{array}{ll}
		-\Delta  \bs{u}+\nabla \pi=\bs{f}_\Omega \, , \quad  \nabla\cdot \bs{u}=0 \ \ &\mbox{ in } \  \Omega \, , \\[5pt]
		\bs{u}=\bs{0} \ \ &\mbox{ on } \  \Gamma_{I}\cup\Gamma_{W}  \, , \\[5pt]
		\dfrac{\partial  \bs{u}}{\partial \bs{n}} - \pi \bs{n}  =\bs{f}_{\Gamma_O}\ \ &\mbox{ on } \  \Gamma_{O} \, ,
	\end{array}
	\right.
\end{equation}
where $\bs{f}=(\bs{f}_\Omega,\bs{f}_{\Gamma_O})\in \mathcal L^2$ is a given force with components acting in $\Omega$ and $\Gamma_O$, respectively. A pair $(\bs{u},\pi)\in \VV_*\times L^2(\Omega)$ is a \textit{weak solution} of \eqref{stokesf} if $\bs{u}$ satisfies
\begin{equation}\label{weak:stokesf}
	\begin{aligned}
		&\big(\nabla \bs{u}, \nabla \bs{\varphi}\big)_\Omega 
		= \big( \bs{f} ,\bs{\varphi}\big)_{\mathcal L^2}\qquad \forall \bs{\varphi} \in \mathcal{V}_{*} \,
	\end{aligned}
\end{equation}
and $(\bs{u},\pi)$ satisfy \eqref{stokesf}$_1$ in the sense of distributions, that is,
\begin{equation}
\big(\nabla \bs{u}, \nabla \bs{\varphi}\big)_\Omega +\big(\pi, \nabla \cdot\bs{\varphi}\big)_\Omega 
= \big( \bs{f} ,\bs{\varphi}\big)_{\mathcal L^2}\qquad \forall \bs{\varphi} \in \mathcal{U}_{*} \, .
\label{weak:stokesf2}
\end{equation}
For any $\bs{f}=(\bs{f}_\Omega,\bs{f}_{\Gamma_O})\in \mathcal L^2$, existence of a unique weak solution to system \eqref{stokesf} follows directly from the Riesz Representation Theorem, as well as standard results concerning representation of bounded functionals vanishing on subspaces of $H^{1}(\Omega)^3$ comprising divergence-free vector fields, see \cite[Section 1.8]{korobkov2024steady}. Regularity of weak solutions to the Stokes problem under mixed boundary conditions and regular data, in polyhedral domains, was studied in a general $W^{2,p}(\Omega)$-setting by Maz'ya \& Rossmann \cite{mazya}; in \cite[Corollary A.3]{benes3d} the Author adapted the result of Maz'ya \& Rossmann to problem \eqref{stokesf} under homogeneous boundary conditions ($\bs{f}_{\Gamma_O} \equiv \bs{0}$), obtaining standard $H^2(\Omega)^3$-regularity whenever $\bs{f}_\Omega\in L^2(\Omega)^3$. Such results could be recovered in our setting by assuming, at least, $\bs{f}=(\bs{f}_\Omega,\bs{f}_{\Gamma_O})\in L^2(\Omega)^3\times H^{1/2}(\Gamma_O)^3$, i.e., a datum more regular than $\LL^2$. When $\bs{f}\in \LL^2$, the situation is far more delicate, and the classical $H^2(\Omega)^3$-regularity seems out of reach.  In the next proposition, proved in Section \ref{proofprop} and inspired by \cite{brown2010mixed}, we describe the expected regularity for \eqref{stokesf} in our setting.
\begin{proposition}\label{benes}
	Let $\Omega\subset \R^{3}$ be an admissible domain, $\bs{f}=(\bs{f}_\Omega,\bs{f}_{\Gamma_O}) \in \mathcal{L}^2$ and $(\bs{u},\pi) \in \VV_*\times L^2(\Omega)$ be the unique weak solution of \eqref{stokesf}. Then, $\bs{u} \in H^{3/2}(\Omega)^3$ and the bound
	\begin{equation}\label{stima}
		\|\bs{u}\|_{H^{3/2}(\Omega)} \leq C\|\bs{f}\|_{\mathcal L^{2}} \, ,
	\end{equation}
	holds for some $C:=C(\Omega)>0$. 
\end{proposition}

The \textit{Stokes operator} under the mixed boundary conditions \eqref{stokesf}$_2$-\eqref{stokesf}$_3$ is the bounded linear application $\mathcal{A} : \mathcal{V}_{*} \longrightarrow \mathcal{V}_{*}'$ given by the formula
$$
\langle \mathcal{A}\bs{v}, \bs{\varphi} \rangle_{\mathcal{V}_{*}',\mathcal{V}_{*}} := \big(\nabla \bs{v}, \nabla \bs{\varphi}\big)_\Omega \qquad \forall \bs{v}, \bs{\varphi} \in \mathcal{V}_{*} \, .
$$
As in the case of homogeneous Dirichlet boundary conditions \cite[Chapter IV, Section 5.2]{boyer2012mathematical}, we can consider $\mathcal{A}$ as an unbounded operator $\mathcal{A} : D(\mathcal{A}) \subset \mathcal{H}_{*} \longrightarrow \mathcal{H}_{*}$, with domain given by
$$
D(\mathcal{A}) := \{ \bs{v} \in \mathcal{V}_{*} \ | \ \mathcal{A}\bs{v} \in \mathcal{H}_{*} \} \, .
$$

In Section \ref{proofstokestheo} we prove the following result:

\begin{theorem} \label{stokesoperatortheo}
Let $\Omega\subset \R^{3}$ be an admissible domain. Then, $\mathcal{A}$ is a densely defined, closed and self-adjoint operator on $\mathcal{H}_{*}$ such that
\begin{equation} \label{inclusionsa}
\left( \mathcal{V}_{*} \cap H^{2}(\Omega)^{3} \right) \subset D(\mathcal{A}) \subset \left( \mathcal{V}_{*} \cap H^{3/2}(\Omega)^{3} \right)  \, .
\end{equation}
There holds the estimate
\begin{equation} \label{estimatea}
\|\bs{v}\|_{H^{3/2}(\Omega)} \leq C \left\| \mathcal{A}\bs{v} \right\|_{{\mathcal L}^2} \qquad \forall \bs{v}\in D(\mathcal{A}) \, ,
\end{equation}
for some $C:=C(\Omega)>0$. Moreover, there exist:
\begin{itemize}[leftmargin=4mm]
	\item[$\bullet$] an increasing sequence $\{\lambda_k\}_{k \geq 1} \subset (0, +\infty)$ such that $\lambda_k \to +\infty$ as $k \to +\infty$;
	\item[$\bullet$] a sequence $\{\bs{v}_k\}_{k \geq 1} \subset \left( \mathcal{V}_{*} \cap H^{2}(\Omega)^{3} \right)$ forming a complete orthonormal system in $\mathcal{H}_{*}$ and a complete orthogonal system in $D(\mathcal{A})$;
	\item[$\bullet$] a sequence $\{\pi_{k}\}_{k \geq 1} \subset H^{1}(\Omega)$;
\end{itemize}
such that, for every $k \geq 1$, they satisfy in strong form the following eigenvalue problem:
\begin{equation}\label{stokes}
	\left\{
	\begin{array}{ll}
		-\Delta  \bs{v}_{k}+\nabla \pi_{k}=\lambda_{k}  \bs{v}_{k} \, , \quad  \nabla\cdot \bs{v}_{k}=0 \ \ &\mbox{ in } \  \Omega \, , \\[5pt]
		\bs{v}_{k}=\bs{0} \ \ &\mbox{ on } \  \Gamma_{I}\cup\Gamma_{W}  \, , \\[5pt]
		\dfrac{\partial  \bs{v}_{k}}{\partial \bs{n}} - \pi_{k} \bs{n}  =\lambda_{k}  \bs{v}_{k}\ \ &\mbox{ on } \  \Gamma_{O} \, .
	\end{array}
	\right.
\end{equation}
\end{theorem}

\subsection{Construction of the reference flow}\label{reference}

Since $\Omega_{2}$ is a given cylinder, we can define a family of \textit{fully developed flows}, characterized by the fact that the only nonzero component of the associated velocity field is directed along the axis of $\Omega_{2}$. 

Put $\bs{x}:=(x,y,z)\in\R^3$ and introduce the horizontal cylinder $\mathcal{P}_2:= \Theta_{2} \times (0,\ell_{2})$. Let $w\in \mathcal{C}^{2}(\overline{\Theta_{2}})$ be the unique strong solution to the following torsion problem:
$$
-\Delta w = 1 \ \ \mbox{ in } \ \ \Theta_{2} \, , \qquad w=0 \ \ \mbox{ on } \ \ \partial \Theta_{2} \, ,
$$
and define
$$
\varrho := \int_{\Theta_{2}} w = \int_{\Theta_2} | \nabla w |^{2} > 0 \, .
$$ 
The time-dependent Hagen-Poiseuille flow associated to $\mathcal{P}_{2}$, having kinematic viscosity $\nu>0$ and time-dependent flow rate $\Phi \in \mathcal{C}([0,T))$, is characterized by the velocity field $\bs{U}_{\! \Phi} : \overline{\mathcal{P}_{2}}\times (0,T) \longrightarrow \mathbb{R}^3$ and scalar pressure $Q_{\Phi}: \overline{\mathcal{P}_{2}}\times (0,T) \longrightarrow \mathbb{R}$ defined as
\begin{equation}\label{poi23d}
	\bs{U}_{\! \Phi}(\bs{x},t) := \dfrac{\Phi(t)}{\varrho} \big(0,0, w(x,y) \big) \quad \text{and} \quad Q_{\Phi}(\bs{x},t) := -\dfrac{\nu \Phi(t)}{\varrho} \left(z - \ell_{2} \right) \quad \forall (\bs{x},t) \in \overline{\mathcal{P}_{2}}\times [0,T) \, .
\end{equation}
Now, given the geometry of $\Omega_{2}$, there exists a point $\bs{y} \in \mathbb{R}^3$ and a rotation matrix $\mathcal{Q} \in \text{SO}(3)$ such that the set $\overline{\Omega_{2}}$ can be mapped onto $\overline{\mathcal{P}_{2}}$ through the following rigid transformation: 
$$
\mathbf{T}(\bs{x}) = \bs{y} + \mathcal{Q} \, \bs{x} \qquad \forall \bs{x}\in \overline{\Omega_{2}} \, .  
$$
Then, the time-dependent Hagen-Poiseuille flow associated to $\Omega_{2}$ can be defined as
\begin{equation}\label{poi23dd}
	\bs{V}_{\! \Phi}(\bs{x},t) := \mathcal{Q}^{\top} \bs{U}_{\! \Phi}\big(\mathbf{T}(\bs{x}),t\big) \quad \text{and} \quad
	P_{\Phi}(\bs{x},t) := Q_{\Phi}(\mathbf{T}(\bs{x}),t) \qquad \forall (\bs{x},t) \in \overline{\Omega_{2}}\times [0,T) \, ,
\end{equation}
so that 
\begin{equation}\label{flux03d}
	\left\{
	\begin{aligned}
		\begin{aligned}
			& -\nu\Delta \bs{V}_{\! \Phi} + \nabla P_{\Phi}  = \bs{0} \, , \quad  \nabla\cdot \bs{V}_{\! \Phi} = 0 \ \ \mbox{ in } \ \ \Omega_{2}\times (0,T) \, , \\[5pt]
			& \bs{V}_{\! \Phi}=\bs{0} \ \ \hspace{47mm}\mbox{ on } \ \ (\Gamma_{W} \cap \partial \Omega_{2})\times (0,T) \, , \\[5pt]
			& \int_{\Sigma} \bs{V}_{\! \Phi}(\bs{x},t) \cdot \bs{n} = \Phi(t) \quad \text{for any inner cross-section $\Sigma \subset \overline{\Omega_{2}}$ of $\Omega_{2}$,  $\forall t\in[0,T)$} \, .
		\end{aligned}
	\end{aligned}
	\right.
\end{equation}

A smooth reference flow, capturing the inlet velocity on $\Gamma_{I}$ and subject to a do-nothing boundary condition on $\Gamma_{O}$, needs to be constructed (moreover, such reference flow will necessarily preserve the flux condition \eqref{flux03d}$_3$). Given $k \in\{1,2\}$ and $\bs{g}_{*}\in H^1(0,T;H_{00}^{k-1/2}(\Gamma_{I})^{3})$,
 we define its \textit{flux across $\Gamma_{I}$} as the function
\begin{equation} \label{gflux}
	\Phi_{*}(t) := - \int_{\Gamma_{I}} \bs{g}_{*}\cdot \bs{n} \,\qquad \forall t\in [0,T) \, ,
\end{equation}
so that $\Phi_{*} \in H^1(0,T)$. As an extension of \cite[Lemma 2.1]{fss}, in Section \ref{prooflemma} we prove the following:
\begin{lemma}\label{lemma0}
	Let $\Omega \subset \mathbb{R}^{3}$ be an admissible domain. For any $k \in \{1,2\}$, $\bs{g}_{*}\in H^1(0,T;H_{00}^{k-1/2}(\Gamma_{I})^3)$ and $\sigma_{*} \in H^1(0,T;H^{k-3/2}(\Gamma_{O}))$, there exists a pair   $$(\bs{W}_{\! \! *}, \Pi_{*}) \in H^1(0,T;H^{k}(\Omega)^3\cap\VV) \times H^{1}(0,T;H^{k-1}(\Omega)),$$
	satisfying
	\begin{equation}\label{stokesintrolemma}
		\left\{
		\begin{aligned}
			& \nabla\cdot \bs{W}_{\! \! *} = 0 \ \ \hspace{16mm}\mbox{ in } \ \ Q_T \, , \\[5pt]
			& \bs{W}_{\! \! *}=\bs{g}_{*}\ \ \hspace{20mm} \mbox{ on } \ \ \Gamma_{I}\times (0,T)  \, , \\[5pt] &\bs{W}_{\! \! *}=\bs{0} \ \ \hspace{22mm}\mbox{ on } \ \ \Gamma_{W}\times (0,T) \, , \\[5pt]
			& \nu \dfrac{\partial \bs{W}_{\! \! *}}{\partial \bs{n}} - \Pi_{*} \, \bs{n} = \sigma_{*} \, \bs{n} \ \ \mbox{ on } \ \ \Gamma_{O}\times (0,T)  \, .
		\end{aligned}
		\right.
	\end{equation}
	Moreover,  
	\begin{equation}\label{flux}
		\bs{W}_{\! \! *}=\bs{V}_{\! \Phi_{*}} \ \ \mbox{ on } \ \ \Gamma_{O}\times (0,T) \, ,
	\end{equation}
	where $\bs{V}_{\! \Phi_{*}} \in H^1(0,T;\mathcal{C}^{2}(\overline{\Omega_{2}})^3)$ is the Hagen-Poiseuille flow defined in \eqref{poi23dd}-\eqref{flux03d} with flux given by \eqref{gflux}, and there hold
	the estimates
	\begin{equation}\label{fluxestimate}
		\begin{split}
			&\| \bs{W}_{\! \! *}\|_{H^1(0,T;H^{k}(\Omega))} \leq C^{*}\| \bs{g}_{*}\|_{H^1(0,T;H^{k-1/2}(\Gamma_{I}))}, \\[6pt]
			&\| \Pi_{*} \|_{H^1(0,T;H^{k-1}(\Omega))} \leq C^{*}   \| \sigma_{*}\|_{H^1(0,T;H^{k-3/2}(\Gamma_{O}))},
		\end{split}
	\end{equation}
	for some constant $C^{*}:=C^{*}(\Omega)> 0$. 
\end{lemma}
\begin{remark}\label{rmk}
	We point out that the reference flow inherits the time regularity assumed for the data $\bs{g}_{*}$ and $\sigma_{*}$. For instance, if $\bs{g}_{*}\in H^m(0,T;H_{00}^{k-1/2}(\Gamma_{I})^3)$ and $\sigma_{*} \in H^m(0,T;H^{k-3/2}(\Gamma_{O}))$ with $m\geq 1$ integer and $k\in\{1,2\}$, then $(\bs{W}_{\! \! *}, \Pi_{*}) \in H^m(0,T;H^{k}(\Omega)^3\cap\VV) \times H^{m}(0,T;H^{k-1}(\Omega))$.
\end{remark}
	
\newpage	
\section{Main results}\label{main}
As specified in the Introduction, we implicitly assume that the compatibility condition between initial datum $\bs{v}_0$ and boundary datum $\bs{g}_{*}$ at $t=0$, i.e. $\gamma_{\Gamma_I}(\bs{v}_{0,\Omega}) = \bs{g}_{*}(\cdot,0)$, holds. Recalling the duality product in \eqref{fomega}, we begin with the definition of weak solutions to problem \eqref{ns}:
\begin{definition}\label{weaksolution}
	Let $T>0$, $\Omega \subset \mathbb{R}^{3}$ be an admissible domain and $\bs{v}_0\in \HH$.	Given $\bs{f} \in L^2(0,T;H^{-1}(\Omega)^3)$, $\bs{g}_{*} \in H^1(0,T;H_{00}^{1/2}(\Gamma_{I})^3)$ and $\sigma_{*} \in H^1(0,T;H^{-1/2}(\Gamma_{O}))$, let $(\bs{W}_{\! \! *}, \Pi_{*}) \in H^1(0,T;\VV) \times H^{1}(0,T;L^2(\Omega))$ be the reference flow from Lemma \ref{lemma0}. 
	A vector field $\bs{v} \in L^\infty(0,T;\mathcal{H})\cap L^2(0,T;\mathcal{V}) $ is called a \textbf{weak solution} of problem \eqref{ns} if $\bs{v} - \bs{W}_{\! \! *} \in L^\infty(0,T;\mathcal{H}_*)\cap L^2(0,T;\mathcal{V}_*)$ and
	\begin{equation} \label{nsweak}
		\begin{split}
			&\int_0^T \phi'(t) (\bs{v}(t),\bs{\varphi})_{\LL^2}\, dt+\phi(0)(\bs{v}_{0},\bs{\varphi})_{\LL^2}=\int_0^T\bigg[\nu (\nabla \bs{v}(t), \nabla \bs{\varphi})_\Omega + \int_\Omega(\bs{v}(t) \cdot \nabla)\bs{v}(t)\cdot \bs{\varphi} +\\& \dfrac{1}{2} \int_{\Gamma_{O}} [\bs{v}(t) \cdot \bs{n}]^{-}\left(\bs{v}(t) - \bs{W}_{\! \! *}(t)\right) \cdot \bs{\varphi} - \langle\bs{f}(t), \bs{\varphi}\rangle  - \langle\sigma_{*}(t)\bs{n}\,,\, \bs{\varphi} \rangle_{H^{-1/2}(\Gamma_O), \, H_{00}^{1/2}(\Gamma_O)} \bigg]\phi(t)dt \, ,
		\end{split}
	\end{equation}
	for every $\bs{\varphi} \in \mathcal{V}_{*}$ and for all $\phi\in \mathcal{C}^\infty_0\big([0,T)\big)$.
\end{definition}

Results concerning existence of weak solutions to the three-dimensional Navier-Stokes equations under mixed boundary conditions, including do-nothing boundary conditions, can be found in \cite{braack2014directional,bruneau1996new}. In the following result, which is proved through the Galerkin method in Section \ref{proof}, we extend to system \eqref{ns} classical weak regularity results on the time-dependent Navier-Stokes problem. 

\begin{theorem}\label{existence}
Let $T>0$, $\Omega \subset \mathbb{R}^{3}$ be an admissible domain and $\bs{v}_0\in \HH$.	Given $\bs{f} \in L^2(0,T;H^{-1}(\Omega)^3)$, $\bs{g}_{*} \in H^1(0,T;H_{00}^{1/2}(\Gamma_{I})^3)$ and $\sigma_{*} \in H^1(0,T;H^{-1/2}(\Gamma_{O}))$, let $(\bs{W}_{\! \! *}, \Pi_{*}) \in H^1(0,T;\VV) \times H^{1}(0,T;L^2(\Omega))$ be the reference flow from Lemma \ref{lemma0}. Then, the problem \eqref{ns} admits a weak solution such that
\begin{equation}\label{reg_weak}
\bs{v}\in L^\infty(0,T;\HH)\cap L^2(0,T;\VV) \, ,\qquad \partial_t\bs{v}\in L^{4/3}(0,T;\VV').
\end{equation}
  Moreover, it satisfies the strong energy inequality
	\begin{equation}\label{sei}
	\begin{split}
		&\frac{1}{2}\|\bs{v}(t)-\bs{W}_{\! \! *}(t)\|^2_{\LL^2}+\nu\int_s^t\|\nabla \big(\bs{v}(\tau)-\bs{W}_{\! \! *}(\tau)\big)\|^2_{L^2(\Omega)}d\tau\\\hspace{0mm}\leq& \, \frac{1}{2}\|\bs{v}(s)-\bs{W}_{\! \! *}(s)\|^2_{\LL^2}\!-\!\int_s^t\Big(\partial_t\bs{W}_{\! \! *}(\tau)\,,\bs{v}(\tau)-\bs{W}_{\! \! *}(\tau)\Big)_{\LL^2}d\tau\\&+ 		\int_s^t\Big\langle\nu \Delta\bs{W}_{\! \! *}(\tau)-(\bs{v}(\tau)\cdot \nabla)\bs{W}_{\! \! *}(\tau)+\bs{f}(\tau)\,,\bs{v}(\tau)-\bs{W}_{\! \! *}(\tau)\Big\rangle d\tau
	\end{split}
\end{equation}
	for almost all $s\geq0$, including $s=0$, and all $t\in[s,T]$.
\end{theorem}
We refer to Section \ref{recpressuresection} for the corresponding recovery of the pressure, in particular, Theorem \ref{press}.

\noindent
Taking  initial data sufficiently small and regular, in the next result we show that the weak solution from Theorem \ref{existence} is, in turn, a global and unique strong solution. The proof is given in Section \ref{proof_strong}.

\begin{theorem}\label{strong}
Let $T>0$, $\Omega \subset \mathbb{R}^{3}$ be an admissible domain and $\bs{v}_0\in \VV$. Given $\bs{f} \in L^2(Q_T)^3$,  $\bs{g}_{*} \in H^1(0,T;H_{00}^{3/2}(\Gamma_{I})^3)$ and $\sigma_{*} \in H^1(0,T;H^{1/2}(\Gamma_{O}))$, let $(\bs{W}_{\! \! *}, \Pi_{*}) \in H^1(0,T;H^2(\Omega)^3\cap\VV) \times H^{1}(Q_T)$ be the reference flow from Lemma \ref{lemma0}. If there holds the inequality
\begin{align}
& \|\nabla (\bs{v}_0- \bs{W}_{\! \! *}(\cdot,0))\|_{L^2(\Omega)}^2 + \dfrac{C}{\nu} \! \int_0^T \! \! \left( \|\bs{W}_{\! \! *}(t)\|_{H^2(\Omega)}^4 +\nu^2\|\bs{W}_{\! \! *}(t)\|_{H^2(\Omega)}^2 +\|\partial_t\bs{W}_{\! \! *}(t)\|_{{\mathcal L}^2}^2+\|\bs{f} (t)\|_{L^2(\Omega)}^{2} \right)dt \notag \\[6pt]
& \leq \frac{\nu^2}{C^2} \exp\!\left(-\frac{C M}{\nu^2}\right) \label{datopiccolo3d} \, ,
\end{align}
where $C:=C(\Omega)$ and $M:=M(\Omega,\nu,T,\bs{v}_0,\bs{W}^*,\bs{f})>0$,
	then the weak solution $\bs{v}$ of problem \eqref{ns} (found in Theorem \ref{existence}) satisfies 
		\begin{equation}\label{regularity3d}
		\bs{v}\in L^\infty(0,T;\VV)\cap L^2(0,T;H^{3/2}(\Omega)^3)\qquad \bs{v}_t,\, \A \bs{v}\in L^2(0,T;\LL^2)\qquad p\in L^2(0,T;W^{1,3/2}(\Omega)),
	\end{equation}
and thus, it is unique and global in time.
	\end{theorem}
\begin{remark}
The condition \eqref{datopiccolo3d} is verified, for instance, taking $\nu>0$ sufficiently large; indeed, $M/\nu^2\rightarrow 0$ as $\nu\rightarrow\infty$, see \eqref{eq0044}. Alternatively, \eqref{datopiccolo3d} holds assuming sufficiently small data ($\bs{v}_0, \bs{f}, \bs{g}_*, \sigma_{*}$ and, in turn, $\bs{W}_{\!\!*}$) in the respective norms. 
\end{remark}

\section{Proofs of the results}\label{proofs}

\subsection{Proof of Lemma \ref{lemaH}}\label{prooflemmaH}
We only prove the second identity in Lemma \ref{lemaH}, as the first one can be obtained similarly. For this purpose, we introduce the set
$$
{\mathcal K}_{*}  := \left\lbrace \bs{v} \in {\LL^2} \mid  \nabla \cdot \bs{v}_\Omega=0 \ \  \mbox{in} \ \  \Omega \, , \ \ \  \bs{v}_\Omega\cdot\bs{n} = 0 \ \ \mbox{on} \ \ \Gamma_{I} \cup \Gamma_{W}\, , \ \ \ \bs{v}_\Omega\cdot \bs{n}=\bs{v}_{\Gamma_O}\cdot \bs{n} \ \ \mbox{on} \ \  \Gamma_{O} \right\rbrace \, ,
$$
and we will show that ${\mathcal K}_{*} =\overline{{\mathcal D}_{*}}^{{\mathcal L}^2} ={\mathcal H}_{*}$. 

Firstly, we prove that ${\mathcal K}_{*}$ is a closed subspace of ${\mathcal L}^2$. Let $\left\{ (\bs{v}^{k}_{\Omega}, \bs{v}^{k}_{\Gamma_{O}}) \right\}_{k \in \mathbb N}$ be a sequence in ${\mathcal K}_*$ which converges in ${\mathcal L}^2$, meaning there exist $\bs{v}_\Omega \in L^2(\Omega)^3$ and $\bs{v}_{\Gamma_O} \in L^2(\Gamma_O)^3$ such that
	$$
	\| \bs{v}^{k}_{\Omega} - \bs{v}_\Omega \|_{L^2(\Omega)} \to 0 \,\, \text{ and }\,\, \| \bs{v}^{k}_{\Gamma_{O}} - \bs{v}_{\Gamma_O} \|_{L^2(\Gamma_O)} \to 0, \text{ as } k \to \infty.
	$$
	Our aim is to show that $(\bs{v}_\Omega,\bs{v}_{\Gamma_O}) \in {\mathcal K}_{*}$. Since $\bs{v}^{k}_{\Omega} \cdot \bs{n} = 0$ on $\Gamma_{I} \cup \Gamma_{W}$ and $\bs{v}^{k}_{\Omega} \cdot \bs{n}=\bs{v}^{k}_{\Gamma_{O}} \cdot \bs{n}$ on $\Gamma_{O}$, the generalized Stokes formula (arguing as in \eqref{weakdiv}) furnishes
	$$
	0 = \int_\Omega \varphi (\nabla \cdot \bs{v}^{k}_{\Omega} )  =  \int_{\Gamma_{O}} \varphi (\bs{v}^{k}_{\Gamma_{O}} \cdot \bs{n}) -  \int_\Omega \bs{v}^{k}_{\Omega}   \cdot  \nabla \varphi  \qquad \forall \varphi \in H^1(\Omega) \, , \ \forall k \in \mathbb{N} \, .
	$$
Passing to the limit as $k \to \infty$ yields
\begin{equation} 	\label{divogo}
\int_{\Gamma_{O}} \varphi (\bs{v}_{\Gamma_{O}} \cdot \bs{n}) =  \int_\Omega \bs{v}_{\Omega}   \cdot  \nabla \varphi  \qquad \forall \varphi \in H^1(\Omega) \, .
\end{equation}
In particular, from \eqref{divogo}, we have
$$
\int_\Omega \bs{v}_\Omega  \cdot  \nabla \varphi   = 0 \quad \forall \varphi \in H^1_0(\Omega) \iff  \langle  \nabla \cdot  \bs{v}_\Omega ,  \varphi  \rangle_{H^{-1}(\Omega), \, H^1_0(\Omega)} = 0  \qquad\forall \varphi \in H^1_0(\Omega),
$$
meaning that $\nabla \cdot \bs{v}_\Omega = 0$ in $\Omega$, in weak sense. From this divergence-free condition and $\bs{v}_\Omega \in L^2(\Omega)^3$, it follows by the generalized Stokes formula that 
$$
0 = \langle  \bs{v}_\Omega \cdot \bs{n} , \varphi  \rangle_{H^{-1/2}(\partial \Omega), \, H^{1/2}(\partial \Omega)} - \int_\Omega \bs{v}_\Omega  \cdot \nabla \varphi \qquad \forall \varphi \in H^1(\Omega) \, ,
$$
which, when confronted with \eqref{divogo}, yields
\begin{equation}
	\langle  \bs{v}_\Omega \cdot \bs{n} , \varphi  \rangle_{H^{-1/2}(\partial \Omega), \, H^{1/2}(\partial \Omega)} = \int_{\Gamma_O} \varphi ( \bs{v}_{\Gamma_O} \cdot \bs{n} ) \qquad \forall \varphi \in H^1(\Omega) \, .
	\label{igualt}
\end{equation}
A particular case of  \eqref{igualt} is
$$
\langle  \bs{v}_\Omega \cdot \bs{n} , \varphi  \rangle_{H^{-1/2}(\partial \Omega), \, H^{1/2}(\partial \Omega)} = \int_{\Gamma_O} \varphi ( \bs{v}_{\Gamma_O} \cdot \bs{n} )  \qquad \forall \varphi \in H^1(\Omega)\, \text{ such that }  \,\gamma_{\Gamma_O}(\varphi) \in H^{1/2}_{00}(\Gamma_O) \, ,
$$
that is,
\begin{equation}
	\langle  \bs{v}_\Omega \cdot \bs{n} -  \bs{v}_{\Gamma_O} \cdot \bs{n} , \varphi  \rangle_{H^{-1/2}(\Gamma_O), \, H^{1/2}_{00}(\Gamma_O)} = 0
	\qquad \forall \varphi \in H^1(\Omega)\, \text{ such that } \, \gamma_{\Gamma_O}(\varphi) \in H^{1/2}_{00}(\Gamma_O) \, .
	\label{tra00}
\end{equation}
Actually, we have 
$$
\langle  \bs{v}_\Omega \cdot \bs{n} -  \bs{v}_{\Gamma_O} \cdot \bs{n} , \psi  \rangle_{H^{-1/2}(\Gamma_O), \, H^{1/2}_{00}(\Gamma_O)} = 0 
\quad \forall \psi \in  H^{1/2}_{00}(\Gamma_O) \iff \bs{v}_\Omega \cdot \bs{n} =  \bs{v}_{\Gamma_O} \cdot \bs{n}\,  \text{ in }\, H^{-1/2}(\Gamma_O).
$$
In order to justify this last assertion, relying on the  validity of \eqref{tra00}, it is sufficient to observe that for any $\psi \in  H^{1/2}_{00}(\Gamma_O)$ there exists $\varphi \in H^1(\Omega)$ such that $ \gamma_{\Gamma_O}(\varphi) = \psi$. This is immediate, since we can consider $\widetilde{\psi} \in  H^{1/2}(\partial \Omega)$, the extension by zero of $\psi$ to the whole boundary $\partial \Omega$, and then take $\varphi \in H^1(\Omega)$ as the extension of  $\widetilde{\psi}$  to $\Omega $. From \eqref{igualt} we also have
$$
\langle  \bs{v}_\Omega \cdot \bs{n} , \varphi  \rangle_{H^{-1/2}(\partial \Omega), \, H^{1/2}(\partial \Omega)} = 0 \qquad \forall \varphi \in H^1(\Omega) \,\text{ such that }\,  \gamma_{\Gamma_O}(\varphi) = 0 \, .
$$
More specifically, denoting by $\gamma_{\Gamma_W \cup \Gamma_I}$ the restriction of the trace operator to $\Gamma_W \cup \Gamma_I$, 
$$
\langle  \bs{v}_\Omega \cdot \bs{n} , \varphi  \rangle_{H^{-1/2}(\Gamma_W \cup \Gamma_I), \, H^{1/2}_{00}(\Gamma_W \cup \Gamma_I)} = 0  \qquad \forall \varphi \in H^1(\Omega)\, \text{ such that } \, \gamma_{\Gamma_W \cup \Gamma_I}(\varphi) \in H^{1/2}_{00}(\Gamma_W \cup \Gamma_I) \, ,
$$
and by an argument similar to the one used above, we conclude that 
$$
\bs{v}_\Omega \cdot \bs{n} =  0 \ \ \text{ in } \ \ H^{-1/2}(\Gamma_W \cup \Gamma_I) \, .
$$
Therefore, $(\bs{v}_\Omega,\bs{v}_{\Gamma_O}) \in {\mathcal K}_{*}$. This completes the proof that ${\mathcal K}_{*}$ is closed in ${\mathcal L}^2$. 

Clearly, ${\mathcal D}_{*} \subset {\mathcal K}_{*}$ and therefore $\HH_*  = \overline{{\mathcal D}_{*}}^{{\mathcal L}^2} \subseteq {\mathcal K}_{*}$. Denote by $\mathcal{M}_{*} \subseteq \mathcal{K}_{*}$ the orthogonal complement of $\mathcal{H}_{*}$ within $\mathcal{K}_{*}$.  We will prove that $\mathcal{M}_{*}  = \{ (\bs{0}, \bs{0}) \}$. Take any element $ (\bs{v}_\Omega,\bs{v}_{\Gamma_O}) \in \mathcal{M}_{*}$. Then,
\begin{equation}
	(\bs{v}_\Omega ,\bs{w}_{\Omega})_{\Omega} + (\bs{v}_{\Gamma_O},\bs{w}_{\Gamma_{O}})_{\Gamma_O}=0, \quad \forall (\bs{w}_\Omega,\bs{w}_{\Gamma_O}) \in {\mathcal H}_{*} \, .
	\label{dens0} 
\end{equation}
In particular, since $\mathcal{C}^\infty_{0,\sigma}(\Omega) \subset {\mathcal H}_{*}$, \eqref{dens0} entails, by density,
$$
(\bs{v}_\Omega,\bs{w})_\Omega=0 \qquad \forall \bs{w}\in  \overline{\mathcal{C}^\infty_{0,\sigma}(\Omega)}^{L^{2}(\Omega)^{3}} \, ,
$$
so that the usual Leray decomposition of the space $L^{2}(\Omega)^{3}$ (see \cite[Chapter 1, Section 2]{ladyzhenskaya1969mathematical}) ensures the existence of $p\in H^1(\Omega)$ such that $\bs{v}_\Omega=\nabla p$ in $\Omega$. Since ${\mathcal D}_{*} \subset {\mathcal H}_{*}$, it then follows from \eqref{dens0} that
\begin{equation}\label{eq0i}
	(\nabla p,\bs{w})_{\Omega} + (\bs{v}_{\Gamma_O},\bs{w})_{\Gamma_O}=0 \quad \forall \bs{w} \in {\mathcal D}_{*}
	\ \iff \ \int_{\Gamma_O}p \bs{w} \cdot \bs{n} +\int_{\Gamma_O}\bs{v}_{\Gamma_O}\cdot\bs{w} = 0 \quad \forall \bs{w} \in {\mathcal D}_{*} \, ,
\end{equation}
and by density of ${\mathcal D}_{*}$ in ${\mathcal V}_{*}$, we get 
$$
\int_{\Gamma_O}p \bs{w} \cdot \bs{n} +\int_{\Gamma_O}\bs{v}_{\Gamma_O}\cdot\bs{w} = 0 \qquad \forall \bs{w} \in {\mathcal V}_{*} \, .
$$
Now, given $\bs{\psi} \in H^{1/2}_{00}(\Gamma_O)^3 \cap L^{2}_{\bs{n},0}(\Gamma_O)^{3}$, there exists $
\bs{w} \in {\mathcal V}_{*}$ such that $\gamma_{\Gamma_O}(\bs{w}) = \bs{\psi} $ (see, for example, \cite[Exercise III.3.5]{galdi2011introduction}). Thus,
\begin{equation}\label{eq0fa}
	\int_{\Gamma_O}p \bs{\psi} \cdot \bs{n} + \int_{\Gamma_O}\bs{v}_{\Gamma_O}\cdot\bs{\psi} = 0 \qquad \forall  \bs{\psi} \in H^{1/2}_{00}(\Gamma_O)^3 \cap L^{2}_{\bs{n},0}(\Gamma_O)^{3} \, ,
\end{equation}
which, by density of $H^{1/2}_{00}(\Gamma_O)$ in $L^{2}(\Gamma_O)$, implies 
\begin{equation}\label{eq0f}
	\int_{\Gamma_O}(p  \bs{n}  +  \bs{v}_{\Gamma_O} ) \cdot \bs{\psi} = 0 \qquad \forall \bs{\psi} \in L^2_{\bs{n},0}(\Gamma_O)^3 \, .
\end{equation}
Since $\bs{v}_\Omega=\nabla p$ in $\Omega$, from \eqref{decL1O} and \eqref{eq0f} we obtain a more complete characterization of $\bs{v}$:
$$
(\bs{v}_\Omega, \bs{v}_{\Gamma_{O}}) \in {\mathcal K}_*, \qquad \bs{v}_\Omega=\nabla p , \qquad \bs{v}_{\Gamma_O} + p \bs{n}=c \, \bs{n} \quad \text{(for some constant $c \in \R$)} \, .
$$
Since $\bs{v}_\Omega\cdot \bs{n}=\bs{v}_{\Gamma_O}\cdot \bs{n}$ on $\Gamma_{O}$, we deduce that $p \in H^1(\Omega)$ satisfies the following Neumann problem in $\Omega$:
$$
\left\{
\begin{aligned}
	&\Delta p = 0 & &\quad  \text{ in  } \Omega \, , \\[1pt]
	&\frac{\partial p}{\partial \bs{n} } = 0 & &\quad  \text{ on }  \Gamma_I \cup \Gamma_W \, ,  \\[1pt]
	&\frac{\partial p}{\partial \bs{n} }+ p  = c  & &\quad  \text{ on } \Gamma_O \, .
\end{aligned}
\right.
$$
From this system, we deduce 
$$
\| \nabla (p-c)\|^2_{L^2(\Omega)} + \| p-c \|^2_{L^2(\Gamma_O)} =0 \, ,
$$
so that $p  \equiv c$ in $\Omega$, and consequently, $(\bs{v}_\Omega, \bs{v}_{\Gamma_{O}})=(\nabla p,(c-p)\bs{n})=(\bs{0}, \bs{0})$. Since $\mathcal{H}_{*} \subseteq {\mathcal K}_{*}$ and its orthogonal complement within $\mathcal{K}_{*}$ is $\{ (\bs{0}, \bs{0}) \}$, we conclude that $\mathcal{H}_{*} = {\mathcal K}_{*}$.\hfill\qed

\subsection{Proof of Lemma \ref{leraydecomp}}\label{prooflemma00}
We prove at first $\mathcal{G}_{*}\subseteq \HH_*^{\perp}$. Let $\bs{\phi}\in \mathcal{G}_*$. By definition of $\mathcal{H}_{*}$ and $\mathcal{G}_{*}$, we have
$$
(\bs{v},\bs{\phi})_{\LL^2}=(\bs{v}_{\Omega},\nabla p)_\Omega+(\bs{v}_{\Gamma_O},-p\bs{n})_{\Gamma_O}=\int_{\Gamma_O} p (\bs{v}_{\Gamma_{O}} \cdot \bs{n})-\int_{\Gamma_O} p (\bs{v}_{\Gamma_O}\cdot\bs{n})=0 \qquad \forall  \bs{v}\in\HH_* \, ,
$$
and, therefore, $\bs{\phi}\in \HH_*^{\perp}$.

Now we prove that $\HH_*^{\perp} \subseteq \mathcal{G}_{*}$.  Given $\bs{\phi}\in\HH_*^\perp$, by definition, it holds $(\bs{\phi},\bs{v})_{\LL^2}=0$ for any $\bs{v}\in\mathcal{H}_*$, i.e.
\begin{equation}\label{eq0}
	(\bs{\phi}_\Omega,\bs{v}_\Omega)_{\Omega}+(\bs{\phi}_{\Gamma_O},\bs{v}_{\Gamma_O})_{\Gamma_O}=0 \qquad \forall \bs{v}\in\mathcal{H}_* \, .
\end{equation}
We have to show the existence of  $p\in H^1(\Omega)$ such that $\bs{\phi}_\Omega=\nabla p$ in $\Omega$ and  $\bs{\phi}_{\Gamma_O} = -p\bs{n}$  on   $\Gamma_{O}$. In particular, \eqref{eq0} holds for $\bs{v}\in\HH_*$ such that $\bs{v}_{\Gamma_O} \equiv \bs{0}$, that is,
$$
(\bs{\phi}_\Omega,\bs{v})_\Omega=0, \quad \forall \bs{v}\in L^2_\sigma(\Omega):=\left\{\bs{u}\in L^2(\Omega)^3 \mid \nabla \cdot \bs{u}=0 \ \ \mbox{in} \ \ \Omega \, , \ \bs{u}\cdot \bs{n} = 0 \ \ \mbox{on} \ \ \partial\Omega \right\} \, .
$$
Due to the classical Helmholtz decomposition of $L^2(\Omega)^3$, there exists $\mathfrak{p} \in H^1(\Omega)$ such that $\bs{\phi}_\Omega=\nabla \mathfrak{p}  $ in $\Omega$. Therefore, \eqref{eq0} becomes
\begin{equation}\label{eq00}
	(\nabla \mathfrak{p} ,\bs{v}_\Omega)_{\Omega}+(\bs{\phi}_{\Gamma_O},\bs{v}_{\Gamma_O})_{\Gamma_O}=0 \quad\forall\bs{v}\in\HH_*
	\ \iff\ \int_{\Gamma_O} \mathfrak{p} (\bs{v}_{\Gamma_O}\cdot \bs{n})+\int_{\Gamma_O}\bs{\phi}_{\Gamma_O}\cdot\bs{v}_{\Gamma_O}=0 \quad\forall \bs{v}\in\HH_* \, .
\end{equation}
Recalling \eqref{l2zero}, the next step is to show that 
\begin{equation} 	\label{psinu0}
\int_{\Gamma_O} \mathfrak{p}  (\bs{\psi} \cdot \bs{n})+\int_{\Gamma_O}\bs{\phi}_{\Gamma_O} \cdot\bs{\psi} = 0 \qquad \forall \bs{\psi} \in L^2_{\bs{n},0}(\Gamma_O)^3 \, .
\end{equation}
For $\bs{\psi} \in L^2_{\bs{n},0}(\Gamma_O)^3$, consider the unique weak solution $u \in H^1(\Omega)$ to the following the Neumann problem:
$$
\left\{
	\begin{aligned}
&\Delta u = 0 & &\quad  \text{ on } \Omega \, , \\[2pt]
&\frac{\partial u }{\partial \bs{n} }  = 0  & &\quad  \text{ on } \Gamma_{I}\cup\Gamma_{W} \, ,  \\[2pt]
&\frac{\partial u}{\partial \bs{n} } = \bs{\psi}  \cdot \bs{n}& &\quad  \text{ on }  \Gamma_O \, ,
\end{aligned}
\right.
$$
and take $\bs{v}_\Omega := \nabla u$, $\bs{v}_{\Gamma_O} := \bs{\psi} $ so that $(\bs{v}_\Omega , \bs{v}_{\Gamma_O} ) \in \HH_*$; \eqref{psinu0} then follows from \eqref{eq00}. Now recall the decomposition \eqref{decL1O} and observe  that  \eqref{psinu0} means $\mathfrak{p}  \bs{n} + \bs{\phi}_{\Gamma_O} \in \left( L^2_{\bs{n},0}(\Gamma_O)^3 \right) ^\perp \cong {\mathbb R} \bs{n}$. Therefore, there exists $c \in \mathbb R$ such that 
$$
\mathfrak{p}  \bs{n} + \bs{\phi}_{\Gamma_O} = c \, \bs{n} \iff \bs{\phi}_{\Gamma_O} = - (\mathfrak{p}  - c) \bs{n} .
$$
Defining $p \doteq \mathfrak{p} - c \in H^{1}(\Omega)$, we have $\bs{\phi}_{\Omega} = \nabla p $ in $\Omega$ and $\bs{\phi}_{\Gamma_O} = - p \bs{n}$ on $\Gamma_{O}$.\hfill\qed

\subsection{Proof of Proposition \ref{benes}}\label{proofprop}
In what follows, $C > 0$ shall denote a generic constant that depends exclusively on $\Omega$, that may change from line to line.

Firstly, let $(\bs{u}_{1},\pi_{1})\in \VV_*\times L^2(\Omega)$ be the unique weak solution to the following Stokes system in $\Omega$: 
\begin{equation}\label{stokesf1}
	\left\{
	\begin{array}{ll}
		-\Delta  \bs{u}_{1} + \nabla \pi_{1}=\bs{f}_\Omega \, , \quad  \nabla\cdot \bs{u}_{1}=0 \ \ &\mbox{ in } \  \Omega \, , \\[5pt]
		\bs{u}_{1}=\bs{0} \ \ &\mbox{ on } \  \Gamma_{I}\cup\Gamma_{W}  \, , \\[5pt]
		\dfrac{\partial  \bs{u}_{1}}{\partial \bs{n}} - \pi_{1} \bs{n}  =\bs{0}\ \ &\mbox{ on } \  \Gamma_{O} \,  .
	\end{array}
	\right.
\end{equation}
The results from \cite[Corollary A.3]{benes3d} ensure that $(\bs{u}_{1},\pi_{1})\in H^2(\Omega)^{3} \times H^1(\Omega)$, together with the estimate
\begin{equation}\label{stokesf1bound}
\|\bs{u}_{1}\|_{H^{2}(\Omega)} \leq C\|\bs{f}_\Omega\|_{L^{2}(\Omega)} \, .
\end{equation}
Applying the Brezis-Mironescu \cite[Theorem A]{brezis} interpolation inequality we find $\bs{u}_{1} \in H^{3/2}(\Omega)^{3}$, together with the bound
\begin{equation}\label{stokesf1bound2}
\|\bs{u}_{1}\|_{H^{3/2}(\Omega)} \leq C \|\bs{u}_{1}\|^{1/2}_{H^{1}(\Omega)} \|\bs{u}_{1}\|^{1/2}_{H^{2}(\Omega)} \leq C\|\bs{u}_{1}\|_{H^{2}(\Omega)} \, .
\end{equation} 
Inserting \eqref{stokesf1bound2} into the left-hand side of \eqref{stokesf1bound} gives us
\begin{equation}\label{stokesf1bound4}
	\|\bs{u}_{1}\|_{H^{3/2}(\Omega)} \leq C\|\bs{f}_\Omega\|_{L^{2}(\Omega)} \, .
\end{equation}

Secondly, let $(\bs{u}_{2},\pi_{2})\in \VV_*\times L^2(\Omega)$ be the unique weak solution to the following Stokes system in $\Omega$: 
\begin{equation}\label{stokesf2}
	\left\{
	\begin{array}{ll}
		-\Delta  \bs{u}_{2} + \nabla \pi_{2}=\bs{0} \, , \quad  \nabla\cdot \bs{u}_{2}=0 \ \ &\mbox{ in } \  \Omega \, , \\[5pt]
		\bs{u}_{2}=\bs{0} \ \ &\mbox{ on } \  \Gamma_{I}\cup\Gamma_{W}  \, , \\[5pt]
		\dfrac{\partial  \bs{u}_{2}}{\partial \bs{n}} - \pi_{2} \bs{n}  =\bs{f}_{\Gamma_{O}} \ \ &\mbox{ on } \  \Gamma_{O} \,  .
	\end{array}
	\right.
\end{equation}
Taking $\bs{u}_{2}$ as a test function in the weak formulation of \eqref{stokesf2}, and then representing the scalar pressure $\pi_{2}$ as the divergence of a suitable vector field in $\mathcal{V}_{*}$ (see \cite[Lemma 2.4]{fss}), we obtain the bound
\begin{equation} \label{essential0}
\| \nabla \bs{u}_{2} \|_{L^2(\Omega)}  +  \| \pi_{2} \|_{L^2(\Omega)} \leq C \| \bs{f}_{\Gamma_{O}} \|_{L^2( \Gamma_{O} )} \, .
\end{equation}
Additionally, from the trace inequality and  \eqref{essential0} we infer
\begin{equation} \label{essential00}
	\| \bs{u}_{2} \|_{L^2(\partial \Omega)}  \leq C \| \nabla \bs{u}_{2} \|_{L^2(\Omega)} \leq C \| \bs{f}_{\Gamma_{O}} \|_{L^2( \Gamma_{O} )} \, .
\end{equation}

Notice that $\Omega$ is a \textit{creased} Lipschitz domain, in the sense of \cite[Definition 2.2]{brown2010mixed}. We can therefore combine \cite[Theorem 6.3]{brown2010mixed} with \cite[Theorem 4.1]{brown2010mixed} in order to furnish the regularity estimate
\begin{equation} \label{essential1}
\| \nabla  \bs{u}_{2} \|^2_{L^2(\partial \Omega)} \leq C \left( \| \nabla \bs{u}_{2} \|^2_{L^2(\Omega)}  +  \| \pi_{2} \|^2_{L^2(\Omega)} + \| \bs{f}_{\Gamma_{O}} \|^2_{L^2( \Gamma_{O} )} \right) \leq C \| \bs{f}_{\Gamma_{O}} \|^{2}_{L^2( \Gamma_{O} )} \, ,
\end{equation}
where the second inequality in \eqref{essential1} follows from \eqref{essential0}. The estimate \eqref{essential1} implies that $\bs{u}_{2} \in H^{1}(\partial \Omega)^{3}$, so that an application of \cite[Theorem 4.15]{fabes1988dirichlet} (see \cite[Theorem 2.2]{brown1995estimates} for a more quantitative statement) ensures that $\bs{u}_{2} \in H^{3/2}(\Omega)^{3}$, together with the bound
\begin{equation} \label{essential2}
	\| \bs{u}_{2} \|^{2}_{H^{3/2}(\Omega)} \leq C 	\| \bs{u}_{2} \|^{2}_{H^{1}(\partial \Omega)} = C \left( \| \bs{u}_{2} \|^2_{L^2(\partial \Omega)} + \| \nabla  \bs{u}_{2} \|^2_{L^2(\partial \Omega)} \right) \leq C \| \bs{f}_{\Gamma_{O}} \|^{2}_{L^2( \Gamma_{O} )} \, ,
\end{equation}
where the last inequality in \eqref{essential2} follows from \eqref{essential00}-\eqref{essential1}.

Finally, setting $\bs{u} := \bs{u}_{1} + \bs{u}_{2}$ and $\pi := \pi_{1} + \pi_{2}$ in $\Omega$, \eqref{stokesf1} and \eqref{stokesf2} imply that $(\bs{u}, \pi) \in \VV_*\times L^2(\Omega)$ is the unique weak solution to \eqref{stokesf}. Owing to \eqref{stokesf1bound4} and \eqref{essential2}, we deduce that $\bs{u} \in H^{3/2}(\Omega)^{3}$ and the validity of the estimate \eqref{stima}.
\hfill\qed

\subsection{Proof of Theorem \ref{stokesoperatortheo}} \label{proofstokestheo}

We start by showing that
\begin{equation} \label{domainareg}
	\A\bs{v} = P_{\mathcal{H}_{*}}\left(- \Delta \bs{v}, \gamma_{\Gamma_O}\left( \dfrac{\partial  \bs{v}}{\partial \bs{n}}  \right) \right) \qquad \forall \bs{v} \in \mathcal{V}_{*} \cap H^{2}(\Omega)^{3} \, ,
\end{equation}
where $ P_{\mathcal{H}_{*}}$ is a generalization of the classical Leray projection; more precisely, in view of Lemma \ref{leraydecomp}, $P_{\mathcal{H}_{*}} : \mathcal{L}^{2} \longrightarrow \mathcal{H}_{*}$ is the orthogonal projection of $\mathcal{L}^{2}$ onto $\mathcal{H}_{*}$.

Take any vector field $\bs{v} \in \mathcal{V}_{*} \cap H^{2}(\Omega)^{3}$, so that $\A\bs{v} \in \mathcal{V}_{*}'$ (a priori). We certainly have
$$
\left(\Delta \bs{v}, \gamma_{\Gamma_O}\left( \dfrac{\partial  \bs{v}}{\partial \bs{n}} \right) \right) \in L^{2}(\Omega)^{3} \times H^{1/2}(\Gamma_{O})^{3} \subset \mathcal{L}^{2} \, .
$$
Then, suppose that
$$
P_{\mathcal{H}_{*}}\left(- \Delta \bs{v}, \gamma_{\Gamma_O}\left( \dfrac{\partial  \bs{v}}{\partial \bs{n}}  \right) \right) = \bs{f} \, ,
$$
for some $\bs{f}=(\bs{f}_\Omega, \bs{f}_{\Gamma_O}) \in \mathcal{H}_{*}$. By the usual characterization of orthogonal projections in Hilbert spaces, this means that
\begin{equation}\label{weak:stokesf3}
	\int_{\Omega} (- \Delta \bs{v} - \bs{f}_\Omega) \cdot \bs{\varphi}_\Omega + \int_{\Gamma_{O}} \left(\dfrac{\partial  \bs{v}}{\partial \bs{n}} - \bs{f}_{\Gamma_O}  \right) \cdot \bs{\varphi}_{\Gamma_O} = 0 \qquad \forall (\bs{\varphi}_\Omega, \bs{\varphi}_{\Gamma_O})  \in \mathcal{H}_{*} \, .
\end{equation}
Notice, in particular, that the identity \eqref{weak:stokesf3} holds for every $\bs{\varphi} \in \mathcal{V}_{*}$. In such case, after integrating by parts in $\Omega$, we realize that \eqref{weak:stokesf} holds for every $\bs{\varphi} \in \mathcal{V}_{*}$, meaning there exists a unique pressure $\pi_0 \in L^{2}(\Omega)$ such that \eqref{weak:stokesf2} is observed. This implies 
\begin{equation}\label{weak:stokesf4}
	\int_{\Omega} \nabla \bs{v} : \nabla \bs{\varphi} - \int_{\Omega} \pi_0 (\nabla \cdot \bs{\varphi}) = \int_{\Omega} \bs{f}_\Omega \cdot \bs{\varphi} \qquad \forall \bs{\varphi} \in \mathcal{U}_{*} \, ,
\end{equation}
and therefore, 
\begin{equation} \label{stokesomega}
	-\Delta  \bs{v} + \nabla \pi_0 = \bs{f}_\Omega \quad \text{ in distributional sense in } \Omega \, .
\end{equation}
Since $\bs{v} \in H^{2}(\Omega)^{3}$, \eqref{stokesomega} further implies that $\nabla \pi_0 \in L^{2}(\Omega)^3$, and then, $\pi_0 \in H^{1}(\Omega)$. In other words, \eqref{stokesomega} holds in strong sense in $\Omega$. Consequently, we can multiply \eqref{stokesomega} by a vector field $\bs{\varphi} \in {\mathcal H}_*$, and integrate by parts in $\Omega$ to obtain
$$
- \int_\Omega \Delta  \bs{v} \cdot \bs{\varphi}_\Omega + \int_{\Gamma_O} \pi_0 \bs{n} \cdot \bs{\varphi}_{\Gamma_O} - \int_\Omega \bs{f}_\Omega \cdot \bs{\varphi}_\Omega = 0 \, , 
$$
which, once subtracted from \eqref{weak:stokesf3}, yields
$$
\int_{\Gamma_{O}} \left(\dfrac{\partial  \bs{v}}{\partial \bs{n}} - \pi_0 \bs{n} - \bs{f}_{\Gamma_O}  \right) \cdot \bs{\varphi}_{\Gamma_O}  =0 \qquad \forall (\bs{\varphi}_\Omega, \bs{\varphi}_{\Gamma_O})  \in \mathcal{H}_{*} \, .
$$
Therefore, recalling \eqref{l2zero},
$$
\int_{\Gamma_{O}} \left(\dfrac{\partial  \bs{v}}{\partial \bs{n}} - \bs{f}_{\Gamma_O} - \pi_0 \bs{n}  \right) \cdot \bs{\varphi}_{\Gamma_O} = 0 \qquad \forall \bs{\varphi}_{\Gamma_O} \in L^2_{\bs{n},0}(\Gamma_O)^3 \, .
$$
Due to \eqref{decL1O}, there exists $c \in \mathbb R$ such that 
\begin{equation} \label{fboundary}
	\dfrac{\partial  \bs{v}}{\partial \bs{n}} - \bs{f}_{\Gamma_O} - \pi_0 \bs{n}  = c \bs{n} \quad \text{almost everywhere on } \ \Gamma_{O} \, .
\end{equation}
Hence, after setting $\pi := \pi_0 + c$ in $\Omega$, from \eqref{stokesomega}-\eqref{fboundary} we deduce that
\begin{equation}\label{proj}
	P_{\mathcal{H}_{*}}\left(- \Delta \bs{v}, \gamma_{\Gamma_O}\left( \dfrac{\partial  \bs{v}}{\partial \bs{n}}  \right) \right) = \bs{f}   =   \begin{cases}
		-\Delta\bs{v}+\nabla \pi \quad \mbox{in }\Omega \, , \\[6pt]
		\dfrac{\partial  \bs{v}}{\partial \bs{n}} - \pi \bs{n} \quad\hspace{4mm} \mbox{on }\Gamma_O \, .
	\end{cases}
\end{equation}
By definition and due to \eqref{proj}, for every $\bs{\varphi} \in \mathcal{V}_{*}$ we have
\begin{equation}
	\begin{aligned}
		\langle \A\bs{v}, \bs{\varphi} \rangle_{\mathcal{V}_{*}',\mathcal{V}_{*}} & = \big(\nabla \bs{v}, \nabla \bs{\varphi}\big)_\Omega = - \int_\Omega \Delta  \bs{v} \cdot \bs{\varphi} + \int_{\Gamma_{O}} \dfrac{\partial  \bs{v}}{\partial \bs{n}} \cdot \bs{\varphi} \\[6pt]
		& = \int_\Omega (-\Delta  \bs{v} + \nabla \pi) \cdot \bs{\varphi} + \int_{\Gamma_{O}} \left( \dfrac{\partial  \bs{v}}{\partial \bs{n}} - \pi \bs{n} \right) \cdot \bs{\varphi} = \left( P_{\mathcal{H}_{*}}\left(- \Delta \bs{v}, \gamma_{\Gamma_O}\left( \dfrac{\partial  \bs{v}}{\partial \bs{n}}  \right) \right) , \bs{\varphi} \right)_{\mathcal{L}^{2}} \\[6pt]
		& = \left\langle P_{\mathcal{H}_{*}}\left(- \Delta \bs{v}, \gamma_{\Gamma_O}\left( \dfrac{\partial  \bs{v}}{\partial \bs{n}}  \right) \right) , \bs{\varphi} \right\rangle_{\mathcal{V}_{*}',\mathcal{V}_{*}} \, ,
	\end{aligned}
\end{equation}
which proves \eqref{domainareg}.

Now, take any $\bs{f}=(\bs{f}_\Omega, \bs{f}_{\Gamma_O}) \in \mathcal{H}_{*}$. By definition, there exists a sequence $(\bs{v}_{k})_{k \in \mathbb{N}} \subset \mathcal{D}_{*}$ such that $\bs{v}_{k} \longrightarrow \bs{f}_\Omega$ strongly in $L^2(\Omega)^{3}$ and $\gamma_{\Gamma_{O}}(\bs{v}_{k}) \longrightarrow \bs{f}_{\Gamma_O}$ strongly in $L^2(\Gamma_{O})^{3}$, as $k \to \infty$. Since $\mathcal{D}_{*}$ is contained in $\mathcal{V}_{*} \cap H^{2}(\Omega)^{3}$, identity \eqref{domainareg} can be applied to deduce that $\mathcal{A}\bs{v}_{k} \in \mathcal{H}_{*}$ for every $k \in \mathbb{N}$. Therefore, $(\bs{v}_{k})_{k \in \mathbb{N}} \subset D(\mathcal{A})$, implying that $D(\mathcal{A})$ is dense in $\mathcal{H}_{*}$. Afterwards, exactly as in the proofs of \cite[Lemma IV.5.3]{boyer2012mathematical} and \cite[Lemma IV.5.4]{boyer2012mathematical}, we can show that $\mathcal{A}$ is a closed and self-adjoint operator, respectively.

We then turn to the task of proving \eqref{inclusionsa}. Firstly, take any $\bs{v} \in \mathcal{V}_{*} \cap H^{2}(\Omega)^{3}$, so that \eqref{domainareg} entails $\mathcal{A}\bs{v} \in \mathcal{H}_{*}$, and subsequently, $\bs{v} \in D(\mathcal{A})$. Thus, $\left( \mathcal{V}_{*} \cap H^{2}(\Omega)^{3} \right) \subset D(\mathcal{A})$.

Secondly, let $\bs{v} \in D(\mathcal{A})$, and put $\bs{f}:=\mathcal{A}\bs{v} \in \mathcal{H}_{*}$. Therefore,
\begin{equation} \label{secondinclusion}
\big(\nabla \bs{v}, \nabla \bs{\varphi}\big)_\Omega = \langle \mathcal{A}\bs{v}, \bs{\varphi} \rangle_{\mathcal{V}_{*}',\mathcal{V}_{*}} = \langle \bs{f}, \bs{\varphi} \rangle_{\mathcal{V}_{*}',\mathcal{V}_{*}} = \left( \bs{f} , \bs{\varphi} \right)_{\mathcal{L}^{2}} \qquad \forall \bs{\varphi} \in \mathcal{V}_{*} \, ,
\end{equation}
thereby implying the existence of a pressure $\pi \in L^{2}(\Omega)$ such that the pair $(\bs{v}, \pi)$ is a weak solution to \eqref{stokesf} in $\Omega$. Proposition \ref{benes} then ensures $\bs{v} \in H^{3/2}(\Omega)^{3}$, and consequently, $D(\mathcal{A}) \subset \left( \mathcal{V}_{*} \cap H^{3/2}(\Omega)^{3} \right)$. 

Let us now prove the estimate \eqref{estimatea}. Given $\bs{v} \in D(\mathcal{A})$, as before we put $\bs{f} := \mathcal{A}\bs{v} \in \mathcal{H}_{*}$, so that \eqref{secondinclusion} holds. By interpreting \eqref{secondinclusion} as the weak formulation of \eqref{stokesf}, through Proposition \ref{benes} we deduce 
$$
\|\bs{v}\|_{H^{3/2}(\Omega)}\le C \|\bs{f}\|_{{\mathcal L}^2} = C \left\| \mathcal{A}\bs{v}  \right\|_{{\mathcal L}^2} \, ,
$$
for some constant $C:=C(\Omega)>0$.

Concerning the spectral decomposition statement, the properties of the operator $\mathcal{A}$ proved so far allow us to invoke \cite[Theorem II.6.6]{boyer2012mathematical} and deduce the existence of a set $\{\bs{v}_k\}_{k \geq 1} \subset D(\mathcal{A}) \setminus \{\bs{0}\}$ forming a complete orthonormal system in $\mathcal{H}_{*}$, a complete orthogonal system in $D(\mathcal{A})$ (and therefore, also a complete orthogonal system in $\mathcal{V}_{*}$) such that
\begin{equation} \label{eigen1}
\mathcal{A}\bs{v}_k = \lambda_{k} \bs{v}_k \qquad \forall k \geq 1 \, ,
\end{equation}
for some sequence $\{\lambda_k\}_{k \geq 1} \subset \mathbb{R}$ such that $\{|\lambda_{k}|\}_{k \geq 1} \subset [0, +\infty)$ is increasing and divergent towards $+ \infty$ as $k \to +\infty$. As in \eqref{secondinclusion}, the identity \eqref{eigen1} implies the existence of a sequence $\{\pi_{k}\}_{k \geq 1} \subset L^{2}(\Omega)$ such that the pair $(\bs{v}_{k} , \pi_{k})$ is a weak solution to the system 
\begin{equation} \label{eigen2}
	\left\{
	\begin{array}{ll}
		-\Delta  \bs{v}_{k}+\nabla \pi_{k}=\lambda_{k}  \bs{v}_{k} \, , \quad  \nabla\cdot \bs{v}_{k}=0 \ \ &\mbox{ in } \  \Omega \, , \\[5pt]
		\bs{v}_{k}=\bs{0} \ \ &\mbox{ on } \  \Gamma_{I}\cup\Gamma_{W}  \, , \\[5pt]
		\dfrac{\partial  \bs{v}_{k}}{\partial \bs{n}} - \pi_{k} \bs{n}  =\lambda_{k}  \bs{v}_{k}\ \ &\mbox{ on } \  \Gamma_{O} \, ,
	\end{array}
	\right.
\end{equation}
for every $k \geq 1$. We invoke again \cite[Corollary A.3]{benes3d} to ensure $\{(\bs{v}_{k},\pi_{k})\}_{k \in \mathbb{N}} \subset H^2(\Omega)^{3} \times H^1(\Omega)$. Moreover, after multiplying \eqref{eigen2} by $\bs{v}_{k}$ and integrating by parts in $\Omega$, we find
$$
\| \nabla \bs{v}_{k} \|^{2}_{L^{2}(\Omega)} = \lambda_{k} \left( \| \bs{v}_{k} \|^{2}_{L^{2}(\Omega)} + \| \bs{v}_{k} \|^{2}_{L^{2}(\Gamma_{O})} \right) \qquad \forall  k \geq 1 \, ,
$$
thus showing $\lambda_{k} > 0$ for every $k \geq 1$. This concludes the proof. 
\hfill\qed

 \subsection{Proof of Lemma \ref{lemma0}}\label{prooflemma}
 In view of Definition \ref{addomain3}, let us introduce (in local coordinates) the sets
 $$
\Omega_{*} := \Omega_{0} \cup \Omega_{1} \cup \left\lbrace \bs{x} \in \Omega_{2} \ \Big\vert \  z < \dfrac{\ell_{2}}{3} \, \right\rbrace \, , \qquad \Gamma_{*} := \overline{\Omega} \cap \left\lbrace \bs{x} \in \mathbb{R}^{3} \ \Big\vert \  z = \dfrac{\ell_{2}}{3} \, \right\rbrace \, .
 $$
 Let $\bs{g}_{*}\in H^1(0,T;H_{00}^{1/2}(\Omega_{*})^3)$ (case $k=1$) and consider the following Stokes system in $\Omega_{*} $ (in terms of the parameter $t \in [0,T)$ of the data, via $\bs{g}_{*}$):
 \begin{equation}\label{stopoi2truncated}
 	\left\{
 	\begin{aligned}
 		& -\Delta \bs{W}_{\! 0} + \nabla \Pi_{0}  = \bs{0} \, , \quad  \nabla\cdot \bs{W}_{\! 0} = 0 \ \ \mbox{ in } \ \ \Omega_{*} \times (0,T) \, , \\[5pt]
 		& \bs{W}_{\! 0}=\bs{g}_{*}\ \ \hspace{3mm}\mbox{ on } \ \ \Gamma_{I}\times (0,T)  \, ,\\[5pt] &\bs{W}_{\! 0}=\bs{0} \ \ \hspace{5mm}\mbox{ on } \ \ (\Gamma_{W} \cap \partial \Omega_{*} )\times (0,T) \, , \\[5pt]
 		& \bs{W}_{\! 0}=\bs{V}_{\! \Phi_{*}} \ \ \mbox{ on } \ \ \Gamma_{*} \times (0,T) \, .
 	\end{aligned}
 	\right.
 \end{equation}
 The existence of a unique weak solution $(\bs{W}_{\! \! 0}, \Pi_{0}) \in H^1(0,T;H^{1}(\Omega_{*} )^3) \times H^1(0,T;L_{0}^{2}(\Omega_{*} ))$ to system \eqref{stopoi2truncated} is guaranteed by standard methods, owing to the compatibility condition coming from \eqref{flux03d}$_3$ and \eqref{gflux}; in particular, we have
 $$
 \int_{\Gamma_{I}} \bs{g}_{*}\cdot \bs{n} + \int_{\Gamma_{*} } \bs{V}_{\! \Phi_{*}} \cdot \bs{n} = -\Phi_{*}(t) + \Phi_{*}(t) = 0 \qquad \forall t\in[0,T) \, .
 $$
In order to justify the above result, consider, for each $t \in(0, T)$, the following stationary Stokes system (with non-homogeneous Dirichlet boundary conditions) in $\Omega_{*}$:
 \begin{equation}\label{stopoi22}
 	\left\{
 	\begin{aligned}
 		& -\Delta \bs{W}_{\! 0} + \nabla \Pi_{0}  = \bs{0} \, , \quad  \nabla\cdot \bs{W}_{\! 0} = 0 \ \ \mbox{ in } \ \ \Omega_{*} \, , \\[5pt]
 		& \bs{W}_{\! 0}=\bs{g}_{*}(t) \ \ \hspace{3mm}\mbox{ on } \ \ \Gamma_{I}  \, , \\[5pt]&\bs{W}_{\! 0}=\bs{0} \ \ \hspace{9mm}\mbox{ on } \ \ \Gamma_{W} \cap \partial \Omega_{*} \, , \\[5pt]
 		& \bs{W}_{\! 0}=\bs{V}_{\! \Phi_{*}}(t) \ \ \mbox{ on } \ \ \Gamma_{*} \, ,
 	\end{aligned}
 	\right.
 \end{equation}
 whose unique solution $(\bs{W}_{\! \! 0}(t), \Pi_{0}(t)) \in H^{1}(\Omega_{*} )^3 \times L_{0}^{2}(\Omega_{*})$ obeys the bound 
 \begin{equation}\label{fluxestimate00}
 		\| \bs{W}_{\! \! 0}(t) \|_{H^1(\Omega_{*} )} + \| \Pi_{0}(t) \|_{L^2(\Omega_{*} )} \leq C\left( \| \bs{g}_{*}(t) \|_{H^{1/2}(\Gamma_{I})} + \| \bs{V}_{\! \Phi_{*}}(t) \|_{H^{1/2}(\Gamma_{*} )} \right) \leq  C \| \bs{g}_{*}(t) \|_{H^{1/2}(\Gamma_{I})} \, ,
 \end{equation} 
for all $t\in(0,T)$, see, for example, \cite[Theorem IV.1.1]{galdi2011introduction}. In the remaining part of this proof, unless otherwise specified, $C:=C(\Omega_*)> 0$ denotes a generic constant, that may change from line to line.
 
  Due to the linearity of the problem \eqref{stopoi22}, for every $t_1,t_2 \in [0,T]$ we have
 \begin{equation}\label{flux estim}
 	\begin{split}
 		& \| \bs{W}_{\! \! 0}(t_1) - \bs{W}_{\! \! 0}(t_2) \|_{H^1(\Omega_{*} )} + \| \Pi_{0}(t_1) - \Pi_{0}(t_2) \|_{L^2(\Omega_{*} )} \\[6pt]
 		\leq & \, C \left( \| \bs{g}_{*}(t_1) -  \bs{g}_{*}(t_2) \|_{H^{1/2}(\Gamma_{I})} \!+\! \| \bs{V}_{\! \Phi_{*}}(t_1) - \bs{V}_{\! \Phi_{*}}(t_2) \|_{H^{1/2}(\Gamma_{*} )} \right)\\[6pt]
 		\leq & \,  C\| \bs{g}_{*}(t_1) - \bs{g}_{*}(t_2) \|_{H^{1/2}(\Gamma_{I})} \, .
 	\end{split}
 \end{equation} 
Since $\bs{g}_{*}\in H^1(0,T;H_{00}^{1/2}(\Gamma_I)^3) \subset \mathcal{C}^{0,1/2}([0,T];H_{00}^{1/2}(\Gamma_I)^3)$, there holds $\bs{W}_{\! \! 0} \in \mathcal{C}^{0,1/2}([0,T];H^{1}(\Omega_{*} )^3)$ and $\Pi_{0} \in \mathcal{C}^{0,1/2}([0,T];L^{2}(\Omega_{*} ))$, where $\mathcal{C}^{0,1/2}$ is the usual H\"older space.  Since $L^2(\Omega_{*} )$ and $H^1(\Omega_{*} )$ are reflexive spaces, for any $J \Subset [0,T]$ and any $h \in \mathbb R$  such that $|h| < \text{dist}(J,\{0,T\})$, we have, owing to \eqref{flux estim},
 \begin{equation}\label{fluxdtestim}
 	\begin{split}
 		& \left(  \int_J \| \bs{W}_{\! \! 0}(t + h) - \bs{W}_{\! \! 0}(t) \|^2_{H^1(\Omega_{*} )}dt \right)^{\frac 12} + \left(  \int_J \| \Pi_{0}(t+h) - \Pi_{0}(t) \|^2_{L^2(\Omega_{*} )} dt \right)^{\frac 12}  \\[6pt]
 		\leq & \, C \left[ \left(  \int_J  \| \bs{g}_{*}(t+h) -  \bs{g}_{*}(t) \|^2_{H^{1/2}(\Gamma_{I})} dt \right)^{\frac 12}  +  \left(  \int_J  \| \bs{V}_{\! \Phi_{*}}(t+h) - \bs{V}_{\! \Phi_{*}}(t) \|^2_{H^{1/2}(\Gamma_{*} )} dt \right)^{\frac 12}  \right] \\[6pt]
 		\leq & \,  C  \left( \int_J  \| \bs{g}_{*}(t+h) - \bs{g}_{*}(t) \|^2_{H^{1/2}(\Gamma_{I})} dt \right)^{\frac 12} \,
 		\\[6pt]
 		\leq & \,  C \| \partial_t \bs{g}_{*} \|_{L^{2}(0,T; H^{1/2}(\Gamma_{I})} |h| \, .
 	\end{split}
 \end{equation} 
 Therefore, by \cite[Proposition 2.2.26]{Papageorgiou2005}, we conclude that $\bs{W}_{\! \! 0} \in H^1(0,T;H^{1}(\Omega_{*} )^3) $ and  $\Pi_{0} \in L^2(0,T;L^{2}(\Omega_{*} )) $. Moreover, from the identity
 \begin{equation}\label{stopoi22dt}
 	\left\{
 	\begin{aligned}
 		& -\Delta (\partial_t \bs{W}_{\! 0}) + \nabla (\partial_t \Pi_{0})  = \bs{0} \, , \quad  \nabla\cdot ( \partial_t  \bs{W}_{\! 0} ) = 0 \ \ \mbox{ in } \ \ \Omega_{*} \, , \\[5pt]
 		& \partial_t  \bs{W}_{\! 0}=(\partial_t \bs{g}_{*})(t) \ \ \hspace{2mm}\mbox{ on } \ \ \Gamma_{I}  \, , \\[5pt] &\bs{W}_{\! 0}=\bs{0} \ \ \hspace{18mm}\mbox{ on } \ \ \Gamma_{W} \cap \partial \Omega_{*} \, , \\[5pt]
 		&\partial_t  \bs{W}_{\! 0}=(\partial_t  \bs{V}_{\! \Phi_{*}})(t) \ \ \mbox{ on } \ \ \Gamma_{*} \, ,
 	\end{aligned}
 	\right.
 \end{equation}
 and the estimate
 \begin{equation}\label{bo2}
 \| (\partial_t  \bs{W}_{\! \! 0})(t) \|_{H^1(\Omega_{*} )} + \| (\partial_t  \Pi_{0})(t) \|_{L^2(\Omega_{*} )}  \leq  C \| (\partial_t  \bs{g}_{*})(t) \|_{H^{1/2}(\Gamma_{I})} \, ,
  \end{equation}
 it follows that
 \begin{equation}\label{bo}
 \| \bs{W}_{\! \! 0}\|_{H^1(0,T;H^1(\Omega_{*} ))} + \| \Pi_{0}\|_{H^1(0,T;L^2(\Omega_{*} ))} \leq C \| \bs{g}_{*}\|_{H^1(0,T;H^{1/2}(\Gamma_{I}))} \, .
 \end{equation}
In the remaining part of the proof we get estimates in space for all $t\in(0,T)$ as \eqref{fluxestimate00}, and it is implied that estimates such as \eqref{bo2}-\eqref{bo}, involving time derivatives, can be inferred repeating the previous arguments.
 
We consider now the case $k=2$, i.e. we assume that $\bs{g}_{*}\in H^1(0,T;H^{3/2}_{00}(\Gamma_I)^3)$.
 Since the domains $\Omega_{1}$ and $\Omega_{2} \cap \Omega_{*} $ are cylinders (thus, convex), and $\Gamma_{W}$ is smooth, we merge the regularity results for the steady-state Stokes equations under non-homogeneous Dirichlet boundary conditions, see \cite[Teorema, page 311]{cattabriga1961problema}, with a localization argument through a partition of unity, e.g. \cite[Theorem A.1]{concaenglish}; hence, we obtain that $(\bs{W}_{\! \! 0}, \Pi_{0}) \in H^1(0,T;H^{2}(\Omega_{*} )^3) \times H^1(0,T; H^{1}(\Omega_{*} ))$ together with the estimate
 \begin{equation}\label{fluxestimate0}
 	\| \bs{W}_{\! \! 0}(t) \|_{H^{2}(\Omega_{*} )} +\| \Pi_{0}(t) \|_{H^{1}(\Omega_{*} )} \leq C \| \bs{g}_{*}(t) \|_{H^{3/2}(\Gamma_{I})} \,\quad \forall t\in(0,T) \, ,
 \end{equation} 
and by \cite[Proposition 2.2.26]{Papageorgiou2005}, 
$$
\| \bs{W}_{\! \! 0}\|_{H^1(0,T;H^2(\Omega_{*} ))} + \| \Pi_{0}\|_{H^1(0,T;H^1(\Omega_{*} ))} \leq C \| \bs{g}_{*}\|_{H^1(0,T;H^{3/2}(\Gamma_{I}))} \, .
$$
 
 Now, for each $t\in (0,T)$ we smoothly connect the vector field $\bs{W}_{\! 0}(t)$, defined in $\Omega_{*} $, with the velocity field $\bs{V}_{\! \Phi_{*}}(t)$, defined in $\Omega_{2}$. For this, we introduce the following subsets of $\Omega_{2}$:
 $$
 \Omega_{\sharp} := \left\lbrace \bs{x} \in \Omega_{2} \ \Big\vert \ 0 < z < \dfrac{\ell_{2}}{3} \, \right\rbrace \, , \qquad \Gamma_{\sharp} :=\overline{\Omega} \cap \left\lbrace \bs{x} \in \mathbb{R}^{3} \ \vert \  z = 0 \, \right\rbrace. 
 $$
 Accordingly, we take a cutoff function $\zeta \in \mathcal{C}^{\infty}_{0}(\mathbb{R}^3)$ such that 
 \begin{equation} \label{cutpro}
 	\zeta \equiv 1 \ \ \text{in a neighborhood of} \ \ \overline{\Omega_{0} \cup \Omega_{1}} \, ; \qquad \zeta \equiv 0 \ \ \text{in a neighborhood of} \ \ \overline{\Omega_{2} \setminus \overline{\Omega_{\sharp}}} \, ,
 \end{equation}
 see Figure \ref{cutoff}.
 \begin{figure}[H]
 	\begin{center}
 		\includegraphics[scale=0.7]{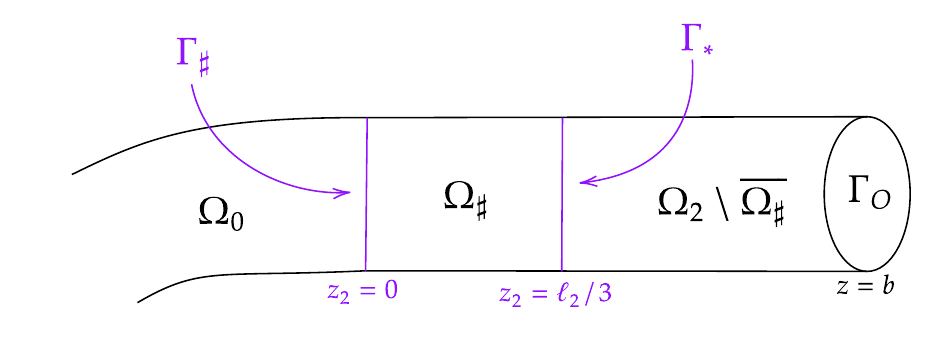}
 	\end{center}
 	\vspace*{-7mm}
 	\caption{The domain $\Omega_{\sharp} \subset \Omega$.}\label{cutoff}
 \end{figure}
 \noindent
 We then introduce the vector field $\bs{Z}_{\! 0} \in H^1(0,T;H^{k}(\Omega_{\sharp})^3)$, with $k\in\{1,2\}$, according to
 $$
 \bs{Z}_{\! 0} := \zeta \bs{W}_{\! 0} + (1-\zeta) \bs{V}_{\! \Phi_{*}} \quad \text{in} \ \ \Omega_{\sharp}\times(0,T) \, ,
 $$
which satisfies
 \begin{equation}\label{divz0}
 	\nabla \cdot \bs{Z}_{\! 0} = \nabla \zeta \cdot (\bs{W}_{\! 0} - \bs{V}_{\! \Phi_{*}}) \quad \text{in} \ \ \Omega_{\sharp}\times (0,T) \, ,
 \end{equation}
 and therefore, $\nabla \cdot \bs{Z}_{\! 0} \in H^1(0,T;H^{k-1}_0(\Omega_{\sharp}))$, due to \eqref{cutpro}. Moreover, from the Divergence Theorem and \eqref{flux03d}$_3$-\eqref{gflux}-\eqref{stopoi2truncated}, we have
 $$
 \int_{\Omega_{\sharp}} \nabla \cdot \bs{Z}_{\! 0} = \int_{\Gamma_{*}} \bs{Z}_{\! 0} \cdot \bs{n} + \int_{\Gamma_{\sharp}} \bs{Z}_{\! 0} \cdot \bs{n} = \int_{\Gamma_{*} } \bs{V}_{\! \Phi_{*}} \cdot \bs{n} + \int_{\Gamma_{\sharp}} \bs{W}_{\! 0} \cdot \bs{n} = \Phi_{*} - \Phi_{*} = 0 \qquad \forall t\in (0,T) \, .
 $$
 Since $\Omega_{\sharp}$ is a Lipschitz domain, for each $t\in(0,T)$, we can invoke \cite[Theorem III.3.3]{galdi2011introduction} to deduce the existence of another vector field $\bs{J}_{\! 0} \in H^1(0,T;H^k_0(\Omega_{\sharp})^3)$ such that $\nabla \cdot \bs{J}_{\! 0} = - \nabla \cdot \bs{Z}_{\! 0}$ in $\Omega_{\sharp}\times (0,T)$, and for all $t\in(0,T)$ and $C:=C(\Omega_{\sharp})>0$ we get
 \begin{equation}\label{fluxestimate1}
 	\| \bs{J}_{\! 0}(t) \|_{H^{k}(\Omega_{\sharp})} \leq C \| \nabla \cdot \bs{Z}_{\! 0}(t) \|_{H^{k-1}(\Omega_{\sharp})} \leq C \left( \| \bs{W}_{\! 0}(t) \|_{H^{k-1}(\Omega_{*} )} + \| \bs{V}_{\! \Phi_{*}}(t) \|_{H^{k-1}(\Omega_{2})} \right),
 \end{equation}
 where the last inequality in \eqref{fluxestimate1} follows from \eqref{divz0}. Again by \cite[Proposition 2.2.26]{Papageorgiou2005}, we obtain
 $$
 \| \bs{J}_{\! 0} \|_{H^1(0,T;H^{k}(\Omega_{\sharp}))} \leq C \| \nabla \cdot \bs{Z}_{\! 0} \|_{H^1(0,T;H^{k-1}(\Omega_{\sharp}))} \leq C \left( \| \bs{W}_{\! 0}\|_{H^1(0,T;H^{k-1}(\Omega_{*} ))} + \| \bs{V}_{\! \Phi_{*}}\|_{H^1(0,T;H^{k-1}(\Omega_{2}))} \right) \, .
 $$
 We are then in position to define the vector field $\bs{W}_{\! \! *} \in H^1(0,T;H^{k}(\Omega)^3)$ by
 \begin{equation}
 	\bs{W}_{\! \! *} :=\begin{cases}
 		\bs{W}_{\! 0} & \ \mathrm{ in } \ \ (\Omega_{0} \cup \Omega_{1})\times (0,T) \, ,  \\[4pt]
 		\bs{J}_{\! 0} + \zeta \bs{W}_{\! 0} + (1-\zeta) \bs{V}_{\! \Phi_{*}} & \ \mathrm{ in }  \ \ \Omega_{\sharp}\times (0,T) \, , \\[4pt]
 		\bs{V}_{\! \Phi_{*}} & \ \mathrm{ in } \ \ (\Omega_{2} \setminus \Omega_{\sharp})\times (0,T) \, ,
 	\end{cases}
 \end{equation}
 which, by the previous construction, complies with \eqref{stokesintrolemma}$_{1,2,3}$. Since $\bs{W}_{\! \! *} = \bs{V}_{\! \Phi_{*}}$ in $(\Omega_{2} \setminus \Omega_{\sharp})\times (0,T)$, there holds
 \begin{equation} \label{normder0}
 	\dfrac{\partial \bs{W}_{\! \! *}}{\partial \bs{n}} = \bs{0} \ \ \mbox{ on } \ \ \Gamma_{O}\times (0,T) \, .
 \end{equation}
 
It then remains to associate a scalar pressure to the vector field $\bs{W}_{\! \! *}$, allowing, in combination with \eqref{normder0}, for the equality \eqref{stokesintrolemma}$_4$ to be observed. Given $\sigma_{*} \in H^1(0,T;H^{k-3/2}(\Gamma_{O}))$, let $\widehat{\sigma} \in H^1(0,T; H^{1}(\Omega))$ be the unique weak solution to the following Zaremba problem in $\Omega$:
\begin{equation}\label{mixedbc}
	\left\{
	\begin{aligned}
		& \Delta \widehat{\sigma}=0 \ \ \mbox{ in } \ \ \Omega \times (0,T) \, , \\[5pt]
		& \widehat{\sigma}=0 \ \ \mbox{ on } \ \ (\Gamma_{I} \cup  \Gamma_{W}) \times (0,T) \, , \\[5pt]
		& \dfrac{\partial \widehat{\sigma}}{\partial \bs{n}} = - \sigma_{*} \ \ \mbox{ on } \ \ \Gamma_{O} \times (0,T) \, ,
	\end{aligned}
	\right.
\end{equation}
which satisfies the bound
\begin{equation}\label{mixedbc2}
\| \nabla \widehat{\sigma}(t) \|_{L^{2}(\Omega)} \leq C \| \sigma_{*}(t) \|_{H^{-1/2}(\Gamma_{O})} \qquad \forall t\in (0,T) \, .
\end{equation}
Recall that the outward unit normal to $\Gamma_{O}$, denoted here by $\bs{n}_{\Gamma_{O}} \in \mathbb{R}^{3}$, is constant. We then put $\Pi_{*}(t) := \nabla\widehat{\sigma}(t) \cdot \bs{n}_{\Gamma_{O}}$ for every $t \in (0,T)$, so that $\Pi_{*} \in H^1(0,T;L^{2}(\Omega))$ and, owing to \eqref{mixedbc2}, it verifies 
\begin{equation}\label{mixedbc3}
	\Pi_{*}= - \sigma_{*} \ \ \mbox{ on } \ \ \Gamma_{O} \times (0,T) \, ; \qquad
	\| \Pi_{*}(t) \|_{L^{2}(\Omega)} \leq C \| \sigma_{*}(t) \|_{H^{-1/2}(\Gamma_{O})} \qquad \forall t\in (0,T) \, ,
\end{equation}
closing the case $k=1$.

In the case when $k=2$ and $\sigma_{*} \in H^1(0,T;H^{1/2}(\Gamma_{O}))$, regularity results such as \cite[Corollary 8.3.2]{mazya2010elliptic} (see also \cite{ott2013mixed} for a relevant discussion) ensure that $\widehat{\sigma} \in H^1(0,T; H^{2}(\Omega))$, together with 
\begin{equation} \label{stimazaremba}
	\| \widehat{\sigma}(t) \|_{H^{2}(\Omega)} \leq C \| \sigma_{*}(t) \|_{H^{1/2}(\Gamma_{O})} \qquad \forall t\in (0,T) \, ,
\end{equation}
thereby implying that $\Pi_{*} \in H^1(0,T;H^{1}(\Omega))$ and
\begin{equation}\label{mixedbc4}
	\| \Pi_{*}(t) \|_{H^{1}(\Omega)} \leq C \| \sigma_{*}(t) \|_{H^{1/2}(\Gamma_{O})} \qquad \forall t\in (0,T) \, .
\end{equation}
Summarizing, there exists $\Pi_{*} \in H^1(0,T;H^{k-1}(\Omega))$ with $k\in\{1,2\}$ such that, in view of \eqref{mixedbc3}-\eqref{mixedbc4}, verifies
\begin{equation}\label{fluxestimate2}
\Pi_{*}= - \sigma_{*} \ \ \mbox{ on } \ \ \Gamma_{O} \times (0,T) \, ; \qquad	\| \Pi_{*}(t) \|_{H^{k-1}(\Omega)} \leq C \| \sigma_{*}(t) \|_{H^{k-\frac{3}{2}}(\Gamma_{O})} \qquad \forall t\in (0,T) \, .
\end{equation}

Now, from \eqref{poi23dd}-\eqref{flux03d}-\eqref{gflux} we observe that
 \begin{equation}\label{fluxestimate4}
 	\| \bs{V}_{\! \Phi_{*}} (t)\|_{H^{k}(\Omega_{2})} + \| \bs{V}_{\! \Phi_{*}}(t) \|_{H^{k-\frac{1}{2}}(\Gamma_{*} )} \leq C \Phi_{*}(t) \leq C \| \bs{g}_{*}(t) \|_{H^{k-\frac{1}{2}}(\Gamma_{I})} \,\quad \forall t\in (0,T) \,,
 \end{equation}
 and, as a consequence of \eqref{fluxestimate00}-\eqref{fluxestimate0}-\eqref{fluxestimate1}-\eqref{fluxestimate2}-\eqref{fluxestimate4}, we get the estimates 
 \begin{equation}\label{fluxestimate_}
 	\| \bs{W}_{\! \! *}(t) \|_{H^{k}(\Omega)} \leq C\| \bs{g}_{*}(t) \|_{H^{k-1/2}(\Gamma_{I})}, \qquad \hspace{17mm}\| \Pi_{*}(t) \|_{H^{k-1}(\Omega)} \leq C \| \sigma_{*}(t) \|_{H^{k-3/2}(\Gamma_{O})}  \,\quad \forall t\in(0,T) \, 
 \end{equation}
and
 \begin{equation}\label{2fluxestimate_}
	\| \bs{W}_{\! \! *} \|_{H^1(0,T;H^{k}(\Omega))} \leq C\| \bs{g}_{*} \|_{H^1(0,T;H^{k-1/2}(\Gamma_{I}))},  \qquad \| \Pi_{*} \|_{H^1(0,T;H^{k-1}(\Omega))} \leq C \| \sigma_{*} \|_{H^1(0,T;H^{k-3/2}(\Gamma_{O}))} \, .
\end{equation} 
This proves \eqref{fluxestimate} and concludes the proof.
\hfill\qed
 
\subsection{Proof of Theorem \ref{existence}}\label{proof}
	The proof of Theorem \ref{existence} is based on the Galerkin method and adapts techniques for the unsteady Navier-Stokes equations under classical boundary conditions, see e.g., \cite{Gal2000a,RRS2016,Tem1977b}. 
	
	\noindent
	\textit{Step 1: Existence of weak solutions.}
	Put $\bs{w}:=\bs{v}-\bs{W}_{\! \! *}$, $q:=p-\Pi_{*}$  and consider the auxiliary problem 
	\begin{equation}\label{ns_aux}
	\left\{
	\begin{array}{ll}
	\partial_t\bs{w}-\nu\Delta \bs{w}+[(\bs{w}+\bs{W}_{\! \! *})\cdot\nabla](\bs{w}+\bs{W}_{\! \! *})+\nabla q=-\partial_t\bs{W}_{\! \! *}+\nu \Delta\bs{W}_{\! \! *} -\nabla\Pi_{*}+\bs{f} \, ,    \ \ &\mbox{ in }  \ Q_T \, , \\[5pt]
	\nabla\cdot \bs{w}=0 \ \ &\mbox{ in } \  Q_T \, , \\[5pt]
	\bs{w}=\bs{0} \ \ &\hspace{-27mm}\mbox{ on } \  (\Gamma_{I}\cup \Gamma_W)\times (0,T) \, , \\[5pt]
		\bs{w}(\cdot,0)= \bs{v}_{0,\Omega}-\bs{W}_{\! \! *}(\cdot,0)|_\Omega:=\bs{w}_{0,\Omega} \ \ &\hspace{3mm}\mbox{ in } \  \Omega \,,\\[5pt]
	\partial_t\bs{w}+\nu \dfrac{\partial \bs{w}}{\partial \bs{n}} -q \bs{n} + \dfrac{1}{2} [(\bs{w}+\bs{W}_{\! \! *}) \cdot \bs{n}]^{-} \bs{w} =-\partial_t\bs{W}_{\! \! *}  \ \ &\hspace{-14mm}\mbox{ on } \  \Gamma_{O}\times (0,T) \, ,\\[5pt]
	\bs{w}(\cdot,0)=\bs{v}_{0,\Gamma_O}-\bs{W}_{\! \! *}(\cdot,0)|_{\Gamma_O}:=\bs{w}_{0,\Gamma_O} \ \ &\mbox{ on } \ \Gamma_O \, .
	\end{array}
	\right.
	\end{equation}
	By Lemma \ref{lemma0}, to prove the existence of a weak solution $\bs{v}$ spatially in $\VV$ is equivalent to show the existence of $\bs{w}$ spatially in $\VV_*$. Testing \eqref{ns_aux} by $\bs{\varphi} \in \mathcal{V}_{*}$ we find
	\begin{equation}\label{weak:evolutive}
	\begin{aligned}
	&\big(\partial_t\bs{w},\bs{\varphi}\big)_{\LL^2}+\nu \big(\nabla \bs{w}, \nabla \bs{\varphi}\big)_\Omega +\int_\Omega (\bs{w} \cdot \nabla)\bs{w}\cdot \bs{\varphi} + \int_\Omega(\bs{w} \cdot \nabla)\bs{W}_{\! \! *}\cdot  \bs{\varphi} + \int_\Omega(\bs{W}_{\! \! *} \cdot \nabla)\bs{w}\cdot \bs{\varphi} +\\&\dfrac{1}{2} \int_{\Gamma_{O}} [(\bs{w} + \bs{W}_{\! \! *}) \cdot \bs{n}]^{-}(\bs{w} \cdot \bs{\varphi}) 
	\!=\!-\big(\partial_t\bs{W}_{\! \! *}\,,\,\bs{\varphi}\big)_{\LL^2}+ \big\langle \nu \Delta\bs{W}_{\! \! *} -(\bs{W}_{\! \! *}\cdot \nabla)\bs{W}_{\! \! *}+\bs{f} \,,\,\bs{\varphi}\big\rangle,
	\end{aligned}
	\end{equation}
	with $\bs{w}_0 :=(\bs{w}_{0,\Omega},\bs{w}_{0,\Gamma_O})\in\mathcal{H}_*$ and $\langle\cdot,\cdot\rangle$ defined in \eqref{fomega}; we recall that, from Lemma \ref{lemma0} ($k=1$), the spatial regularity of $\bs{W}_{\! \! *}$ is $H^1(\Omega)^3$ so that $\Delta\bs{W}_{\! \! *}\in H^{-1}(\Omega)^3$ and $(\bs{W}_{\! \! *}\cdot \nabla)\bs{W}_{\! \! *}\in L^{3/2}(\Omega)^3$.

	We use, as basis in the Galerkin approximation, the eigenfunctions $\bs{u}_k$ satisfying \eqref{stokes} so that we can write the $n^{th}$-order approximation of $\bs{w}$ as
	\begin{equation}\label{approx}
	\bs{ w}^n(\bs{x},t):=\sum_{k=1}^n c^n_{k}(t)\,\bs{u}_k(\bs{x}) \qquad \forall (\bs{x},t) \in Q_{T} \, .
	\end{equation}
Let $\VV_*^n:=\text{span}\{\bs{u}_1,\bs{u}_2,\dots,\bs{u}_n\}$, so that $\bs{w}^n$ in \eqref{approx} satisfies \eqref{weak:evolutive}:
	\begin{equation}\label{ode0}
\begin{aligned}
&\big(\partial_t\bs{w}^n(t),\bs{\varphi}\big)_{\LL^2}\!+\!\nu \big(\nabla \bs{w}^n(t), \nabla \bs{\varphi}\big)_\Omega +\big( (\bs{w}^n(t)\! \cdot\! \nabla)\bs{w}^n(t), \bs{\varphi}\big)_\Omega\! +\! \big((\bs{w}^n(t)\! \cdot\! \nabla)\bs{W}_{\! \! *}(t), \bs{\varphi}\big)_\Omega+ \\& \big((\bs{W}_{\! \! *}(t) \cdot \nabla)\bs{w}^n(t), \bs{\varphi}\big)_\Omega + \dfrac{1}{2} \int_{\Gamma_{O}} [(\bs{w}^n(t) + \bs{W}_{\! \! *}(t)) \cdot \bs{n}]^{-}(\bs{w}^n(t)\cdot \bs{\varphi}) 
=\\&-\big(\partial_t\bs{W}_{\! \! *}(t),\bs{\varphi}\big)_{\LL^2} +\big\langle \nu \Delta\bs{W}_{\! \! *}(t) -(\bs{W}_{\! \! *}(t)\cdot \nabla)\bs{W}_{\! \! *}(t)+\bs{f}(t) \,,\,\bs{\varphi}\rangle\qquad \forall \bs{\varphi}\in\VV_*^n. 	
\end{aligned}
\end{equation}
Now we consider \eqref{ode0} with $\bs{\varphi}=\bs{u}_k\in\VV_*$, i.e. 
	\begin{equation}\label{ode}
	\begin{cases}
	\begin{aligned}
	&\big(\partial_t\bs{w}^n(t),\bs{u}_k\big)_{\LL^2}\!+\!\nu \big(\nabla \bs{w}^n(t), \nabla \bs{u}_k\big)_\Omega +\big( (\bs{w}^n(t)\! \cdot\! \nabla)\bs{w}^n(t), \bs{u}_k\big)_\Omega\! +\! \big((\bs{w}^n(t)\! \cdot\! \nabla)\bs{W}_{\! \! *}(t), \bs{u}_k\big)_\Omega+ \\& \big((\bs{W}_{\! \! *}(t) \cdot \nabla)\bs{w}^n(t), \bs{u}_k\big)_\Omega + \dfrac{1}{2} \int_{\Gamma_{O}} [(\bs{w}^n(t) + \bs{W}_{\! \! *}(t)) \cdot \bs{n}]^{-}(\bs{w}^n(t)\cdot \bs{u}_k) 
	=\\&-\big(\partial_t\bs{W}_{\! \! *}(t),\bs{u}_k\big)_{\LL^2} +\big\langle \nu \Delta\bs{W}_{\! \! *}(t) -(\bs{W}_{\! \! *}(t)\cdot \nabla)\bs{W}_{\! \! *}(t)+\bs{f}(t) \,,\,\bs{u}_k\rangle 	\quad\quad k=1,\dots,n,
	\end{aligned}\\
	\bs{w}^n(\cdot,0)=\bs{w}_0^n \, ,
	\end{cases}
	\end{equation}
	where  $\bs{w}_0^n:=\sum_{k=1}^n (\bs{w}_0,\bs{u}_k)\,\bs{u}_k$ is the projection
	(within $\HH_*$)  of $\bs{w}_0$ onto $\VV_*^n$. By the classical theory of ODE's systems, \eqref{ode} admits a unique (local) solution.

	Being $c^n_k(t)$ smooth coefficients, we multiply  \eqref{ode}$_1$ by $c^n_k(t)$
	and we sum over $k$, getting
	\begin{equation}\label{eq000}
	\begin{split}
	\dfrac{1}{2}	&\frac{d}{dt}\|\bs{w}^n(t)\|^2_{\LL^2}+\nu\|\nabla \bs{w}^n(t)\|^2_{L^2(\Omega)}\!=\!-\big((\bs{w}^n(t)\!\cdot\!\nabla) \bs{w}^n(t), \bs{w}^n(t)\big)_\Omega-\! \big((\bs{w}^n(t) \!\cdot\! \nabla)\bs{W}_{\! \! *}(t),  \bs{w}^n(t)\big)_\Omega\\& - \big((\bs{W}_{\! \! *}(t) \cdot \nabla)\bs{w}^n(t), \bs{w}^n(t) \big)_\Omega -\dfrac{1}{2} \int_{\Gamma_{O}} [(\bs{w}^n(t)+\bs{W}_{\! \! *}(t)) \cdot \bs{n}]^{-}|\bs{w}^n(t)|^2\\&-\big(\partial_t\bs{W}_{\! \! *}(t),\bs{w}^n(t)\big)_{\LL^2}+\big\langle\nu \Delta\bs{W}_{\! \! *}(t)-(\bs{W}_{\! \! *}(t)\cdot \nabla)\bs{W}_{\! \! *}(t)+\bs{f}(t)\,,\,\bs{w}^n(t)\big\rangle.
	\end{split}
	\end{equation}
	Integrating by parts it holds
	\begin{equation} \label{vnonlin}
	\begin{aligned}
	\big( (\bs{w}^{n} \cdot \nabla)\bs{w}^{n} , \bs{w}^{n}\big)_\Omega = \dfrac{1}{2} \int_{\Gamma_{O}} | \bs{w}^{n} |^{2}(\bs{w}^{n} \cdot \bs{n}) \qquad 	\big( (\bs{W}_{\! \! *} \cdot \nabla)\bs{w}^{n} , \bs{w}^{n}\big)_\Omega = \dfrac{1}{2} \int_{\Gamma_{O}} | \bs{w}^{n} |^{2}(\bs{W}_{\! \! *}\cdot \bs{n}),
	\end{aligned}
	\end{equation}
 then \eqref{eq000} becomes
	\begin{equation}\label{eq001}
	\begin{split}
	\dfrac{1}{2}&	\frac{d}{dt}\|\bs{w}^n(t)\|^2_{\LL^2}+\nu\|\nabla \bs{w}^n(t)\|^2_{L^2(\Omega)}=- \big((\bs{w}^n(t) \cdot \nabla)\bs{W}_{\! \! *}(t),  \bs{w}^n(t)\big)_\Omega\\&-\big(\partial_t\bs{W}_{\! \! *}(t)\,,\,\bs{w}^n(t)\big)_{\LL^2}+\big\langle \nu \Delta\bs{W}_{\! \! *}(t)-(\bs{W}_{\! \! *}(t)\cdot \nabla)\bs{W}_{\! \! *}(t)+\bs{f}(t)\,,\,\bs{w}^n(t)\big\rangle\\&-\dfrac{1}{2} \int_{\Gamma_{O}} \big\{(\bs{w}^n(t)+\bs{W}_{\! \! *}(t))\cdot \bs{n}+[(\bs{w}^n(t)+\bs{W}_{\! \! *}(t)) \cdot \bs{n}]^{-}\big\}|\bs{w}^n(t)|^2 \\[2mm]
	\leq& \big|\big((\bs{w}^n(t) \cdot \nabla)\bs{W}_{\! \! *}(t),  \bs{w}^n(t)\big)_\Omega\big|+\big|\big(\partial_t\bs{W}_{\! \! *}(t)\,,\,\bs{w}^n(t)\big)_{\LL^2}\big|+\\&+\big|\big\langle \nu \Delta\bs{W}_{\! \! *}(t)-(\bs{W}_{\! \! *}(t)\cdot \nabla)\bs{W}_{\! \! *}(t)+\bs{f}(t)\,,\,\bs{w}^n(t)\big\rangle\big|,
	\end{split}
	\end{equation}
where we used 
$$
\int_{\Gamma_{O}} \big\{(\bs{w}^n(t)+\bs{W}_{\! \! *}(t))\cdot \bs{n}+[(\bs{w}^n(t)+\bs{W}_{\! \! *}(t)) \cdot \bs{n}]^{-}\big\}|\bs{w}^n(t)|^2=\int_{\Gamma_{O}} [(\bs{w}^n(t)+\bs{W}_{\! \! *}(t)) \cdot \bs{n}]^{+}|\bs{w}^n(t)|^2 \geq 0.
$$

In order to bound the nonlinear term on the right-hand side of \eqref{eq001}, recall the Gagliardo-Nirenberg inequality: 
\begin{equation}
	\begin{split}
	\|\bs{v}\|_{L^4(\Omega)}\leq C
		\|\bs{v}\|^{1/4}_{L^2(\Omega)}\|\nabla \bs{v}\|^{3/4}_{L^2(\Omega)}
\qquad \forall \bs{v}\in  \VV_*,\quad (C:=C(\Omega)>0).
	\end{split}
\end{equation}
From Lemma \ref{lemma0} (with $k=1$) we know that $\bs{W}_{\! \! *}\in H^1(0,T;\VV)$, so that the Young inequality in the form $|ab|\leq \nu\frac{a^{4/3}}{4}+\frac{Cb^4}{\nu^3}$, $C\in\R^+$, entails
	\begin{equation}\label{eq002}
	\begin{split}
		\big|\big((\bs{w}^n \cdot \nabla)\bs{W}_{\! \! *},  \bs{w}^n\big)_\Omega\big|&\leq  \|\bs{w}^n\|^2_{L^4(\Omega)}\|\nabla\bs{W}_{\! \! *}\|_{L^2(\Omega)}\\&\leq \|\nabla\bs{W}_{\! \! *}\|_{L^2(\Omega)}
			\|\bs{w}^n\|^{1/2}_{L^2(\Omega)}\|\nabla \bs{w}^n\|^{3/2}_{L^2(\Omega)} \\&\leq \frac{\nu}{4}\|\nabla\bs{w}^n\|_{L^2(\Omega)}^2+\frac{C}{\nu^{3}}\|\nabla\bs{W}_{\! \! *}\|^{4}_{L^2(\Omega)}\|\bs{w}^n\|_{L^2(\Omega)}^2.
	\end{split}
\end{equation}
Using Schwartz and Young inequalities, we find	
		\begin{equation}
		\begin{split}
			\Big|\big(\partial_t\bs{W}_{\! \! *}\,,\,\bs{w}^n\big)_{\LL^2}\Big|\leq \frac{1}{2}\|\bs{w}^n\|^2_{\mathcal{L}^2}+\frac{1}{2}\|\partial_t\bs{W}_{\! \! *}\|_{\mathcal{L}^2}^2 \, ,
		\end{split}
	\end{equation}
as well as
	\begin{equation}\label{eq0002}
		\begin{split}
	&\big|\big\langle \nu \Delta\bs{W}_{\! \! *}-(\bs{W}_{\! \! *}\cdot \nabla)\bs{W}_{\! \! *}+\bs{f}\,,\,\bs{w}^n\big\rangle\big|\\\leq& C \big(\nu\|\Delta\bs{W}^*\|_{H^{-1}(\Omega)}\!+\!\|(\bs{W}_{\! \! *}\!\cdot\! \nabla)\bs{W}_{\! \! *}\|_{H^{-1}(\Omega)}\!+\!\|\bs{f}\|_{H^{-1}(\Omega)}\big)\|\nabla\bs{w}^n\|_{L^2(\Omega)}\\\leq& \frac{\nu}{4}\|\nabla\bs{w}^n\|_{L^2(\Omega)}^2+\frac{C}{\nu}\big(\nu^2\|\Delta\bs{W}^*\|^2_{H^{-1}(\Omega)}+\|(\bs{W}_{\! \! *}\cdot \nabla)\bs{W}_{\! \! *}\|^2_{H^{-1}(\Omega)}+\|\bs{f}\|^2_{H^{-1}(\Omega)}\big),
	\end{split}
	\end{equation}
	obtaining $C:=C(\Omega)>0$ such that 
	\begin{equation}\label{eq003}
	\begin{split}
	\frac{d}{dt}\|\bs{w}^n\|^2_{\mathcal{L}^2}+\nu\|\nabla \bs{w}^n\|^2_{L^2(\Omega)}\leq C
	&\bigg[\left(1+\dfrac{\|\nabla\bs{W}_{\! \! *}\|^{4}_{L^2(\Omega)}}{\nu^{3}}\right)\|\bs{w}^n\|^2_{\mathcal{L}^2}+\|\partial_t\bs{W}_{\! \! *}\|_{\mathcal{L}^2}^2\\&+\nu\|\Delta\bs{W}^*\|^2_{H^{-1}(\Omega)}+\tfrac{1}{\nu}\|(\bs{W}_{\! \! *}\cdot \nabla)\bs{W}_{\! \! *}\|^2_{H^{-1}(\Omega)}+\tfrac{1}{\nu}\|\bs{f}\|^2_{H^{-1}(\Omega)}\bigg].
	\end{split}
	\end{equation}
	Integrating on $(0,t)$ and applying the Gronwall Lemma, we infer the existence of a positive constant $\mathcal{M}:=\mathcal{M}(\Omega, \nu,T,\bs{w}_0,\bs{W}^*,\bs{f})$ such that, for all $t\in (0,T)$,
	\begin{equation}\label{eq004}
	\begin{split}
		\|\bs{w}^n(t)\|^2_{\mathcal{L}^2}\leq& 	\big[\|\bs{w}_0\|^2_{\mathcal{L}^2}+C\big(\|\partial_t\bs{W}_{\! \! *}\|_{L^2(0,T;\LL^2)}^2+\nu\|\Delta\bs{W}_{\! \! *}\|_{L^2(0,T;H^{-1}(\Omega))}^2\\&+\tfrac{1}{\nu}\|(\bs{W}_{\! \! *}\cdot \nabla)\bs{W}_{\! \! *}\|^2_{L^2(0,T;H^{-1}(\Omega))}+\tfrac{1}{\nu}\|\bs{f}\|_{L^2(0,T;H^{-1}(\Omega))}^2\big)\big]\cdot\\&\hspace{2mm}\cdot
		\exp\left[CT \left(1+\tfrac{1}{\nu^{3}}\|\nabla\bs{W}_{\! \! *}(t)\|^{4}_{L^\infty(0,T;L^2(\Omega))}\right)\right]:=\mathcal{M},
	\end{split}
	\end{equation}
	since $\|\bs{w}^n_0\|_{\LL^2}\leq \|\bs{w}_0\|_{\LL^2}$.
	This gives a \textit{uniform bound} for $\bs{w}^n$ in $L^\infty(0,T;\HH_*)$, from which we also infer the boundedness of $\bs{w}^n$ in $L^2(0,T;\VV_*)$, indeed
	\small{
		\begin{equation}\label{eq0044}
	\begin{split}
	\nu\int_0^T\|\nabla \bs{w}^n(t)&\|^2_{L^2(\Omega)}dt\leq\|\bs{w}_0\|^2_{\LL^2}+ C\big[T\mathcal{M}
	\big(1+\tfrac{1}{\nu^{3}}\|\nabla\bs{W}_{\! \! *}\|^{4}_{L^\infty(0,T;L^2(\Omega))}\big)+\|\partial_t\bs{W}_{\! \! *}\|_{L^2(0,T;\LL^2)}^2+\\&\nu\|\Delta\bs{W}_{\! \! *}\|_{L^2(0,T;H^{-1}(\Omega))}^2+\tfrac{1}{\nu}\|(\bs{W}_{\! \! *}\cdot \nabla)\bs{W}_{\! \! *}\|^2_{L^2(0,T;H^{-1}(\Omega))}+\tfrac{1}{\nu}\|\bs{f}\|_{L^2(0,T;H^{-1}(\Omega))}^2\big]:=M,
	\end{split}
	\end{equation}}\normalsize
with $M:= M(\Omega, \nu,T,\bs{w}_0,\bs{W}^*,\bs{f})$. We apply a diagonal argument to choose a subsequence of $\bs{w}^n$ such that, due to the compact embeddings $H^1(\Omega)\subset L^2(\Omega)$ and $H^{1/2}(\partial \Omega) \subset L^2(\partial\Omega)$, it verifies for any $T >0$
	\begin{equation}\label{conv}
	\bs{w}^n\rightarrow\bs{w}\quad \text{in }L^2(0,T;\mathcal{H}_*),\qquad 	\bs{w}^n\overset{\ast}{\rightharpoonup}\bs{w}\quad \text{in }L^\infty(0,T;\mathcal{H}_*),\qquad \nabla\bs{w}^n\rightharpoonup\nabla\bs{w}\quad \text{in }L^2(Q_T)^{3\times 3}.
\end{equation}
	This implies the existence of a weak solution $\bs{w}\in L^\infty(0,T;\HH_*)\cap L^2(0,T;\VV_*)$, and, in turn, this yields $ \bs{v}=\bs{w}+\bs{W}_{\! \! *}\in L^\infty(0,T;\HH)\cap L^2(0,T;\VV)$, being  $\bs{W}_{\! \! *}\in H^1(0,T; \VV)\subset \mathcal{C}([0,T];\VV)$.

We notice that  $$\left(\partial_t\bs{w}^n,\,\bs{\psi}\right)_{\LL^2}=\left(\partial_t\bs{w}^n,\,\mathbb{P}\bs{\psi}\right)_{\LL^2}\quad \forall\bs{\psi}\in \VV_*,$$
where $\mathbb{P}$ is the projection onto $\VV_*^n$; from the Gagliardo-Nirenberg inequality, we have 
\begin{equation}
	\begin{split}
	\big|\big( (\bs{w}^n\! \cdot\! \nabla)\bs{w}^n, \mathbb{P}\bs{\psi}\big)_\Omega\big|\leq&\|\bs{w}^n\|_{L^3(\Omega)}\|\nabla\bs{w}^n\|_{L^2(\Omega)}\|\mathbb{P}\bs{w}^n\|_{L^6(\Omega)}\\\leq& C\|\bs{w}^n\|^{1/2}_{L^2(\Omega)}\|\nabla\bs{w}^n\|^{3/2}_{L^2(\Omega)}\|\nabla\mathbb{P}\bs{\psi}\|_{L^2(\Omega)}\\=&C\|\bs{w}^n\|^{1/2}_{L^2(\Omega)}\|\nabla\bs{w}^n\|^{3/2}_{L^2(\Omega)}\|\mathbb{P}\bs{\psi}\|_{\VV_*} ,\\
	\bigg|\int_{\Gamma_{O}} [(\bs{w}^n(t) + \bs{W}_{\! \! *}) \cdot \bs{n}]^{-}(\bs{w}^n\cdot \mathbb{P}\bs{\psi}) \bigg|\leq& \|\bs{w}^n + \bs{W}_{\! \! *}\|_{L^4(\Gamma_O)}\|\bs{w}^n\|_{L^2(\Gamma_O)}\|\mathbb{P}\bs{\psi}\|_{L^4(\Gamma_O)}\\\leq& C\|\nabla\bs{w}^n + \nabla\bs{W}_{\! \! *}\|_{L^2(\Omega)}\|\bs{w}^n\|_{\LL^2}\|\mathbb{P}\bs{\psi}\|_{\VV_*}.
	\end{split}
\end{equation}
Now we consider \eqref{ode0} with $\bs{\varphi}=\mathbb{P}\bs{\psi}$; collecting these and previous estimates, we get
	\begin{equation}\label{odestime}
\begin{aligned}
\big|\big(\partial_t\bs{w}^n(t),\mathbb{P}\bs{\psi}\big)_{\LL^2}\big|\!\leq C \bigg( &\nu\|\nabla\bs{w}^n\|_{L^2(\Omega)}+\|\bs{w}^n\|^{1/2}_{L^2(\Omega)}\|\nabla\bs{w}^n\|^{3/2}_{L^2(\Omega)}+\\&\|\bs{w}^n\|^{1/2}_{L^2(\Omega)}\|\nabla\bs{w}^n\|^{1/2}_{L^2(\Omega)}\|\nabla\bs{W}_{\! \! *}\|_{L^2(\Omega)}+\|\bs{W}_{\! \! *}\|^{1/2}_{L^2}\|\nabla\bs{W}_{\! \! *}\|^{1/2}_{L^2(\Gamma_O)}\|\nabla\bs{w}^n\|_{L^2(\Omega)}+\\&\|\nabla\bs{w}^n + \nabla\bs{W}_{\! \! *}\|_{L^2(\Omega)}\|\bs{w}^n\|_{\LL^2}+
\|\partial_t\bs{W}_{\! \! *}\|_{\LL^2(\Omega)}+\\&\nu\|\Delta\bs{W}^*\|_{H^{-1}(\Omega)}+\|(\bs{W}_{\! \! *}\cdot \nabla)\bs{W}_{\! \! *}\|_{H^{-1}(\Omega)}+\|\bs{f}\|_{H^{-1}(\Omega)}\bigg)\|\mathbb{P}\bs{\psi}\|_{\VV_*}.
\end{aligned}
\end{equation}
Therefore, 
	$$
	\langle\partial_t\bs{w}^n,\bs{\psi}\rangle_{\VV_*',\VV_*}=\left(\partial_t\bs{w}^n,\bs{\psi}\right)_{\LL^2}=\left(\partial_t\bs{w}^n,\mathbb{P}\bs{\psi}\right)_{\LL^2},
	$$
	and, since $\|\mathbb{P}\bs{\psi}\|_{\VV_*}\leq \|\bs{\psi}\|_{\VV_*}$, from \eqref{odestime}
	we infer
		\begin{equation}\label{odestime2}
	\begin{aligned}
\|\partial_t\bs{w}^n\|_{\VV_*'}\leq C \bigg( &\nu\|\nabla\bs{w}^n\|_{L^2(\Omega)}+\|\bs{w}^n\|^{1/2}_{L^2(\Omega)}\|\nabla\bs{w}^n\|^{3/2}_{L^2(\Omega)}+\\&\|\bs{w}^n\|^{1/2}_{L^2(\Omega)}\|\nabla\bs{w}^n\|^{1/2}_{L^2(\Omega)}\|\nabla\bs{W}_{\! \! *}\|_{L^2(\Omega)}+\|\bs{W}_{\! \! *}\|^{1/2}_{L^2}\|\nabla\bs{W}_{\! \! *}\|^{1/2}_{L^2(\Gamma_O)}\|\nabla\bs{w}^n\|_{L^2(\Omega)}+\\&\|\nabla\bs{w}^n + \nabla\bs{W}_{\! \! *}\|_{L^2(\Omega)}\|\bs{w}^n\|_{\LL^2}+
	\|\partial_t\bs{W}_{\! \! *}\|_{\LL^2(\Omega)}+\\&\nu\|\Delta\bs{W}^*\|_{H^{-1}(\Omega)}+\|(\bs{W}_{\! \! *}\cdot \nabla)\bs{W}_{\! \! *}\|_{H^{-1}(\Omega)}+\|\bs{f}\|_{H^{-1}(\Omega)}\bigg).
	\end{aligned}
	\end{equation}
	Using the previous estimates on $\bs{w}^n$ we conclude that $\partial_t\bs{w}^n$ is bounded in $L^{4/3}(0,T;\VV'_*)$, so, up to a subsequence, we obtain $\partial_t\bs{w} \in L^{4/3}(0,T;\VV')$, and,  in turn, $\partial_t\bs{v} \in L^{4/3}(0,T;\VV')$.

	\medskip
	\noindent
	\textit{Step 2: Strong energy inequality.}
	To prove \eqref{sei}, we observe that $\bs{w}^n(t)\rightarrow\bs{w}(t)$ strongly in $\mathcal{L}^2$ for almost all $t\in[0,T)$, so in particular 
	\begin{equation}\label{L2_strong}
	\|\bs{w}^n(t)\|_{\LL^2}^2\rightarrow\|\bs{w}(t)\|_{\LL^2}^2.
	\end{equation}
	Integrating the equation in \eqref{eq001} over $(s,t)$ we find
	\small{
		\begin{equation}\label{eq008}
		\begin{split}
			&\dfrac{\|\bs{w}^n(t)\|^2_{\LL^2}}{2}+\nu\int_s^t\|\nabla \bs{w}^n(\tau)\|^2_{L^2(\Omega)}d\tau\leq \dfrac{\|\bs{w}^n(s)\|^2_{\LL^2}}{2}- \int_s^t\big((\bs{w}^n(\tau) \cdot \nabla)\bs{W}_{\! \! *},  \bs{w}^n(\tau)\big)_\Omega d\tau\\&-\int_s^t\big(\partial_t\bs{W}_{\! \! *}(\tau)\,,\,\bs{w}^n(\tau)\big)_{\LL^2} d\tau+\int_s^t\big\langle\nu \Delta\bs{W}_{\! \! *}(\tau)-(\bs{W}_{\! \! *}(\tau)\cdot \nabla)\bs{W}_{\! \! *}(\tau)+\bs{f}(\tau)\,,\,\bs{w}^n(\tau)\big\rangle d\tau.
		\end{split}
	\end{equation}}
\normalsize	We check the convergence of the right-hand side terms; firstly, we observe that 
	\begin{equation}\label{ineq0}
	\begin{split}
	&\big((\bs{w}^n \cdot \nabla)\bs{W}_{\! \! *},  \bs{w}^n\big)_\Omega -\big((\bs{w} \cdot \nabla)\bs{W}_{\! \! *},  \bs{w}\big)_\Omega=\big((\bs{w}^n \cdot \nabla)\bs{W}_{\! \! *},  \bs{w}^n-\bs{w}\big)_\Omega -\big([(\bs{w}-\bs{w}^n) \cdot \nabla]\bs{W}_{\! \! *},  \bs{w}\big)_\Omega;
		\end{split}
\end{equation}
hence, as $n\rightarrow+\infty$ applying \eqref{conv}, H\"older, Sobolev and interpolation inequalities, we find  $C:=C\big(\Omega\big)>0$ such that
\small{
	\begin{equation}\label{ineq2}
		\begin{split}
				\bigg|&\int_s^t\big((\bs{w}^n(\tau) \cdot \nabla)\bs{W}_{\! \! *},  \bs{w}^n(\tau)-\bs{w}(\tau)\big)_\Omega d\tau\bigg|\\&\leq 
				\int_s^t\|\bs{w}^n(\tau)\|_{L^6(\Omega)} \|\nabla\bs{W}_{\! \! *}(\tau)\|_{L^2(\Omega)}  \|\bs{w}^n(\tau)-\bs{w}(\tau)\|_{L^3(\Omega)} d\tau\\&\leq  C\|\nabla\bs{W}_{\! \! *}\|_{L^\infty(0,T;L^2(\Omega))}\int_s^t\left(\|\nabla\bs{w}^n(\tau)\|_{L^2(\Omega)}\|\nabla\bs{w}^n(\tau)-\nabla\bs{w}(\tau)\|^{1/2}_{L^2(\Omega)}\|\bs{w}^n(\tau)-\bs{w}(\tau)\|^{1/2}_{L^2(\Omega)}\right)d\tau\\&\leq  C\|\nabla\bs{W}_{\! \! *}\|_{L^\infty(0,T;L^2(\Omega))}\|\nabla\bs{w}^n\|_{L^2(0,T;L^2(\Omega))}\|\nabla\bs{w}^n-\nabla\bs{w}\|^{1/2}_{L^2(0,T;L^2(\Omega))}\|\bs{w}^n-\bs{w}\|^{1/2}_{L^2(0,T;L^2(\Omega))}\rightarrow 0,\\
				\bigg|&\int_s^t\big([(\bs{w}(\tau)-\bs{w}^n(\tau)) \cdot \nabla]\bs{W}_{\! \! *},  \bs{w}(\tau)\big)_\Omega d\tau\bigg|\\&\leq  \int_s^t\|\bs{w}^n(\tau)-\bs{w}(\tau)\|_{L^3(\Omega)}\|\nabla\bs{W}_{\! \! *}(\tau)\|_{L^2(\Omega)}\|\bs{w}(\tau)\|_{L^6(\Omega)}d\tau\\&\leq  C\|\nabla\bs{W}_{\! \! *}\|_{L^\infty(0,T;L^2(\Omega))}\|\nabla\bs{w}\|_{L^2(0,T;L^2(\Omega))}\|\nabla\bs{w}^n-\nabla\bs{w}\|^{1/2}_{L^2(0,T;L^2(\Omega))}\|\bs{w}^n-\bs{w}\|^{1/2}_{L^2(0,T;L^2(\Omega))}\rightarrow 0.
		\end{split}
	\end{equation}}
\normalsize
	The linear terms on the right-hand side pass to the limit due to \eqref{conv}; for instance, we have
		\begin{equation}\label{ineq}
		\begin{split}
			&\bigg|\int_s^t\big(\partial_t\bs{W}_{\! \! *}(\tau)\,,\,\bs{w}^n(\tau)-\bs{w}(\tau)\big)_{\LL^2}\bigg|\leq \|\partial_t\bs{W}_{\! \! *}\|_{L^2(0,T;\mathcal{L}^2)}\|\bs{w}^n-\bs{w}\|_{L^2(0,T;\mathcal{L}^2)}\rightarrow 0.
		\end{split}
	\end{equation}
	Therefore, we pass to the limit on both sides of \eqref{eq008} and, due to \eqref{conv}, \eqref{L2_strong} we obtain \eqref{sei}.
Since $\bs{w}^n(0)$  is the projection of $\bs{w}_0$, it holds $\|\bs{w}^n(0)\|_{\LL^2}^2\rightarrow\|\bs{w}(0)\|_{\LL^2}^2$, and \eqref{sei} is true also in $s=0$. \hfill\qed

\subsection{Recovery of the pressure for weak solutions} \label{recpressuresection}
Now, we consider the pressure recovery, starting from the weak formulation:
\begin{equation}\label{odep}
\begin{cases}
\begin{aligned}
\displaystyle  	\frac{ d }{dt}& \big(\bs{v} (t),\bs{\varphi}\big)_{\LL^2} + \nu \big(\nabla \bs{v}(t), \nabla \bs{\varphi} \big)_\Omega  +\big( (\bs{v}(t)\! \cdot\! \nabla)\bs{v}(t), \bs{\varphi} \big)_\Omega\! \\
&  \displaystyle  + \dfrac{1}{2} \int_{\Gamma_{O}} [\bs{v}(t)  \cdot \bs{n}]^{-}(\bs{w}(t)\cdot \bs{\varphi}) 
= \big\langle \bs{f}(t) \,,\,\bs{\varphi} \rangle \, + \langle\sigma_{*}(t)\bs{n}\,,\, \bs{\varphi} \rangle_{H^{-1/2}(\Gamma_O), \, H_{00}^{1/2}(\Gamma_O)},
\end{aligned} \\
\displaystyle	 \big(\bs{v} (0),\bs{\varphi}\big)_{\LL^2} =  \big(\bs{v}_0,\bs{\varphi}\big)_{\LL^2}\,  
\end{cases} \forall \bs{\varphi} \in \mathcal{V}_{*} \, .
\end{equation}
Recalling the space in \eqref{U}, the adjoint of the divergence operator $\text{div}_{\Omega} : \mathcal{U}_{*} \longrightarrow L^{2}(\Omega)$ is defined by the map $\text{div}^{*}_{\Omega} : L^{2}(\Omega) \longrightarrow \mathcal{U}'_*$ such that
\begin{equation}
\langle \text{div}^{*}_{\Omega}(q) , \bs{\varphi} \rangle_{*} := \int_{\Omega} q (\nabla \cdot \bs{\varphi}) \qquad \forall q \in L^{2}(\Omega) \, , \quad \forall \bs{\varphi} \in \mathcal{U}_{*} \, .
\label{div*def}
\end{equation}
Therefore, by the Closed Range Theorem of Banach, we have
\begin{equation}
\text{Range}(\text{div}^{*}_{\Omega}) = [ \text{Ker}(\text{div}_{\Omega})]^{\circ} = (\mathcal{V}_{*})^{\circ} := \{ \bs{F} \in
\mathcal{U}'_* \ | \ \langle  \bs{F}, \bs{\varphi} \rangle_{*} = 0 \quad \forall \bs{\varphi} \in \mathcal{V}_{*}\, \} \, .
\label{relU}
\end{equation}

Given $\bs{v}  \in {\mathcal V}$ and $\bs{w}  \in {\mathcal V}_*$, the applications
$$
\begin{aligned}
& \bs{\varphi} \in \mathcal{U}_{*} \longmapsto \big(\bs{v},\bs{\varphi}\big)_{\LL^2} \, , \\[6pt]
& \bs{\varphi} \in \mathcal{U}_{*} \longmapsto \int_{\Omega} \nabla \bs{v} : \nabla \bs{\varphi}  = : \langle {\mathcal G}( \bs{v}) ,\bs{\varphi}  \rangle_{*} \, , \\[6pt]
& \bs{\varphi} \in \mathcal{U}_{*} \longmapsto \int_{\Omega} (\bs{v} \cdot \nabla) \bs{v} \cdot \bs{\varphi}  + \dfrac{1}{2} \int_{\Gamma_{O}} [\bs{v} \cdot \bs{n}]^{-} \bs{w} \cdot \bs{\varphi} = : \langle {\mathcal B}( \bs{v} , \bs{w} ) , \bs{\varphi} \rangle_{*}\, , \\[6pt]
& \bs{\varphi} \in \mathcal{U}_{*} \longmapsto \big\langle \bs{f} \,,\,\bs{\varphi} \rangle + \langle\sigma_{*}\bs{n}\,,\, \bs{\varphi} \rangle_{H^{-1/2}(\Gamma_O), \, H_{00}^{1/2}(\Gamma_O)} =: \langle \mathcal F , \bs{\varphi} \rangle_{*} \, ,
\end{aligned}
$$
clearly define linear continuous functionals on $\mathcal{U}_{*}$. Then, \eqref{odep} may be written as
$$
\big(\bs{v},\bs{\varphi}\big)_{\LL^2} - \big(\bs{v}_0,\bs{\varphi}\big)_{\LL^2} + \nu \, \int_0^t \langle {\mathcal G}\left( \bs{v}\right) ,\bs{\varphi}  \rangle_{*} + \int_0^t  \langle {\mathcal B}\left( \bs{v} , \bs{w} \right) , \bs{\varphi} \rangle_{*}  - \int_0^t \langle {\mathcal F} , \bs{\varphi} \rangle_{*}  = 0\, ,\quad \forall \bs{\varphi} \in \mathcal{V}_{*} \, ,
$$
for almost any $t \in (0,T)$. Let 
$$
\langle  \bs{F} (t), \bs{\varphi}  \rangle_{*}  := \big(\bs{v} (t),\bs{\varphi}\big)_{\LL^2} - \big(\bs{v}_0,\bs{\varphi}\big)_{\LL^2} + \nu \, \int_0^t \langle {\mathcal G}\left( \bs{v}(t)\right) ,\bs{\varphi}  \rangle_{*}  + \int_0^t  \langle {\mathcal B}\left( \bs{v}(t) , \bs{w}(t) \right) , \bs{\varphi} \rangle_{*}  - \int_0^t \langle {\mathcal F}(t) , \bs{\varphi} \rangle_{*} \, ,
$$
for $\bs{\varphi} \in \mathcal{U}_{*}$. Then 
$$
\langle  \bs{F} (t), \bs{\varphi}  \rangle_{*}  = 0,  \quad \forall \bs{\varphi}  \in \mathcal{V}_{*}, \quad \text{ for a.a. } t \in (0,T), 
$$
and then, by \eqref{relU},  we have $\bs{F} (t) \in \text{Range}(\text{div}^{*}_{\Omega})$ for a.a. $t \in (0,T)$. Consequently, from \eqref{div*def}, there exists $p(t) \in L^2(\Omega)$ such that 
$$
\langle  \bs{F} (t), \bs{\varphi}  \rangle_{*}  = \int_{\Omega} p(t) (\nabla \cdot \bs{\varphi}),  \quad \forall \bs{\varphi}  \in \mathcal{U}_{*}, \quad \text{ for a.a. } t \in (0,T). 
$$
To characterize more precisely the time dependence of $p$, we follow the approach of \cite{Neustupa2020}. Note that ${\mathcal G}( \bs{v})  \in L^2(0,T;{\mathcal U}'_*)$, ${\mathcal B}( \bs{v},\bs{w})  \in L^{4/3}(0,T;{\mathcal U}'_*)$ and ${\mathcal F} \in L^2(0,T;{\mathcal U}'_*)$. 

Consider the decompositions 
$$
\mathcal{U}_{*} = \mathcal{V}_{*} \oplus \mathcal{V}_{*}^\perp \cong \text{Ker}(\text{div}_{\Omega})  \oplus \text{Range}(\text{div}^*_{\Omega}) \, ,
$$
 the projection $Q_{\mathcal{V}_{*}^\perp}$ from $\mathcal{U}_{*}$ onto $\mathcal{V}_{*}^\perp$ and its adjoint projection $Q^*_{\mathcal{V}_{*}^\perp}$ in $\mathcal{U}'_*$, which, in virtue of the Closed Range Theorem and \eqref{relU}, satisfy
$$
\text{Range}(Q^*_{\mathcal{V}_{*}^\perp})  = [\text{Ker}(Q_{\mathcal{V}_{*}^\perp})]^\circ  = 
(\mathcal{V}_{*})^\circ  = [\text{Ker}(\text{div}_{\Omega})]^\circ = \text{Range}(\text{div}^{*}_{\Omega}) .
$$
Since $\bs{F} (t) \in \text{Range}(\text{div}^{*}_{\Omega})$ for a.a. $t \in (0,T)$, we have 
$$
\langle  \bs{F} (t), (I_{\mathcal{U}_{*}} - Q_{\mathcal{V}_{*}^\perp} ) \bs{\varphi}  \rangle_{*}  = 0,  \quad \forall \bs{\varphi}  \in \mathcal{U}_{*}, \quad \text{ for a.a. } t \in (0,T), 
$$
$$
\iff \langle  \bs{F} (t),\bs{\varphi}  \rangle_{*}  = \langle  \bs{F} (t), Q_{\mathcal{V}_{*}^\perp} \bs{\varphi}  \rangle_{*} ,  \quad \forall \bs{\varphi}  \in \mathcal{U}_{*}, \quad \text{ for a.a. } t \in (0,T), 
$$
$$
\iff \langle  \bs{F} (t),\bs{\varphi}  \rangle_{*}  = \langle  Q^*_{\mathcal{V}_{*}^\perp}  \bs{F} (t), \bs{\varphi}  \rangle_{*} ,  \quad \forall \bs{\varphi}  \in \mathcal{U}_{*}, \quad \text{ for a.a. } t \in (0,T), 
$$
$$
\implies \langle  \bs{F}' ,\bs{\varphi}  \rangle_{*}  = \langle  Q^*_{\mathcal{V}_{*}^\perp}  \bs{F}', \bs{\varphi}  \rangle_{*} ,  \quad \forall \bs{\varphi}  \in \mathcal{U}_{*}, \quad \text{ in } \mathcal{C}_0^\infty (0,T)', 
$$
where $\bs{F}'$  denotes the distributional derivative with respect to $t$.  Since $Q^*_{\mathcal{V}_{*}^\perp}  \bs{F}(t) \in \text{Range}(Q^*_{\mathcal{V}_{*}^\perp})  = \text{Range}(\text{div}^{*}_{\Omega})$,  for a.a. $t \in (0,T)$, we conclude the following: 
$$
\langle  Q^*_{\mathcal{V}_{*}^\perp}  \bs{F}'(t), \bs{\varphi}  \rangle_{*} = J'_1 + J_2 + J_3 + J_4 \, ,
$$
where
\begin{equation}
\begin{split}
J_1 (t) &= \big((Q_{\mathcal{V}_{*}^\perp}  \bs{v} )(t),\bs{\varphi} \big)_{\LL^2} = \int_{\Omega} p_0 (t)(\nabla \cdot \bs{\varphi}) , \qquad\hspace{9.5mm} p_0 \in L^\infty(0,T;L^2(\Omega))\\
J_2 (t)&= \nu \, \langle Q^*_{\mathcal{V}_{*}^\perp}  {\mathcal G}( \bs{v}(t)) ,\bs{\varphi}  \rangle_{*}  = \int_{\Omega} p_1 (t)(\nabla \cdot \bs{\varphi}), \qquad \hspace{7mm}p_1 \in L^2(0,T;L^2(\Omega))\\
J_3 (t)&= \langle Q^*_{\mathcal{V}_{*}^\perp}  {\mathcal B}( \bs{v} (t), \bs{w}(t) ) , \bs{\varphi}  \rangle_{*}  = \int_{\Omega} p_2 (t)(\nabla \cdot \bs{\varphi}), \qquad p_2 \in L^{4/3}(0,T;L^2(\Omega))\\
J_4 (t)&= \langle Q^*_{\mathcal{V}_{*}^\perp}  {\mathcal F} (t) , \bs{\varphi}  \rangle_{*}  = \int_{\Omega} p_3 (t) (\nabla \cdot \bs{\varphi}) , \qquad \hspace{14.5mm}p_3 \in L^2(0,T;L^2(\Omega)).
\end{split}
\end{equation}
We obtain
\begin{equation}\label{odep2}
\begin{cases}
\begin{aligned}
\displaystyle \frac{ d }{dt}  \Big[ \big(\bs{v} (t),\bs{\varphi}\big)_{\LL^2}  &- \int_{\Omega} p_0(t) (\nabla \cdot \bs{\varphi}) \Big]   + \nu \big(\nabla \bs{v}(t), \nabla \bs{\varphi} \big)_\Omega   +\big( (\bs{v}(t)\! \cdot\! \nabla)\bs{v}(t), \bs{\varphi} \big)_\Omega\! \\
&  \displaystyle  + \dfrac{1}{2} \int_{\Gamma_{O}} [\bs{v}(t)  \cdot \bs{n}]^{-}(\bs{w}(t)\cdot \bs{\varphi})  - \int_{\Omega} \left(p_1(t) + p_2(t) +p_3(t)\right)  (\nabla \cdot \bs{\varphi}) 
\\&=\big\langle \bs{f}(t) \,,\,\bs{\varphi} \rangle \, + \langle\sigma_{*}(t)\bs{n}\,,\, \bs{\varphi} \rangle_{H^{-1/2}(\Gamma_O), \, H^{1/2}_{00}(\Gamma_O)} \, , 
\end{aligned} \\
\displaystyle	 \big(\bs{v} (0),\bs{\varphi}\big)_{\LL^2} =  \big(\bs{v}_0,\bs{\varphi}\big)_{\LL^2}\, ,  
\end{cases} \quad \forall \bs{\varphi} \in \mathcal{U}_{*} \, ,
\end{equation}
so that, 
\begin{equation}\label{press0}
p:=\partial_tp_0+p_1+p_2+p_3
\end{equation}
being $\partial_t p_0$ the time-derivative of $p_0$ in distributional sense, and
\begin{equation}\label{press1}
\begin{split}
&p_0\in L^{\infty}(0,T;L^2(\Omega))\qquad p_1\in L^{2}(0,T;L^2(\Omega))\\
&p_2\in L^{4/3}(0,T;L^2(\Omega))\qquad p_3\in L^2(0,T;L^2(\Omega)).
\end{split}
\end{equation}
We summarize our findings about the pressure in the following theorem.
\begin{theorem}\label{press}
If $\bs{v}$ is a weak solution to \eqref{ns}, then there exists an associated pressure $p$ (as a distribution in $Q_T$) as in \eqref{press0}, where $p_0, p_1, p_2, p_3$ have the regularity in \eqref{press1}.
\end{theorem}

\subsection{Proof of Theorem \ref{strong}}\label{proof_strong}
In order to prove the theorem we state some preliminary lemmas. They are used to estimate the convective term in \eqref{ns}.
\begin{lemma}\label{lemma:gagliardo}
	Let $\Omega\subset \R^3$ be an admissible domain. For all $\bs{w} \in D(\A)$ and $\bs{\varphi}\in \LL^2$ the following inequality holds:
	\begin{equation}\label{trilinear}
		\left|\left((\bs{w}\cdot \nabla)\bs{w},\bs{\varphi}\right)_{\Omega}\right|\le C
			\|\nabla\bs{w}\|_{L^{2}(\Omega)} \left\| \mathcal{A}\bs{w} \right\|_{\mathcal{L}^2}\|\bs{\varphi}\|_{\mathcal{L}^2} \, ,
	\end{equation}
	for some $C:=C(\Omega)>0$.
\end{lemma}
\noindent
\begin{proof}
	By the H{\"o}lder inequality we have
	\begin{equation}\label{bd}
		|\left((\bs{w}\cdot \nabla)\bs{w},\bs{\varphi}\right)_{\Omega}|\leq \|(\bs{w}\cdot \nabla)\bs{w}\|_{L^2(\Omega)}\|\bs{\varphi}\|_{L^2(\Omega)}\leq \|\bs{w}\|_{L^\ell(\Omega)}\|\nabla \bs{w}\|_{L^q(\Omega)}\|\bs{\varphi}\|_{L^2(\Omega)},
	\end{equation}
	for some $\ell,q$ that satisfies 
	\begin{equation}
		\frac{1}{\ell}+\frac{1}{q}=\frac{1}{2}.
		\label{rellq}
	\end{equation}
	
We have the embedding $H^{3/2}(\Omega) \hookrightarrow W^{1,3}(\Omega) \hookrightarrow L^{p}(\Omega)$, for all $p \geq 1$ and \eqref{bd}-\eqref{rellq} hold
	for $\ell \geq 6$ and $2 < q \leq 3$.
	If we use the classical Gagliardo-Nirenberg inequality to bound the $L^\ell$-norm of $ \bs{w}$ in \eqref{bd}, we obtain
	$$
	\|\bs{w}\|_{L^\ell(\Omega)}\leq C \|\bs{w}\|^{\frac 3\ell - \frac 12}_{L^2(\Omega)} \|\nabla\bs{w}\|^{\frac 32 - \frac 3\ell}_{L^2(\Omega)} ,
	\quad \forall \ell \in ( 2,6] \, ,
	$$
	which, taking into account the restriction on the values of $\ell$, will reduce to $\ell=6$ and
	\begin{equation}
		\|\bs{w}\|_{L^6(\Omega)}\leq C  \|\nabla\bs{w}\|_{L^2(\Omega)}\qquad \forall \bs{w}\in  \VV_*.
		\label{estwl3}
	\end{equation}
	Therefore, from \eqref{rellq} we have $q=3$ and
	\begin{equation}
		\|\nabla \bs{w}\|_{L^3(\Omega)}\leq C \|\bs{w}\|_{H^{3/2}(\Omega)} \qquad \forall \bs{w}\in  H^{3/2}(\Omega)^3\cap {\mathcal V}_*.
		\label{estgwq3}
	\end{equation}
Hence we estimate the nonlinear term as
	\begin{equation}\label{stima3d}
	\begin{split}
		& |\left((\bs{w}\cdot \nabla)\bs{w},\bs{\varphi}\right)_{\Omega}|
		\leq C   \|\bs{w}\|_{L^6(\Omega)} \|\nabla\bs{w}\|_{L^3(\Omega)}  \|\bs{\varphi}\|_{L^2(\Omega)} \leq C\|\nabla\bs{w}\|_{L^2(\Omega)} \|\bs{w}\|_{H^{3/2}(\Omega)}\|\bs{\varphi}\|_{\LL^2},
	\end{split}
\end{equation}
for all $\bs{w}\in  H^{3/2}(\Omega)^3\cap {\mathcal V}_*$ and $\bs{\varphi}\in \LL^2$.
Using Theorem \ref{stokesoperatortheo}, we obtain the thesis.
\end{proof}
\medskip

The proof of Theorem \ref{strong} is also based on the Galerkin method, see for instance \cite{Gal2000a,RRS2016,Tem1977b} where the problem under classical boundary conditions is studied. 

Now we consider $\bs{W}_{\! \! *}$ given by Lemma \ref{reference} with $k=2$.
 Since $\bs{v}_0\in \VV$ and $\bs{W}_{\! \! *}\in H^1(0,T;H^2(\Omega)^3\cap\VV)$, we have $\bs{w}_0=\bs{v}_0-\bs{W}_{\! \! *}(\cdot,0)\in \VV_*$. We consider the equality \eqref{ode0} with the test function $\bs{\varphi}=\lambda_k\bs{u}_k$, getting for $k=1,\dots,n$
	\begin{equation}\label{weak2}
\begin{aligned}
&\big(\partial_t\bs{w}^n(t),\lambda_k\bs{u}_k\big)_{\LL^2}+\nu \big(\nabla \bs{w}^n(t), \nabla \lambda_k\bs{u}_k\big)_\Omega +\big( (\bs{w}^n(t) \cdot \nabla)\bs{w}^n(t), \lambda_k\bs{u}_k\big)_\Omega \\&+ \big((\bs{w}^n(t) \cdot \nabla)\bs{W}_{\! \! *}(t),  \lambda_k\bs{u}_k\big)_\Omega + \big((\bs{W}_{\! \! *}(t) \cdot \nabla)\bs{w}^n(t), \lambda_k\bs{u}_k\big)_\Omega=\\&-\dfrac{1}{2} \int_{\Gamma_{O}} [(\bs{w}^n(t) + \bs{W}_{\! \! *}(t)) \cdot \bs{n}]^{-}(\bs{w}^n(t) \cdot \lambda_k\bs{u}_k) 
-\big(\partial_t\bs{W}_{\! \! *}(t)\,,\,\lambda_k\bs{u}_k\big)_{\LL^2}\\&+ \big(\nu \Delta\bs{W}_{\! \! *}(t) -(\bs{W}_{\! \! *}(t)\cdot \nabla)\bs{W}_{\! \! *}(t)+ \bs{f}(t) \,,\,\lambda_k\bs{u}_k\big)_\Omega.
\end{aligned}
\end{equation}
We observe that, integrating by parts,
	\begin{equation}\label{wt}
	\begin{split}
	\big(\partial_t\bs{w}^n,\lambda_k\bs{u}_k\big)_{\LL^2}&=\big(\partial_t\bs{w}^n,-\Delta \bs{u}_k+\nabla \pi_k\big)_\Omega+\left(\partial_t\bs{w}^n,\frac{\partial\bs{u}_k}{\partial\bs{n}}- \pi_k\bs{n}\right)_{\Gamma_O}=\big(\nabla (\partial_t\bs{w}^n),\nabla \bs{u}_k\big)_\Omega.
	\end{split}
\end{equation}
We introduce
	\begin{equation}\label{approxp}
q^n(\bs{x},t):=\sum_{k=1}^n c^n_{k}(t)\,\pi_k(\bs{x}) \qquad \forall (\bs{x},t) \in Q_{T} \, ,
\end{equation}
getting
\begin{equation}\label{q}
	\int_\Omega \nabla q^n(t)\lambda_k\bs{u}_k=\int_{\Gamma_O} q^n(t)\lambda_k(\bs{u}_k\cdot\bs{n})-\int_\Omega\lambda_kq^n(t)(\nabla\cdot \bs{u}_k)=	\int_{\Gamma_O} q^n(t)\lambda_k(\bs{u}_k\cdot\bs{n})
\end{equation}
and 
	\begin{equation}
\begin{split}
\big(\nabla \bs{w}^n(t), \nabla \lambda_k\bs{u}_k\big)_\Omega&=\int_{\Gamma_O}\lambda_k\bs{u}_k\cdot	\frac{\partial\bs{w}^n(t)}{\partial\bs{n}}- \big(\Delta \bs{w}^n(t), \lambda_k\bs{u}_k\big)_\Omega\\
&=\int_{\Gamma_O}\lambda_k\bs{u}_k\cdot	\Big(\frac{\partial\bs{w}^n(t)}{\partial\bs{n}}-q^n(t)\bs{n}\Big)- \big(\Delta \bs{w}^n(t)-\nabla q^n(t), \lambda_k\bs{u}_k\big)_\Omega\\
&=\int_{\Gamma_O}\left(\frac{\partial\bs{u}_k}{\partial\bs{n}}- \pi_k\bs{n}\right)\!\cdot\!	\Big(\frac{\partial\bs{w}^n(t)}{\partial\bs{n}}-q^n(t)\bs{n}\Big)+ \!\big(\!-\Delta \bs{w}^n(t)+\nabla q^n(t),-\Delta \bs{u}_k+\nabla \pi_k\big)_\Omega.
\end{split}
\end{equation}
Hence, multiplying \eqref{weak2} by $c^n_k(t)$,  summing for $k$ from 1 to $n$ and recalling \eqref{domainareg} and \eqref{proj}, we obtain
	\begin{equation}\label{eq005}
	\begin{split}
	\dfrac{1}{2}	&\frac{d}{dt}\|\nabla\bs{w}^n(t)\|^2_{L^2(\Omega)}+\nu\|\mathcal{A}\bs{w}^n(t)\|^2_{\LL^2}= -\big((\bs{w}^n(t)\cdot\nabla) \bs{w}^n(t), \A \bs{w}^n(t)\big)_\Omega \\&- \big((\bs{w}^n(t) \cdot \nabla)\bs{W}_{\! \! *}(t),  \A \bs{w}^n(t)\big)_\Omega - \big((\bs{W}_{\! \! *}(t) \cdot \nabla)\bs{w}^n(t), \A \bs{w}^n(t) \big)_\Omega\\&
	-\dfrac{1}{2} \int_{\Gamma_{O}} [(\bs{w}^n(t) + \bs{W}_{\! \! *}(t)) \cdot \bs{n}]^{-}\bigg[\bs{w}^n(t) \cdot \bigg(\frac{\partial \bs{w}^n(t)}{\partial\bs{n}}-q^n(t)\bs{n}\bigg)\bigg]-\big(\partial_t\bs{W}_{\! \! *}(t)\,,\,\A \bs{w}^n(t)\big)_{\LL^2}
	\\
	&+\big(\nu \Delta\bs{W}_{\! \! *}(t)-(\bs{W}_{\! \! *}(t)\cdot \nabla)\bs{W}_{\! \! *}(t)+\bs{f}(t)\,,\,\A \bs{w}^n(t)\big)_\Omega,
	\end{split}
	\end{equation}
where $\A \bs{u}$ is the generalized Stokes operator introduced in Subsection \ref{substokes}.

Now we need bounds for the terms on the right-hand side of \eqref{eq005}.
Using the (spatial) regularity of $\bs{W}_{\! \! *}\in H^2(\Omega)^3\subset \mathcal{C}(\overline\Omega)^3$, see Lemma \ref{lemma0}, H{\"o}lder, Sobolev and Young inequalities ($\epsilon>0$), we obtain
\begin{equation}\label{eq0020}
	\begin{split}
		\big|\big((\bs{w}^n \cdot \nabla)\bs{W}_{\! \! *}\,,\,  \A \bs{w}^n\big)_\Omega\big|&\leq  \|(\bs{w}^n \cdot \nabla)\bs{W}_{\! \! *}\|_{L^2(\Omega)}\|\A \bs{w}^n\|_{L^2(\Omega)}\\&\leq \|\bs{w}^n\|_{L^4(\Omega)}\|\nabla\bs{W}_{\! \! *}\|_{L^4(\Omega)}\|\A \bs{w}^n\|_{L^2(\Omega)}\\&\hspace{0mm}\leq \frac{\epsilon}{2}\|\A \bs{w}^n\|^2_{\mathcal{L}^2}+ \dfrac{C \|\nabla\bs{W}_{\! \! *}\|_{L^4(\Omega)}^2}{\epsilon}\|\nabla\bs{w}^n\|_{L^2(\Omega)}^2 ;\\
		|\big((\bs{W}_{\! \! *} \cdot \nabla)\bs{w}^n\,,\, \A \bs{w}^n \big)_\Omega|&\leq \|\bs{W}_{\! \! *}\|_{L^\infty(\Omega)}\|\nabla \bs{w}^n\|_{L^2(\Omega)}\|\A \bs{w}^n\|_{L^2(\Omega)}\\&\leq \frac{\epsilon}{2}\|\A \bs{w}^n\|^2_{\mathcal{L}^2}+ \dfrac{\|\bs{W}_{\! \! *}\|_{L^\infty(\Omega)}^2}{2\epsilon}\|\nabla\bs{w}^n\|_{L^2(\Omega)}^2 ;\\
		|\big(\partial_t\bs{W}_{\! \! *}\,,\,\A \bs{w}^n\big)_{\LL^2}|&\leq\frac{\epsilon}{2}\|\A \bs{w}^n\|^2_{\mathcal{L}^2}+\frac{1}{2\epsilon}\|\partial_t\bs{W}_{\! \! *}\|^2_{\LL^2};\\
		\big|\big(\nu \Delta\bs{W}_{\! \! *}+\bs{f}\,,\,\A \bs{w}^n\big)_\Omega\big|&\leq \big(\nu\|\Delta\bs{W}_{\! \! *}\|_{L^2(\Omega)}+\|\bs{f}\|_{L^2(\Omega)}\big)\|\A\bs{w}^n\|_{L^2(\Omega)}\\&\leq \frac{\epsilon}{2}\|\A \bs{w}^n\|^2_{\mathcal{L}^2}+\frac{1}{\epsilon}\big(\nu^2\|\Delta\bs{W}_{\! \! *}\|^2_{L^2(\Omega)}+\|\bs{f}\|^2_{L^2(\Omega)}\big);\\
		\big|\big((\bs{W}_{\! \! *}\cdot \nabla)\bs{W}_{\! \! *}\,,\,\A \bs{w}^n\big)_\Omega\big|&\leq \|(\bs{W}_{\! \! *}\cdot \nabla)\bs{W}_{\! \! *}\|_{L^2(\Omega)}\|\A \bs{w}^n\|_{L^2(\Omega)}\\& \leq \|\bs{W}_{\! \! *}\|_{L^\infty(\Omega)}\|\nabla\bs{W}_{\! \! *}\|_{L^2(\Omega)}\|\A \bs{w}^n\|_{L^2(\Omega)}\\&\leq \frac{\epsilon}{2}\|\A\bs{w}^n\|_{\mathcal{L}^2}^2+\frac{\|\bs{W}_{\! \! *}\|_{L^\infty(\Omega)}^2\|\nabla\bs{W}_{\! \! *}\|_{L^2(\Omega)}^2}{2\epsilon}.
	\end{split}
\end{equation}
Similarly, the boundary nonlinear term
\begin{equation}\label{bd0}
	\begin{split}
	\bigg| \int_{\Gamma_{O}} [(\bs{w}^n+ \bs{W}_{\! \! *}) \cdot \bs{n}]^{-}\bigg[\bs{w}^n \cdot \bigg(\frac{\partial \bs{w}^n}{\partial\bs{n}}-q^n\bs{n}\bigg)\bigg]\bigg|&\leq\|\bs{w}^n+ \bs{W}_{\! \! *}\|_{L^4(\Gamma_O)}\|\bs{w}^n\|_{L^4(\Gamma_O)}\|\A \bs{w}^n\|_{L^2(\Gamma_O)}\\
	&\hspace{-40mm}\leq C\|\nabla\bs{w}^n+ \nabla\bs{W}_{\! \! *}\|_{L^2(\Omega)}\|\nabla\bs{w}^n\|_{L^2(\Omega)}\|\A \bs{w}^n\|_{L^2(\Gamma_O)}\\
	&\hspace{-40mm}\leq \frac{\epsilon}{2}\|\A \bs{w}^n\|^2_{\LL^2}+\frac{C}{\epsilon}\bigg(\|\nabla\bs{w}^n\|_{L^2(\Omega)}^4+\|\nabla\bs{w}^n\|_{L^2(\Omega)}^2\|\nabla\bs{W}^*\|_{L^2(\Omega)}^2\bigg).
	\end{split}
\end{equation}
Here and in what follows, $C:=C(\Omega)> 0$ will always denote a generic constant, that may change from line to line.
To bound $\big((\bs{w}^n(t)\cdot\nabla) \bs{w}^n(t), \A \bs{w}^n(t)\big)_\Omega$ in \eqref{eq005} we apply Lemma \ref{lemma:gagliardo} with $\bs{\varphi}=\A \bs{w}^n$. Therefore, due to the previous estimates, the equation \eqref{eq005} becomes
	\begin{equation}\label{eq0063d}
\begin{split}
\dfrac{1}{2}&\frac{d}{dt}\|\nabla\bs{w}^n(t)\|^2_{L^2(\Omega)} + \nu\|\A \bs{w}^n(t)\|^2_{\mathcal{L}^2} \leq 
\\& \left(\frac{11}{4}\epsilon+C\|\nabla\bs{w}^n(t)\|_{L^{2}(\Omega)}\right)\|\A \bs{w}^n(t)\|_{\LL^2}^{2} + \dfrac{C}{\epsilon}\bigg[\|\nabla \bs{w}^n(t)\|_{L^2(\Omega)}^4+\\& \left(\|\nabla\bs{W}_{\! \! *}(t)\|_{L^4(\Omega)}^2+\|\nabla\bs{W}_{\! \! *}(t)\|_{L^2(\Omega)}^2+\|\bs{W}_{\! \! *}(t)\|_{L^\infty(\Omega)}^2\right)\|\nabla\bs{w}^n(t)\|_{L^2(\Omega)}^2+\\&\|\partial_t\bs{W}_{\! \! *}(t)\|_{\mathcal{L}^2}^2+\|\bs{W}_{\! \! *}(t)\|_{L^\infty(\Omega)}^2\|\nabla\bs{W}_{\! \! *}(t)\|_{L^2(\Omega)}^2+\nu^2\|\Delta\bs{W}_{\! \! *}(t)\|^2_{L^2(\Omega)}+\|\bs{f} (t)\|_{L^2(\Omega)}^{2}\bigg].
\end{split}
\end{equation}
Applying the Young inequality, we obtain
\begin{equation}\label{eq63d}
 \frac{d}{dt}\|\nabla\bs{w}^n(t)\|^2_{L^2(\Omega)}+\left(2\nu-\frac{11}{2}\epsilon-C\|\nabla\bs{w}^n(t)\|_{L^{2}(\Omega)}\right)\|\A \bs{w}^n(t)\|^2_{{\mathcal L}^2} \leq \frac{C}{\epsilon}\Big(\|\nabla \bs{w}^n(t)\|_{L^2(\Omega)}^4+W_\nu(t)+F(t)\Big),
\end{equation}
where
$$
W_\nu(t) :=\|\bs{W}_{\! \! *}(t)\|_{H^2(\Omega)}^4+\nu^2\|\bs{W}_{\! \! *}(t)\|_{H^2(\Omega)}^2 +\|\partial_t\bs{W}_{\! \! *}(t)\|_{{\mathcal L}^2}^2, \qquad F(t) := \|\bs{f} (t)\|_{L^2(\Omega)}^{2}.
$$
We take $\epsilon=\frac{\nu}{11}$ in the previous inequality and we study at first
\begin{equation}\label{eq4mod}
	\begin{split}
		& \frac{d}{dt}\|\nabla\bs{w}^n(t)\|^2_{L^2(\Omega)}+\left(\frac 32 \nu - C\|\nabla\bs{w}^n(t)\|_{L^{2}(\Omega)}\right) \|\A \bs{w}^n(t)\|^2_{{\mathcal L}^2}\\&\quad \leq \frac{C}{\nu}\left(\|\nabla \bs{w}^n(t)\|_{L^2(\Omega)}^4 
		+ W_\nu(t) + F(t) \right).
	\end{split}
\end{equation}
Moreover, we assume
\begin{equation}
\|\nabla\bs{w}^n(0)\|_{L^{2}(\Omega)}\leq	\|\nabla\bs{w}_0\|_{L^{2}(\Omega)} < \dfrac{\nu}{C}, \quad \forall n \in {\mathbb N};
	\label{ass1}
\end{equation}
hence, by continuity of $[0,T] \ni t \mapsto \|\nabla\bs{w}^n(t)\|_{L^{2}(\Omega)} $, for each $n$ there exists $T_n >0$ such that 
\begin{equation}
	\|\nabla \bs{w}^n(t)\|_{L^2(\Omega)} \leq \frac{\nu}{C}, \quad \forall t \in  [0,T_n].
	\label{tn1}
\end{equation} 
The next step it is to find conditions on the data so that $T_n = T$ in \eqref{tn1}. A consequence of \eqref{tn1} is 
$$
\left(\frac 32 \nu - C\|\nabla\bs{w}^n(t)\|_{L^{2}(\Omega)}\right) \|\A \bs{w}^n(t)\|^2_{{\mathcal L}^2} \geq \frac \nu 2 \|\A \bs{w}^n(t)\|^2_{{\mathcal L}^2}, \quad \forall t \in  [0,T_n],
$$
and thus,  we can drop the term $\left(\frac 32 \nu - C\|\nabla\bs{w}^n(t)\|_{L^{2}(\Omega)}\right) \|\A \bs{w}^n(t)\|^2_{{\mathcal L}^2}$ in \eqref{eq4mod}, studying the inequality
\begin{equation}
\frac{d}{dt} \|\nabla\bs{w}^n(t)\|^2_{L^2(\Omega)} 
\leq \frac{C}{\nu} \Big(\|\nabla \bs{w}^n(t)\|_{L^2(\Omega)}^4+W_\nu(t) + F(t)\Big), \quad \forall t \in [0,T_n].
	\label{ineq1}
\end{equation} 
From \eqref{ineq1}, the previous estimates and Gronwall inequality, we get
\begin{equation}\label{0.13}
\|\nabla\bs{w}^n(t)\|^2_{L^2(\Omega)} \leq g_n(t), \quad \forall t\in [0,T_n],
\end{equation}
where $g_n: [0,T] \to [0,\infty]$ is given by (the increasing function)
\begin{equation}\label{eqteste}
	\begin{aligned}
		&g_n(t):= \exp\left( \frac{C}{\nu} \int_0^t \|\nabla\bs{w}^n(s)\|^2_{L^2(\Omega)} ds \right)\bigg[ \|\nabla\bs{w}^n(0)\|^2_{L^2(\Omega)} 
		\\ & +  \frac{C}{\nu} \int_0^t \left[W_\nu(s) + F(s)\right] \exp\left( - \frac{C}{\nu} \int_0^s \|\nabla\bs{w}^n(\tau)\|^2_{L^2(\Omega)} d\tau\right)ds\bigg]; 
	\end{aligned}
\end{equation}

If $\|\nabla\bs{w}^n(s)\|_{L^2(\Omega)}$ is identically zero on $[0,T]$ there is nothing to prove. Otherwise, recalling the uniform estimate \eqref{eq0044} on weak solutions and \eqref{ass1} we find
\begin{equation}
	g_n(t) < \exp\!\left(\frac{C}{\nu^2} M\right)
	\left[
	\|\nabla \bs{w}_0\|_{L^2(\Omega)}^2 + \frac{C}{\nu}\int_0^T (W_\nu(s)+F(s))\,ds
	\right]
	=: G,
	\quad \forall t\in[0,T], \ \forall n\in\mathbb{N}.
	\label{0.14}
\end{equation}
In view of \eqref{tn1}, \eqref{0.13} and the uniform bound \eqref{0.14} on the sequence $\{g_n\}$ in $[0,T]$, we require the data $\nu, \bs{w}_0, \bs{W}^*, \bs{f}$ so that $G \le \frac{\nu^2}{C^2}$.  
Note that the restriction
\begin{equation}
	\exp\!\left(\frac{C M(\Omega,\nu,T,\bs{w}_0,\bs{W}^*,\bs{f})}{\nu^2}\right)
	\left[
	\|\nabla \bs{w}_0\|_{L^2(\Omega)}^2 + \frac{C}{\nu}\int_0^T (W_\nu(s)+F(s))\,ds
	\right]
	\le \frac{\nu^2}{C^2},
	\label{0.15}
\end{equation}
equivalently to \eqref{datopiccolo3d},
implies that our initial assumption \eqref{ass1} holds true. 

Our aim now is to prove that $T_n=T$  in \eqref{tn1}. For contradiction, suppose there exists $n_0\in\mathbb{N}$ such that $T_{n_0}<T$ and
\[
\|\nabla \bs{w}^{n_0}(T_{n_0})\|_{L^2(\Omega)}^2
= \frac{\nu^2}{C^2}.
\]
Since $g_{n_0}(t)$ is increasing and satisfies \eqref{0.13}-\eqref{0.14}, we have
\[
\|\nabla \bs{w}^{n_0}(T_{n_0})\|_{L^2(\Omega)}^2\leq g_{n_0}(T_{n_0}) < G \le \frac{\nu^2}{C^2},
\]
giving the contradiction
\[
0\leq g_{n_0}(T_{n_0})
- \|\nabla \bs{w}^{n_0}(T_{n_0})\|_{L^2(\Omega)}^2
<\frac{\nu^2}{C^2}
- \|\nabla \bs{w}^{n_0}(T_{n_0})\|_{L^2(\Omega)}^2=0.
\]
Therefore, if \eqref{0.15} holds true we can take $T_n=T$, for all $n\in\mathbb{N}$, and we infer the uniform bound
\begin{equation}\label{stima2}
	\|\nabla \bs{w}^n(t)\|_{L^2(\Omega)}^2
	\le G \le \frac{\nu^2}{C^2},
	\qquad \forall t\in[0,T].
\end{equation}
Assuming \eqref{0.15}, from \eqref{stima2} the inequality \eqref{eq4mod} becomes
\begin{equation}
	\begin{split}
	\dfrac{d}{dt}\|\nabla \bs{w}^n(t)\|_{L^2(\Omega)}^2
+ \dfrac{\nu}{2}\|\A \bs{w}^n(t)\|_{\LL^2}^2
&\le
\dfrac{C}{\nu}\Big(
\|\nabla \bs{w}^n(t)\|_{L^2(\Omega)}^4
+W_\nu(t)+F(t)
\Big)\\&\le
\dfrac{C}{\nu}\left(
G^2
+W_\nu(t)+F(t)
\right),
\qquad \forall t\in[0,T].
	\end{split}
\end{equation}
We integrate  from 0 e $t\in[0,T]$ and we obtain a uniform bound on the norm of $\A \bs{w}^n$
\begin{equation}\label{stima3}
	\int_0^t\|\A \bs{w}^n(\tau)\|_{\LL^2}^2d\tau
	\le \dfrac{C}{\nu^2}\left(\nu\|\nabla\bs{w}_0\|^2_{L^2(\Omega)}+
	G^2t
	+\int_0^t[W_\nu(\tau)+F(\tau)]d\tau\right):=H(t),
	\quad \forall t\in[0,T].
\end{equation}
 In order to estimate the norm of $\partial_t\bs{w}^n$ we multiply  \eqref{ode}$_1$ by $\frac{d}{dt}c^n_k(t)$ and we sum over $k$, getting
 \begin{equation}\label{eq2_2d}
 	\begin{aligned}
 		&\|\partial_t\bs{w}^n(t)\|^2_{\LL^2}=-\nu \big(\nabla \bs{w}^n(t), \nabla(\partial_t\bs{w}^n)(t)\big)_\Omega\! -\!\big( (\bs{w}^n(t) \cdot \nabla)\bs{w}^n(t), \partial_t\bs{w}^n(t)\big)_\Omega\! -\! \big((\bs{w}^n(t) \cdot \nabla)\bs{W}_{\! \! *}(t),  \partial_t\bs{w}^n(t)\big)_\Omega \\&- \big((\bs{W}_{\! \! *}(t) \cdot \nabla)\bs{w}^n(t), \partial_t\bs{w}^n(t)\big)_\Omega 
 		-\dfrac{1}{2} \int_{\Gamma_{O}} [(\bs{w}^n(t)+\bs{W}_{\! \! *}(t)) \cdot \bs{n}]^{-}(\bs{w}^n(t)\cdot \partial_t\bs{w}^n(t))\\&-\big(\partial_t\bs{W}_{\! \! *}(t)\,,\,\partial_t\bs{w}^n(t)\big)_{\LL^2}+\big(\nu \Delta\bs{W}_{\! \! *}(t)-(\bs{W}_{\! \! *}(t)\cdot \nabla)\bs{W}_{\! \! *}(t)+\bs{f}(t)\,,\,\partial_t\bs{w}^n(t)\big)_\Omega.
 	\end{aligned}
 \end{equation}
 Applying Lemma \ref{lemma:gagliardo} with $\bs{\varphi}=\partial_t\bs{w}^n$ we find
\begin{equation}\label{stima2d3}
	\begin{split}
		|\big((\bs{w}^n\cdot\nabla) \bs{w}^n, \partial_t\bs{w}^n\big)_\Omega|&\leq C\|\nabla\bs{w}^n\|_{L^{2}(\Omega)}\|\A \bs{w}^n\|_{\LL^2}\|\partial_t\bs{w}^n\|_{\mathcal{L}^2}\leq C\nu\|\A \bs{w}^n\|_{\LL^2}\|\partial_t\bs{w}^n\|_{\mathcal{L}^2},
	\end{split}
\end{equation}
 while, recalling the definition of $q^n$ in \eqref{approxp}, with an integration by parts similar to \eqref{wt}, we observe 	
 \begin{equation}\label{stima32d}
 	\begin{split}
 		|\big(\nabla \bs{w}^n, \nabla(\partial_t\bs{w}^n)\big)_\Omega|&=\bigg|\int_\Omega (-\Delta\bs{w}^n+\nabla q^n)\cdot \partial_t\bs{w}^n+\int_{\Gamma_O}\left(\frac{\partial\bs{w}^n}{\partial\bs{n}}-q^n\bs{n}\right)\cdot \partial_t\bs{w}^n\bigg|\\
 		&\leq\|\partial_t\bs{w}^n\|_{\LL^2}\|\A\bs{w}^n\|_{\LL^2}.
 	\end{split}
 \end{equation}
 Therefore, through H\"{o}lder, Young inequalities ($\epsilon>0$) and the bounds \eqref{eq0020}, \eqref{bd0}, \eqref{stima2d3} and \eqref{stima32d}, the inequality \eqref{eq2_2d} becomes
 \begin{equation}\label{eq3_2d}
 	\begin{split}
 		\|\partial_t\bs{w}^n\|_{\LL^2}\leq&\nu \|\A \bs{w}^n\|_{\LL^2}+C\nu\|\A \bs{w}^n\|_{\LL^2} +\|(\bs{w}^n \cdot \nabla)\bs{W}_{\! \! *}\|_{L^2(\Omega)}\\& + \|(\bs{W}_{\! \! *} \cdot \nabla)\bs{w}^n\|_{L^2(\Omega)}
 		+C\|\nabla\bs{w}^n+ \nabla\bs{W}_{\! \! *}\|_{L^2(\Omega)}\|\nabla\bs{w}^n\|_{L^2(\Omega)}\\&+ \|\partial_t\bs{W}_{*}\|_{\mathcal{L}^2}+\|(\bs{W}_{\! \! *}\cdot \nabla)\bs{W}_{\! \! *}\|_{L^2(\Omega)} +\nu\|\Delta\bs{W}_{\! \! *}\|_{L^2(\Omega)}+\| \bs{f}\|_{L^2(\Omega)} \\[2mm]
 		\leq&C\big( \nu\|\A \bs{w}^n\|_{\LL^2}+\|\nabla \bs{w}^n\|^{2}_{L^2(\Omega)}+\|\bs{W}_{\! \! *}\|^2_{H^2(\Omega)}+\nu\|\bs{W}_{\! \! *}\|_{H^2(\Omega)}+\|\partial_t\bs{W}_{\! \! *}\|_{\mathcal{L}^2}+\|\bs{f}\|_{L^2(\Omega)}\big).
 	\end{split}
 \end{equation}
After squaring, integrating from $0$ to $t\in [0,T]$, we obtain
\begin{equation}\label{eq4_3d}
	\begin{aligned}
		\int_0^t\|\partial_t\bs{w}^n(\tau)\|^2_{\mathcal{L}^2}d\tau\leq& C \int_0^t\big(\nu^2\|\A \bs{w}^n(\tau)\|^2_{\mathcal{L}^2}\!+\!\|\nabla \bs{w}^n(\tau)\|^{4}_{L^2(\Omega)}\!+\\&\hspace{10mm}\|\bs{W}_{\! \! *}(\tau)\|^4_{H^2(\Omega)}+\nu^2\|\bs{W}_{\! \! *}(\tau)\|^2_{H^2(\Omega)}+\|\partial_t\bs{W}_{\! \! *}(\tau)\|^2_{\mathcal{L}^2}\!+\!\|\bs{f}(\tau)\|^2_{L^2(\Omega)}\big)d\tau\\
		\leq& C\bigg[\nu^2H(t)+G^2 t+\\&\hspace{3mm}\int_0^t(\|\bs{W}_{\! \! *}(\tau)\|^4_{H^2(\Omega)}+\nu^2\|\bs{W}_{\! \! *}(\tau)\|^2_{H^2(\Omega)}+\|\partial_t\bs{W}_{\! \! *}(\tau)\|^2_{\mathcal{L}^2}+\|\bs{f}(\tau)\|^2_{L^2(\Omega)})d\tau\bigg], 
	\end{aligned}
\end{equation}
\normalsize for all $t\in[0,T]$ in which we also used the bounds  \eqref{stima2} and \eqref{stima3}.
From~\eqref{stima2}-\eqref{stima3}-\eqref{eq4_3d} we infer respectively the boundedness of $\bs{w}^n$ in $L^\infty(0,T;\VV_*)$ and of $\partial_t\bs{w}^n$, $\A \bs{w}^n$ in $L^2(0,T;\LL^2)$. Hence, up to a subsequence, we infer weak-star and weak convergence in the respective spaces to $\bs{w}$ and $\partial_t\bs{w}$, $\A \bs{w}$. 
Moreover, from \eqref{estimatea} and \eqref{stima3} we get
\begin{equation}\label{stima3'}
\int_0^t\|\bs{w}^n(\tau)\|^2_{H^{3/2}(\Omega)}d\tau\leq C\int_0^t\|\A \bs{w}^n(\tau)\|_{\LL^2}^2d\tau
\le C H(t),
\quad \forall t\in[0,T],
\end{equation}
so that, up to a subsequence, we have weak convergence of $\bs{w}^n$ in $L^2(0,T;H^{3/2}(\Omega)^{3})$  to $\bs{w}$.

 Since
$$
\big(\nabla \bs{w}, \nabla \bs{\varphi}\big)_\Omega =  \int_{\Gamma_O}\frac{\partial \bs{w}}{\partial \bs{n}} \cdot \bs{\varphi}  - 
\int_{\Omega} \Delta \bs{w} \cdot \bs{\varphi},
$$
 from the Definition \ref{weaksolution} of weak solution, we obtain
\begin{equation}\label{weakp}
	\big(\partial_t\bs{v},\bs{\varphi}\big)_{\LL^2} + \int_{\Omega} \left[ (\bs{v} \cdot \nabla)\bs{v}  - \nu  \Delta \bs{v} - \bs{f} \right]  \cdot \bs{\varphi} + \left( \nu \frac{\partial \bs{v}}{\partial \bs{n}} + \dfrac{1}{2} [\bs{v}  \cdot \bs{n}]^{-}\bs{w} - \sigma_{*} \bs{n} , \bs{\varphi} \right)_{\Gamma_{O}} = 0 , \text{ a.e. in } (0,T),
\end{equation}
for all $\bs{\varphi} \in \mathcal{V}_{*}$. In particular, we have
\begin{equation}\label{weakprest}
	\int_{\Omega} \left[ \partial_t\bs{v} + (\bs{v} \cdot \nabla)\bs{v}  - \nu  \Delta \bs{v} - \bs{f} \right]  \cdot \bs{\varphi}  = 0 
	, \text{ a.e. in } (0,T),
\end{equation}
for all $
\bs{\varphi} \in  \mathcal{C}^\infty_{0,\sigma}(\Omega)^3$, and by density, for all $\bs{\varphi} \in \mathcal{H}_{3} $, the closure of $ \mathcal{C}^\infty_{0,\sigma}(\Omega)^3$ in $L^3(\Omega)^3$. This means $\partial_t\bs{v} + (\bs{v} \cdot \nabla)\bs{v}  - \nu  \Delta \bs{v} - \bs{f}  \in \mathcal{H}_{3}^\perp \subset \mathcal{G}_{3/2}$. Therefore, there exists 
$\mathfrak{p} \in L^2(0,T;W^{1,3/2}(\Omega))$ such that
\begin{equation}
	\partial_t\bs{v} + (\bs{v} \cdot \nabla)\bs{v}  - \nu  \Delta \bs{v} - \bs{f}  = -\nabla \mathfrak{p}.
	\label{pres32}
\end{equation}
Here we used the Helmholtz-Weyl decomposition $L^q(\Omega)^3= \mathcal{H}_{q} \oplus \mathcal{G}_{q}$, for $q = 3/2,3$. These decompositions are valid in general Lipschitz domains for $3/2-\eps<q<3+\eps$, where $\eps>0$ depends on the domain, see \cite[Theorem 11.1]{fabesmitreanew}. Now we put \eqref{pres32} in \eqref{weakp}, and apply the Divergence Theorem to obtain
\begin{equation}\label{weakp2}
\int_{\Gamma_{O}}\left(\partial_t\bs{v}+\nu \frac{\partial \bs{v}}{\partial \bs{n}}-\mathfrak{p}  \bs{n} + \dfrac{1}{2} [\bs{v}  \cdot \bs{n}]^{-}\bs{w} - \sigma_{*} \bs{n}  \right)\cdot \bs{\varphi} = 0 , \text{ a.e. in } (0,T) \qquad \forall \bs{\varphi}\in\VV_*.
\end{equation}
The previous equality holds for all $\bs{\varphi}\in H^{1/2}_{00}(\Gamma_O)\cap L^2_{\bs{n},0}(\Gamma_O)$ and by density, for all $\bs{\varphi}\in L^2_{\bs{n},0}(\Gamma_O)$. Then, using the decomposition \eqref{decL1O}, we infer the existence of $c\in\R$ such that
\begin{equation}\label{weakp4}
\partial_t\bs{v}+\nu \frac{\partial \bs{v}}{\partial \bs{n}}-\mathfrak{p}  \bs{n} + \dfrac{1}{2} [\bs{v}  \cdot \bs{n}]^{-}\bs{w} - \sigma_{*} \bs{n} = c\bs{n}.
\end{equation}
Therefore, taking $p=\mathfrak{p}+c$ we recover the boundary condition \eqref{ns}$_4$.\hfill\qed

\newpage
\noindent
{\bf Acknowledgements.} The research of Alessio Falocchi is supported by the grant \textit{Dipartimento di Eccellenza 2023-2027},
issued by the Ministry of University and Research (Italy) and is a part of the INdAM-GNAMPA project entitled ``EDP e Applicazioni: Dinamica dei Fluidi e Teoria Spettrale'' (15/01/2026-31/12/2026). Ana L. Silvestre acknowledges the financial support of \textit{Funda\c{c}\~ao para a Ci\^encia e a Tecnologia}  (FCT, Portuguese Agency for Scientific Research), through the project UIDB/04621/2025 of CEMAT/IST-ID (https://doi.org/10.54499/UID/04621/2025). The research of Gianmarco Sperone is supported by the \textit{Chilean National Agency for Research and Development} (ANID) through the \textit{Fondecyt Iniciación} grant 11250322.
\par\smallskip
\noindent
{\bf Data availability statement.} Data sharing not applicable to this article as no datasets were generated or analyzed during the current study.
\par\smallskip
\noindent
{\bf Conflict of interest statement}.  The Authors declare that they have no conflict of interest.

\phantomsection
\addcontentsline{toc}{section}{References}
\bibliographystyle{abbrv}
\bibliography{references2}

\vspace{5mm}
\noindent
\hspace{0.1mm}
\begin{minipage}{140mm}
	\textbf{Alessio Falocchi}\\
	Dipartimento di Matematica\\
	Dipartimento di Eccellenza MUR 2023-2027\\
	Politecnico di Milano\\
	Piazza Leonardo da Vinci 32\\
	20133 Milan - Italy\\
	E-mail: alessio.falocchi@polimi.it
	\vspace{0.5cm}	
\end{minipage}
\newline
\vspace{0.5cm}
\noindent
\begin{minipage}{100mm}
	\textbf{Ana Leonor Silvestre}\\
	Centro de Matemática Computacional e Estocástica\\
	Instituto Superior Técnico, Universidade de Lisboa\\
	Avenida Rovisco Pais 1\\
	1049-001 Lisbon - Portugal\\
	and\\
	Departamento de Matemática do Instituto Superior Técnico\\
	Universidade de Lisboa\\
	Avenida Rovisco Pais 1\\
	1049-001 Lisbon - Portugal\\
	E-mail: ana.silvestre@math.tecnico.ulisboa.pt
\end{minipage}
\newline
\vspace{0.5cm}
\begin{minipage}{100mm}
	\textbf{Gianmarco Sperone}\\
	Facultad de Matemáticas\\
	Pontificia Universidad Católica de Chile\\
	Avenida Vicuña Mackenna 4860\\
	7820436 Santiago - Chile\\
	E-mail: gianmarco.sperone@uc.cl
\end{minipage}
\end{document}